\documentclass[10pt,reqno]{amsart}

\usepackage{a4wide}
\usepackage{amssymb}

\usepackage[utf8]{inputenc}
\usepackage{t1enc}

\usepackage{dsfont}
\usepackage{mathrsfs}
\usepackage[hidelinks]{hyperref}
\usepackage{upgreek}
\usepackage{mathabx}
\usepackage{graphicx} 

\newtheorem{proposition}{Proposition}[section]
\newtheorem{theorem}[proposition]{Theorem}
\newtheorem{corollary}[proposition]{Corollary}
\newtheorem{lemma}[proposition]{Lemma}
\theoremstyle{definition}
\newtheorem{definition}[proposition]{Definition}
\newtheorem{remark}[proposition]{Remark}
\newtheorem{questions}[proposition]{Questions}
\newtheorem{example}[proposition]{Example}
\newtheorem{examples}[proposition]{Examples}

\newcommand{\cst}{\ensuremath{\mathrm{C}^*}}
\newcommand{\comp}{\circ}
\newcommand{\dd}{\,\mathrm{d}}
\newcommand{\ee}{\mathrm{e}}
\newcommand{\eps}{\varepsilon}
\newcommand{\ph}{\varphi}
\newcommand{\hh}[1]{\widehat{#1}}
\newcommand{\I}{\mathds{1}}
\newcommand{\id}{\mathrm{id}}
\newcommand{\ii}{\mathrm{i}}

\newcommand{\Int}{\int\limits}
\newcommand{\is}[2]{{\left\langle{#1}\,\vline\,#2\right\rangle}}
\newcommand{\biggis}[2]{{\biggl\langle{#1}\,\biggr|\biggl.\,#2\biggr\rangle}}
\newcommand{\ket}[1]{{\left|#1\right\rangle}}
\newcommand{\bra}[1]{{\left\langle#1\right|}}
\newcommand{\tens}{\otimes}
\newcommand{\vtens}{\overline{\otimes}}
\newcommand{\btens}{\boxtimes}
\newcommand{\atens}{\otimes_{\text{\tiny{\rm{alg}}}}\!}

\newcommand{\CC}{\mathbb{C}}
\newcommand{\EE}{\mathbb{E}}
\newcommand{\HH}{\mathbb{H}}
\newcommand{\GG}{\mathbb{G}}
\newcommand{\KK}{\mathbb{K}}
\newcommand{\NN}{\mathbb{N}}
\newcommand{\QQ}{\mathbb{Q}}
\newcommand{\RR}{\mathbb{R}}
\newcommand{\TT}{\mathbb{T}}
\newcommand{\ZZ}{\mathbb{Z}}

\newcommand{\sH}{\mathsf{H}}
\newcommand{\sK}{\mathsf{K}}

\newcommand{\cC}{\mathscr{C}}
\newcommand{\cO}{\mathcal{O}}
\newcommand{\cQ}{\mathcal{Q}}

\newcommand{\bA}{{\boldsymbol{A}}}
\newcommand{\bh}{{\boldsymbol{h}}}

\DeclareMathOperator{\B}{B}
\DeclareMathOperator{\C}{C}
\DeclareMathOperator{\Dom}{Dom}
\DeclareMathOperator{\HS}{HS}
\DeclareMathOperator{\Lone}{\mathsf{L}^1}
\DeclareMathOperator{\Ltwo}{\mathsf{L}^2}
\DeclareMathOperator{\Linf}{\mathsf{L}^{\!\infty}}
\DeclareMathOperator{\cZ}{\mathscr{Z}}
\DeclareMathOperator{\SU}{SU}

\newcommand{\modOp}[1][\hspace{0.87pt}]{\nabla_{\!#1}}
\newcommand{\btau}[1][\hspace{0.65pt}]{\boldsymbol{\tau}_{\;\!\!\!#1}}
\newcommand{\rM}{\mathrm{M}}
\newcommand{\rN}{\mathrm{N}}
\DeclareMathOperator*{\barbigotimes}{\overline{\bigotimes}}
\DeclareMathOperator{\Ttau}{\mathnormal{T}^{\tau\!}}
\DeclareMathOperator{\TtauInn}{\mathnormal{T}^{\tau}_{\operatorname{Inn}\!}}
\DeclareMathOperator{\TtauAInn}{\mathnormal{T}^{\tau}_{\overline{\operatorname{Inn}}}}
\DeclareMathOperator{\Sd}{\mathnormal{S}\!\!\;\mathnormal{d}}

\numberwithin{equation}{section}

\title[Compact quantum groups with $\operatorname{\mathsf{L}^{\!\infty}}(\mathbb{G})$ a factor]{Examples of compact quantum groups with $\operatorname{\mathsf{L}^{\!\infty}}(\mathbb{G})$ a factor}

\author{Jacek Krajczok}
\address{School of Mathematics and Statistics, University of Glasgow, UK}
\email{jacek.krajczok@vub.be}

\author{Piotr M.~So{\l}tan}
\address{Department of Mathematical Methods in Physics, Faculty of Physics, University of Warsaw}
\email{piotr.soltan@fuw.edu.pl}


\makeatletter
\@namedef{subjclassname@2020}{\textup{2020} Mathematics Subject Classification}
\makeatother


\begin{document}

\begin{abstract}
For each $\lambda\in\left]0,1\right]$ we exhibit an uncountable family of compact quantum groups $\mathbb{G}$ such that the von Neumann algebra $\mathsf{L}^{\!\infty}(\mathbb{G})$ is the injective factor of type $\mathrm{III}_\lambda$ with separable predual. We also show that uncountably many injective factors of type $\mathrm{III}_0$ arise as $\mathsf{L}^{\!\infty}(\mathbb{G})$ for some compact quantum group $\mathbb{G}$. We introduce natural invariants of quantum groups related to the scaling group, modeled on the Connes invariant $T$ for von Neumann algebras.
We compute the values of these invariants in our examples, which allows us to distinguish between them. We also investigate their connection to the Connes invariants $T(\mathsf{L}^{\!\infty}(\mathbb{G}))$, $S(\mathsf{L}^{\!\infty}(\mathbb{G}))$.

In the final section we show that for a compact quantum group $\mathbb{G}$ the von Neumann algebra $\mathsf{L}^{\!\infty}(\mathbb{G})$ cannot have a direct summand of the form $\operatorname{B}(\mathsf{H})$ with $\dim{\mathsf{H}}=\infty$. In particular, factors of type $\mathrm{I}$ (of any dimension) cannot be obtained as $\mathsf{L}^{\!\infty}(\mathbb{G})$ for a non-trivial compact quantum group $\mathbb{G}$.
\end{abstract}

\maketitle

\tableofcontents

\allowdisplaybreaks

\section{Introduction}

The theory of quantum groups on operator algebra level has for a long time been a rich source of examples of interesting operator algebras. Beginning with the algebra $\C(\operatorname{SU}_q(2))$ associated to the famous quantum $\operatorname{SU}(2)$ group (\cite{su2}) there has been a lot of interest in the operator algebraic properties of the \cst-algebras and von Neumann algebras arising as ``functions on quantum groups'' (see e.g.~\cite{BrannanVergnioux,DeCommerFreslonYamashita} or \cite{typeI}). In this paper we will show that for any $\lambda\in\left]0,1\right]$ there exist uncountably many pairwise non-isomorphic compact quantum groups $\GG$ with the von Neumann algebra $\Linf(\GG)$ isomorphic to the injective factor of type $\mathrm{III}_\lambda$ with separable predual. Furthermore, uncountably many different type $\mathrm{III}_0$ factors with separable predual arise as $\Linf(\GG)$ as well. In particular, any of these von Neumann algebras carries a faithful and ergodic action of a compact quantum group.

Von Neumann factors appeared very early in the theory of compact quantum groups. Already in \cite{Banica} it was shown that the algebra of functions on the free unitary group $\operatorname{U}_2^+$ is isomorphic to the group von Neumann algebra of the free group on two generators, which in our notation would be written as $\Linf(\operatorname{U}_2^+)\cong\Linf(\hh{\mathbb{F}_2})$. The investigation of von Neumann algebras arising from compact quantum groups is still being carried out with a recent important result of \cite{BrannanVergnioux} stating that for $n\geq{3}$ the von Neumann algebra of functions on the free orthogonal group $\operatorname{O}_n^+$ (which is known to be a factor of type $\mathrm{II}_1$) is not isomorphic to $\Linf(\hh{\mathbb{F}_k})$ for any $k$. In fact, factors of type $\mathrm{II}_1$ are fairly common in the context of quantum groups, but only of Kac type (\cite[Remark A.2]{qBohr}). However, many non-isomorphic quantum groups may share the same von Neumann algebra. For example the injective factor of type $\mathrm{II}_1$ appears as the group von Neumann algebra of uncountably many pairwise non-isomorphic amenable groups (see e.g.~\cite[Section 2.4]{delaHarpe}).

The first comprehensive work describing how various von Neumann factors appear as algebras related to quantum groups was carried out by Pierre Fima (\cite{Fima}) who studied locally compact quantum groups $\GG$ such that the von Neumann algebras $(\Linf(\GG),\Linf(\hh{\GG}))$ are a pair of factors of types $(\mathrm{I}_\infty,\mathrm{I}_\infty)$, $(\mathrm{II}_\infty,\mathrm{II}_\infty)$ and $(\mathrm{III}_\lambda,\mathrm{III}_\lambda)$ for $\lambda\in[0,1]$ (see Section \ref{ConnesTandS} for a brief survey of these topics and further references). However, Fima's examples are not compact, since for a non-trivial compact quantum group $\GG$ the algebra $\Linf(\hh{\GG})$ is never a factor.

Let us mention here that due to existence of the unitary antipode (see \cite{KustermansVaes,SoltanWoronowicz}) any von Neumann algebra of the form $\Linf(\GG)$ with $\GG$ a locally compact quantum group is anti-isomorphic to itself. Thus in view of \cite{Connes-anti} one cannot expect to obtain all factors of type $\mathrm{III}$ as $\Linf(\GG)$. However, the examples of factors which are not anti-isomorphic to themselves are not injective (cf.~Remarks \ref{remninj0} and \ref{remninjlambda}).

Our investigation is based on the recent paper \cite{KrajczokWasilewski} where the authors constructed an uncountable family of pairwise non-isomorphic compact quantum groups all of whose algebras are isomorphic to the injective factor of type $\mathrm{II}_\infty$. We will also make use of other properties of these compact quantum groups discovered in \cite{KrajczokWasilewski}.

Let us briefly describe the contents of the paper. After recalling several notational conventions in Section \ref{conventions}, we give some necessary background and references on von Neumann factors and their invariants in Section \ref{ConnesTandS}. Section \ref{ITPFac} is devoted to standard results on infinite tensor products of factors which are usually formulated for tensor products of finite factors, hence the need to include their proofs in the more general case. Section \ref{sectIPQG} deals with the construction of the infinite Cartesian product of compact quantum groups which will be instrumental in constructing our examples. This section also contains the crucial Theorem \ref{theCenters} concerning the center of $\Linf(\GG)^{\sigma}$ (fixed point subalgebra of the modular group) which will be used extensively later on.

The first batch of examples of compact quantum groups giving rise to various factors is constructed in Section \ref{TheExamples} which begins with the introduction of the building blocks: the quantum groups $\operatorname{SU}_q(2)$ and their bicrossed products $\HH_{\nu,q}$ defined in \cite[Section 4.2]{KrajczokWasilewski}. Then, in Section \ref{sectTheFactors1}, we construct examples of compact quantum groups $\GG$ with $\Linf(\GG)$ the injective factors of type $\mathrm{III}_\lambda$ for all $\lambda\in\left]0,1\right]$ as well as uncountably many (non)-injective factors of type $\mathrm{III}_0$.

In Section \ref{sectTtau} we introduce three natural invariants of compact quantum groups which are modeled on the Connes invariant $T$ for von Neumann algebras. They are defined using the scaling automorphism group. More precisely, real number $t$ belongs to invariant $\Ttau(\GG)$ (resp.~$\TtauInn(\GG)$, $\TtauAInn(\GG)$) if and only if $\tau^{\GG}_t$ is trivial (resp.~inner, approximately inner). Let us mention that invariant $\TtauInn(\GG)$ is similar to the invariant $T(\Linf(\GG),\Delta_{\GG})$ introduced in \cite[Definition 3.4]{Vaes}, however in general they are different (see Remark \ref{remark1}). These invariants carry interesting information about quantum group, and can be used to distinguish between compact quantum groups which share the same factor as their associated von Neumann algebra. After establishing some of their basic properties we compute the newly introduced invariants for the examples discussed previously. Next, in Section \ref{sectBicrossed}, we tweak those examples by taking the bicrossed product by an arbitrary subgroup of $\RR$ (with the discrete topology) acting by the scaling automorphisms. This results in uncountably many non-isomorphic compact quantum groups $\GG$ with $\Linf(\GG)$ isomorphic to the injective factor of type $\mathrm{III}_\lambda$ for any $\lambda\in\left]0,1\right]$. Taking Cartesian product with $\hh{\mathbb{F}_2}$ (the dual of the free group on two generators) produces non-injective factors. The results of this section rely on information on the relative commutant of the algebra of class functions on $\HH_{\nu,q}$ discovered in \cite{KrajczokWasilewski}. In Section \ref{TtauandT} we provide more information on the connection between the quantum group invariants introduced in Section \ref{sectTtau} and the Connes invariant $T$ of the corresponding von Neumann algebra. In particular, we discuss the importance of a certain symmetry assumption on the spectra of the $\uprho$-operators (see Section \ref{conventions}).

Section \ref{sectTypeI} is devoted to the proof of an unrelated result that factors of type $\mathrm{I}_\infty$ cannot arise as algebras of functions on compact quantum groups. Finally the Appendix (Section \ref{sectApp}) deals with the technical aspects of infinite tensor products of closed operators needed for the results of Section \ref{ITPFac}.

\subsection{Notation and conventions}\label{conventions}\hspace*{\fill}

We refer to \cite[Chapter 1]{NeshveyevTuset} for the rudiments of the theory of compact quantum groups. We will also adopt some of the notational conventions of this monograph. Thus, for a compact quantum group $\GG$ described by the \cst-algebra with comultiplication $(\C(\GG),\Delta)$ the Haar measure of $\GG$ will be denoted by $\bh_\GG$ (or $\bh$ if there is no danger of confusion). The G.N.S.~Hilbert space for $\bh$ will be denoted by $\Ltwo(\GG)$ and we will write $\Linf(\GG)$ for the strong closure of the image of $\C(\GG)$ under the G.N.S.~representation. The canonical cyclic vector (image of the unit of $\C(\GG)$ in $\Ltwo(\GG)$) will be denoted by $\Omega_\bh$ and the objects of the Tomita-Takesaki theory (\cite{Connes}, \cite[Chapter III]{Takesaki2}) will be $S_\bh$, $J_\bh$ and $\modOp[\bh]$ (see also Section \ref{ConnesTandS}). Finally $\|\cdot\|_2$ will denote the Hilbert space norm on $\Linf(\GG)$ coming from the injection $\Linf(\GG)\ni{x}\mapsto{x}\Omega_\bh\in\Ltwo(\GG)$.

We will use the symbol $\operatorname{Irr}{\GG}$ to denote the set of equivalence classes of irreducible representations of $\GG$. For each $\alpha\in\operatorname{Irr}{\GG}$ we will always make a choice of a unitary representation $U^\alpha\in\alpha$. Recall that $U^\alpha$ is then a unitary element of the \cst-algebra $\B(\sH_\alpha)\tens\C(\GG)$ for a certain finite-dimensional Hilbert space $\sH_\alpha$. We will furthermore always choose an orthonormal basis of $\sH_\alpha$ diagonalizing the canonical operator $\uprho_{U^\alpha}$ responsible for the modular properties of $\bh$ (\cite[Section 1.4]{NeshveyevTuset}) and for $i,j\in\{1,\dotsc,\dim{U^\alpha}\}$ the symbols $U^\alpha_{i,j}$ ($i,j\in\{1,\dotsc,\dim{U^\alpha}\}$) will denote the matrix elements of $U^\alpha$ with respect to this basis. Furthermore we will sometimes refer to $\uprho_{U^\alpha}$ simply as $\uprho_\alpha$ and denote the eigenvalues of this operator by $\uprho_{\alpha,1},\dotsc,\uprho_{\alpha,\dim{\alpha}}$, where we used $\dim{\alpha}$ as a synonym for $\dim{U^\alpha}$. We will also sometimes refer to $\uprho_\alpha$ without specifying the class $\alpha$ as the \emph{$\uprho$-operator}.

The scaling group and the modular group of $\GG$ will be denoted by $(\tau^\GG_t)_{t\in\RR}$ and $(\sigma^\bh_t)_{t\in\RR}$ respectively. Throughout the paper we will be using the well-known formulas for the action of these groups on matrix elements of irreducible representations, namely
\begin{equation}\label{sigmatau}
\left\{\begin{array}{@{}r@{\;=\;}l@{}}
\sigma^\bh_t(U^\alpha_{i,j})&\uprho_{\alpha,i}^{\ii{t}}U^\alpha_{i,j}\uprho_{\alpha,j}^{\ii{t}}\\[5pt]
\tau^\GG_t(U^\alpha_{i,j})&\uprho_{\alpha,i}^{\ii{t}}U^\alpha_{i,j}\uprho_{\alpha,j}^{-\ii{t}}
\end{array}\right.,\quad\qquad{t}\in\RR,\:i,j\in\{1,\dotsc,\dim{\alpha}\}
\end{equation}
(cf.~\cite[Section 1.7]{NeshveyevTuset}). The subalgebras of elements $x\in\Linf(\GG)$ invariant under the modular or the scaling group will be denoted respectively by $\Linf(\GG)^{\sigma}$ and $\Linf(\GG)^\tau$.

We will use $\operatorname{Pol}(\GG)$ to refer to the canonical dense Hopf $*$-subalgebra of $\C(\GG)$ and $\Lone(\GG)$ will denote the predual of $\Linf(\GG)$. Finally, we say that a compact quantum groups $\GG$ is \emph{second countable} if $\Lone(\GG)$ is separable (see also \cite[Lemma 14.6]{KrajczokTypeI} for equivalent conditions).

\section{Factors and Connes invariants}\label{ConnesTandS}

The theory of von Neumann algebras originates in the groundbreaking papers of Murray and von Neumann \cite{MurrayVonNeumann1,MurrayVonNeumann2,vonNeumann3,MurrayVonNeumann4} and it was already in \cite{MurrayVonNeumann1} that \emph{factors}, i.e.~von Neumann algebras with trivial center, were introduced and classified into types $\mathrm{I}$, $\mathrm{II}$ and $\mathrm{III}$ with further subdivisions. The theory has since been greatly developed and refined. We refer the reader to series of monographs such as \cite{Takesaki1,Takesaki2,Takesaki3} or \cite{StratilaZsido,Stratila} as well as e.g.~the collection \cite{Varenna} for general theory and details of various constructions. In what follows we will give a very brief account of results (mainly due to Connes) which we will need in our paper.

Via the famous Tomita-Takesaki theory a normal semifinite faithful (n.s.f.) weight $\ph$ on a von Neumann algebra $\rM$ gives rise to a positive self-adjoint (in most cases unbounded) operator on the G.N.S.~Hilbert space $\sH_\ph$ called the \emph{modular operator} for $\ph$ which we will denote $\modOp[\ph]$ (since the traditional notation $\Delta_\ph$ is likely to cause confusion with the standard notation for the coproduct) and the \emph{modular group} $\{\sigma_t^\ph\}_{t\in\RR}$ of $\rM$ defined by
\[
\pi_\ph\bigl(\sigma^\ph_t(x)\bigr)=\modOp[\ph]^{\ii{t}}\pi_\ph(x)\modOp[\ph]^{-\ii{t}},\qquad{t}\in\RR,\:x\in\rM,
\]
where $\pi_\ph$ is the G.N.S.~representation. The famous result of Connes about the modular homomorphism states that the class of $\sigma_t^\ph$ in the outer automorphism group $\operatorname{Out}(\rM)=\operatorname{Aut}(\rM)/\operatorname{Inn}(\rM)$ does not depend on $\ph$ and is therefore intrinsically determined by $\rM$ itself. The invariant $T(\rM)$ is defined as the kernel of the modular homomorphism, i.e.
\[
T(\rM)=\bigl\{t\in\RR\,\bigr|\bigl.\,\sigma_t^\ph\in\operatorname{Inn}(\rM)\bigr\}.
\]
It is known that $\rM$ is semifinite if and only if $T(\rM)=\RR$ (the ``if'' part holds only for $\rM$ with separable predual). Furthermore this invariant can be used to distinguish many non-isomorphic factors (in particular, the famous Powers' factors $\mathrm{R}_\lambda$, \cite{Powers}).

The other invariant introduced by Connes is $S(\rM)$ which is defined as
\begin{equation}\label{SM}
S(\rM)=\bigcap_\ph\operatorname{Sp}{\modOp[\ph]},
\end{equation}
where $\ph$ runs over the set of all n.s.f.~weights on a factor $\rM$. The first crucial property of $S(\rM)$ is that for a type $\mathrm{III}$ factor the subset $S(\rM)\subset\RR_{\geq{0}}$ must be one of the following
\begin{itemize}
\item $S(\rM)=\{0\}\cup\{\lambda^n\,\bigr|\bigl.\,n\in\ZZ\}$ for some $\lambda\in\left]0,1\right[$,
\item $S(\rM)=\{0,1\}$,
\item $S(\rM)=\RR_{\geq{0}}$.
\end{itemize}
Assigning the to the second case $\lambda=0$ and $\lambda=1$ to the third one, we obtain a classification of type $\mathrm{III}$ factors into classes $\mathrm{III}_\lambda$ ($\lambda\in[0,1]$). Moreover, it is known that for $\lambda>0$, among von Neumann algebras with separable predual, there is only one injective factor of type $\mathrm{III}_\lambda$ (for $\lambda=1$ this result is due to Haagerup \cite{Haagerup}).

The second crucial property of $S(\rM)$ is that it is often computable. More precisely, if $\rM$ is a factor then the intersection of spectra \eqref{SM} is equal to the spectrum of one particular modular operator $\modOp[\ph]$ if $\ph$ is such that its \emph{centralizer}
\[
\rM_\ph=\bigl\{y\in\rM\,\bigr|\bigl.\,\sigma_t^\ph(y)=y\text{ for all }t\in\RR\bigr\}
\]
is a factor (\cite[Theorem 28.3]{Stratila}, see also \cite[Theorem 3.11 \emph{b})]{Connes}).

\subsection{Infinite tensor product of semifinite factors}\label{ITPFac}\hspace*{\fill}

For $n\in\NN$ let $\rN_n$ be a factor equipped with a n.s.f.~tracial weight $\btau[n]$ represented on the G.N.S.~Hilbert space $\sH_n$ for a faithful normal non-tracial state $\omega_n$. Let $\Omega_n\in\sH_n$ be the corresponding cyclic vector. By \cite[Section C.10.2]{StratilaZsido}, for each $n$ we have $\omega_n=\btau[n](\cdot{h_n})$ for a strictly positive self-adjoint operator $h_n$ affiliated with $\rN_n$.

The infinite tensor product
\[
\rN=\barbigotimes_{n=1}^{\infty}\rN_n
\]
is defined as the weak closure of the set of operators of the form
\[
a_1\tens\dotsm\tens{a_N}\tens\I\tens\I\tens\dotsm
\]
(often denoted by $a_1\tens\dotsm\tens{a_N}\tens\I^{\tens\infty}$) on the infinite tensor product $\bigotimes\limits_{n=1}^{\infty}(\sH_n,\Omega_n)$ of Hilbert spaces with distinguished unit vector, where $a_i\in\rN_i$ for each $i$ and $N$ is arbitrary. Furthermore we let $\omega$ be the infinite product state $\bigotimes\limits_{n=1}^{\infty}\omega_n$ on $\rN$. Then
\begin{itemize}
\item $\rN$ is a factor (\cite[Corollary XIV.1.10]{Takesaki3}),
\item we have $\sigma_t^\omega=\bigotimes\limits_{n=1}^{\infty}\sigma_t^{\omega_n}$ for all $t\in\RR$ (\cite[Proposition XIV.1.11]{Takesaki3}),
\item $\rN$ is semifinite if and only if $\sum\limits_{n=1}^\infty\bigl(1-|\omega_n(h_n^{\ii{t}})|\bigr)<+\infty$ for all $t\in\RR$ (\cite[Theorem XIV.1.14]{Takesaki3}).
\end{itemize}

\begin{proposition}\label{typeIIIfactor}
Let $\bigl\{(\rN_n,\omega_n)\bigr\}_{n\in\NN}$ and $\rN$ be as above. Assume that there is a von Neumann algebra with a faithful normal state which appears in the sequence $\bigl((\rN_n,\omega_n)\bigr)_{n\in\NN}$ infinitely many times. Then $\rN$ is of type $\mathrm{III}$.
\end{proposition}

\begin{proof}
Let $n_1<n_2<\dotsm$ be such that $(\rN_{n_k},\omega_{n_k})=(\rN_{n_1},\omega_{n_1})$ for all $k\geq{2}$. Since $\omega_{n_1}$ is not tracial, the operator $h=h_{n_1}$ is not a (real) scalar multiple of the identity, and hence there exists $t$ such that $h^{\ii{t}}$ is not a scalar multiple of the identity. It follows that $h^{\ii{t}}\Omega_{n_1}$ is not a scalar multiple of $\Omega_{n_1}$ because otherwise for any $x'\in\rN_{n_1}^{\:\prime}$ we would have $h^{\ii{t}}x'\Omega_{n_1}=x'h^{\ii{t}}\Omega_{n_1}=\lambda{x'}\Omega_{n_1}$ for some scalar $\lambda$ and $\Omega_{n_1}$ is cyclic for $\rN_{n_1}'$. Consequently we must have
\[
\bigl|\omega_{n_1}(h^{\ii{t}})\bigr|=\bigl|\is{\Omega_{n_1}}{h^{\ii{t}}\Omega_{n_1}}\bigr|<\|\Omega_{n_1}\|^2=1.
\]
It follows that a non-zero term is repeated infinitely many times in the series $\sum\limits_{n=1}^\infty\bigl(1-|\omega_n(h_n^{\ii{t}})|\bigr)$ and so $\rN$ is not semifinite. Since it is a factor, it must be of type $\mathrm{III}$.
\end{proof}

\begin{proposition}\label{injTP}
Let $\bigl\{(\rN_n,\omega_n)\bigr\}_{n\in\NN}$ and $\rN$ be as above. If $\rN_n$ is injective for all $n$ then $\rN$ is injective.
\end{proposition}

\begin{proof}
We need to show that $\rN$ is hyperfinite, i.e.~there is an increasing directed family of finite-dimensional unital $*$-subalgebras $\rM_i\subset\rN$ such that $\bigcup\limits_i\rM_i$ is strongly dense in $\rN$. For each $n$ there exists such a family for $\rN_n$, $\{\rM_{i,n}\}_{i\in{I_n}}$. Then for $N\in\NN,i_1\in{I_1},\dotsc,i_N\in{I_N}$ define
\[
\rM_{i_1,\dotsc,i_N}=\operatorname{span}\bigl(\widetilde{\rM}_{i_1,1}\dotsm\widetilde{\rM}_{i_N,N}\bigr),
\]
where $\widetilde{\rM}_{i_k,k}$ is $\rM_{i_k,k}$ canonically embedded into $\rN$ (so the subalgebras for different $k$ commute). These clearly are finite-dimensional unital $*$-subalgebras and they form an increasing directed family if we declare $(N,i_1,\dotsc,i_N)\leq(M,j_1,\dotsc,j_M)$ if and only if $N\leq{M}$ and $i_k\leq{j_k}$ in $I_k$ for $1\leq{k}\leq{N}$. The union of all these subalgebras is strongly dense in $\rN$.
\end{proof}

Let us now consider the objects of modular theory arising from the product state $\omega$ on $\rN$. We already mentioned that for each $t$ we have
\begin{equation}\label{sigmaprod}
\sigma_t^\omega=\bigotimes_{n=1}^{\infty}\sigma_t^{\omega_n}
\end{equation}
i.e.~$\sigma_t^\omega$ is the unique $\sigma$-weakly continuous extension of the map
\[
x_1\tens\dotsm\tens{x_N}\tens\I\tens\dotsm\longmapsto
\sigma_t^{\omega_1}(x_1)\tens\dotsm\tens\sigma_t^{\omega_N}(x_N)\tens\I\tens\dotsm
\]
defined on a dense subset of $\rN$. One of the consequences of the identification of the modular group for the product state is the following fundamental result on the invariant $T$ of infinite tensor products of semifinite factors (cf.~\cite[Lemma, p.~812]{Stormer}, \cite[Theorem 2.7]{Connes}):

\begin{proposition}\label{propStormer}
Let $\rN$ be the infinite tensor product of semifinite factors as above. Then a real number $t$ belongs to $T(\rN)$ if and only if
\[
\sum_{n=1}^\infty\Bigl(1-\bigl|\btau[n]\bigl(h_n^{1+\ii{t}}\bigr)\bigr|\Bigr)<+\infty.
\]
\end{proposition}

\begin{proof}
For each $n\in\NN$ and $t\in\RR$ we have $\sigma_t^{\omega_n}=\operatorname{Ad}\bigl(h_n^{\ii{t}}\bigr)$. Then \cite[Theorem XIV.1.13]{Takesaki3} tells us that $\sigma_t^\omega=\bigotimes\limits_{n=1}^\infty\sigma_t^{\omega_n}\in\operatorname{Inn}(\rN)$ if and only if
\[
\sum_{n=1}^\infty\Bigl(1-\bigl|\omega_n(h_n^{\ii{t}}\bigr)\bigr|\Bigr)
=\sum_{n=1}^\infty\Bigl(1-\bigl|\btau[n]\bigl(h_n^{1+\ii{t}}\bigr)\bigr|\Bigr)<+\infty.
\]
\end{proof}

One can also identify the G.N.S.~Hilbert space for $\omega$ with the infinite tensor product $\bigotimes\limits_{n=1}^{\infty}(\sH_n,\Omega_n)$. The next proposition, based on Lemma \ref{lemB} from the appendix, identifies the corresponding modular operator.\footnote{Note that since the K.M.S.~condition characterizes the modular group uniquely, it is not necessary to study the G.N.S.~space for $\omega$ or compute the modular operator to obtain \eqref{sigmaprod}.}

\begin{proposition}\label{propPsNabla}
The modular operator for the state $\omega$ is $\modOp[\omega]=\bigotimes\limits_{n=1}^{\infty}\modOp[\omega_n]$ and
\[
\operatorname{Sp}\modOp[\omega]
=\overline{\bigl\{\lambda_1\dotsm\lambda_N\,\bigr|\bigl.\,N\in\NN,\:\lambda_i\in\operatorname{Sp}(\modOp[\omega_i]),\:i\in\{1,\dotsc,N\}\bigr\}}.
\]
\end{proposition}

\begin{proof}
As we already mentioned the G.N.S.~Hilbert space for $\omega$ can be identified with $\bigotimes\limits_{n=1}^{\infty}(\sH_n,\Omega_n)$ and the corresponding cyclic vector is
\[
\Omega=\Omega_1\tens\Omega_2\tens\dotsm.
\]
For any $N\in\NN$ and $i\in\{1,\dotsc,N\}$ take $x_i\in\rN_i$. By \eqref{sigmaprod} we have
\begin{align*}
\modOp[\omega]^{\ii{t}}(x_1\tens\dotsm\tens{x_N}\tens\I^{\tens\infty})\Omega
&=\bigl(\sigma_t^{\omega_1}(x_1)\tens\dotsm\tens\sigma_t^{\omega_N}(x_N)\tens\I^{\tens\infty}\bigr)\Omega\\
&=\sigma_t^{\omega_1}(x_1)\Omega_1\tens\dotsm\tens\sigma_t^{\omega_N}(x_N)\Omega_N\tens\Omega_{N+1}\tens\dotsm\\
&=\modOp[\omega_1]^{\ii{t}}x_1\Omega_1\tens\dotsm\tens\modOp[\omega_N]^{\ii{t}}x_N\Omega_N\tens\Omega_{N+1}\tens\dotsm\\
&=\bigl(\modOp[\omega_1]^{\ii{t}}\tens\dotsm\tens\modOp[\omega_N]^{\ii{t}}\tens\I^{\tens\infty}\bigr)(x_1\tens\dotsm\tens{x_N}\tens\I^{\tens\infty})\Omega\\
&=\biggl(\bigotimes\limits_{n=1}^{\infty}\modOp[\omega_n]^{\ii{t}}\biggr)(x_1\tens\dotsm\tens{x_N}\tens\I^{\tens\infty})\Omega
\end{align*}
which implies
\[
\modOp[\omega]^{\ii{t}}=\bigotimes\limits_{n=1}^{\infty}\modOp[\omega_n]^{\ii{t}}.
\]
Now by Lemma \ref{lemB} and \cite[Theorem VIII.20]{ReedSimon1} for each $t\in\RR$ we have
\[
\biggl(\bigotimes\limits_{n=1}^{\infty}\modOp[\omega_n]\biggr)^{\ii{t}}
=\text{\sc{sot}-}\!\!\!\lim_{N\to\infty}\biggl(\bigotimes_{n=1}^N\modOp[\omega_n]\tens\I^{\tens\infty}\biggr)^{\ii{t}}
=\text{\sc{sot}-}\!\!\!\lim_{N\to\infty}\biggl(\bigotimes_{n=1}^N\modOp[\omega_n]^{\ii{t}}\tens\I^{\tens\infty}\biggr)=\bigotimes_{n=1}^{\infty}\modOp[\omega_n]^{\ii{t}}.
\]
Consequently $\modOp[\omega]=\bigotimes\limits_{n=1}^{\infty}\modOp[\omega_n]$.
\end{proof}

\section{Infinite product of compact quantum groups}\label{sectIPQG}

The construction of what should be called the Cartesian product of compact quantum groups was described under the name ``tensor product of compact quantum groups'' already in \cite{WangTensor}. We will briefly sketch the steps of the construction of the infinite Cartesian product of compact quantum groups in the context of von Neumann algebras.

Let $\{\GG_n\}_{n\in\NN}$ be a countable family of compact quantum groups with Haar measures $\{\bh_n\}_{n\in\NN}$. Let us, for the time being, denote by $\sH_n$ the Hilbert space $\Ltwo(\GG_n)$ and for each $n$ let $\Omega_n\in\sH_n$ be the corresponding cyclic vector. The compact quantum group $\GG=\bigtimes\limits_{n=1}^{\infty}\GG_n$ is defined by letting
\[
\Linf(\GG)=\barbigotimes_{n=1}^{\infty}\Linf(\GG_n)
\]
which is a von Neumann algebra acting on the Hilbert space
\[
\sH=\bigotimes_{n=1}^{\infty}(\sH_n,\Omega_n).
\]
The product state $\bh=\bigotimes\limits_{n=1}^\infty\bh_n$ on $\Linf(\GG)$ is the vector state for the (cyclic) vector $\Omega=\Omega_1\tens\Omega_2\tens\dotsm\in\sH$ which is also separating (cf.~\cite[Proposition XIV.1.11]{Takesaki3}) and $(\sH,\Omega)$ can be identified with the G.N.S.~representation of $\Linf(\GG)$ for $\bh$ (\cite[page 86]{Takesaki3}).

To define the comultiplication on $\Linf(\GG)$ we denote by $U$ the unitary operator
\[
\sH\tens\sH\longrightarrow\bigotimes\limits_{n=1}^{\infty}(\sH_n\tens\sH_n,\Omega_n\tens\Omega_n)
\]
extending the isometric map given on simple tensors by
\[
\resizebox{\textwidth}{!}{\ensuremath{
(\xi_1\tens\dotsm\tens\xi_N\tens\Omega_{N+1}\tens\dotsm)
\tens(\eta_1\tens\dotsm\tens\eta_N\tens\Omega_{N+1}\tens\dotsm)\longmapsto
(\xi_1\tens\eta_1)\tens\dotsm\tens(\xi_N\tens\eta_N)\tens(\Omega_{N+1}\tens\Omega_{N+1})\tens\dotsm.
}}
\]
Then with $W_n$ denoting the \emph{Kac-Takesaki} operator of $\GG_n$ (which can be identified with the \emph{right regular representation}, \cite[Section 1.5]{NeshveyevTuset}) we set $W=U^*\biggl(\bigotimes\limits_{n=1}^{\infty}W_n\biggr)U\in\B(\sH\tens\sH)$ which is well defined since $W_n(\Omega_n\tens\Omega_n)=\Omega_n\tens\Omega_n$ for all $n$ (\cite[Theorem 1.5.2]{NeshveyevTuset}). The comultiplication $\Delta\colon\Linf(\GG)\to\Linf(\GG)\vtens\Linf(\GG)$ is then
\[
\Delta(a)=W(a\tens\I)W^*,\qquad{a}\in\Linf(\GG).
\]
One easily checks that the above defines a compact quantum group and that $\bh$ is its Haar measure.

As we already mentioned, on the \cst-algebraic level the infinite Cartesian product of a family of compact quantum groups was first described in \cite{WangTensor}. In particular, the set of equivalence classes of irreducible representations of $\GG=\bigtimes\limits_{n=1}^{\infty}\GG_n$ was found to be
\begin{equation}\label{IrrProdG}
\operatorname{Irr}{\GG}=\biggl\{\bigboxtimes_{n=1}^\infty\alpha_n\,\bigg|\biggl.\,\text{almost all $\alpha_n$ are trivial}\biggr\},
\end{equation}
where $\btens$ denotes the \emph{exterior tensor product}, so that $\alpha_1\btens\alpha_2$ is the class of $\bigl(U^{\alpha_1}\bigr)_{13}\bigl(U^{\alpha_2}\bigr)_{24}$ with $U^{\alpha_1}\in\alpha_1$ and $U^{\alpha_2}\in\alpha_2$. By a slight abuse of notation we will denote the infinite exterior tensor product $\bigboxtimes\limits_{n=1}^\infty\alpha_n$ by $\alpha_1\btens\dotsm\btens\alpha_N$ whenever $\alpha_n$ is trivial for all $n>N$.

In what follows we fix for every $n$ and every $\alpha\in\operatorname{Irr}{\GG_n}$ a representation $U^\alpha$ of $\GG_n$ on $\sH_\alpha$ and an orthonormal basis of $\sH_\alpha$ such that the canonical operator $\uprho_\alpha$ is diagonal with diagonal entries $\uprho_{\alpha,1},\dotsc,\uprho_{\alpha,\dim{\alpha}}$. The matrix elements of $U^\alpha$ are taken with respect to this basis.

\begin{theorem}\label{theCenters}
Let $\{\GG_n\}_{n\in\NN}$ be a family of compact quantum groups such that each $\GG_m$ appears in the sequence $(\GG_n)_{n\in\NN}$ infinitely many times and let $\GG=\bigtimes\limits_{n=1}^\infty\GG_n$. Then
\begin{enumerate}
\item\label{theCenters1} $\cZ\bigl(\Linf(\GG)^\sigma\bigr)=\Linf(\GG)\cap\bigl(\Linf(\GG)^\sigma\bigr)'=\cZ\bigl(\Linf(\GG)\bigr)$,
\item\label{theCenters2} $\cZ\bigl(\Linf(\GG)^\tau\bigr)\subset\Linf(\GG)\cap\bigl(\Linf(\GG)^\tau\bigr)'=\cZ\bigl(\Linf(\GG)\bigr)$.
\end{enumerate}
\end{theorem}

\begin{remark}
We have $\cZ(\Linf(\GG))=\barbigotimes\limits_{n=1}^{\infty}\cZ(\Linf(\GG_n))$ (see \cite[Section XIV.1]{Takesaki3}).
\end{remark}

\begin{proof}[Proof of Theorem \ref{theCenters}]
Ad \eqref{theCenters1}. Since $\sigma^\bh$ acts trivially on the center (\cite[Proposition VI.1.23]{Takesaki2}), we have $\cZ(\Linf(\GG))\subset\cZ(\Linf(\GG)^\sigma)\subset\Linf(\GG)\cap(\Linf(\GG)^\sigma)'$.

Take $x\in\Linf(\GG)\cap(\Linf(\GG)^\sigma)'$. In order to show that $x\in\cZ(\Linf(\GG))$ we fix $k\in\NN$, $\beta\in\operatorname{Irr}{\GG_k}$, $i,j\in\{1,\dotsc,\dim{\beta}\}$ and $\eps>0$.

Choose $y\in\operatorname{Pol}(\GG)$ such that $\|x-y\|_2\leq\tfrac{\eps}{2}$. There exists $N\in\NN$ such that $y$ is a linear combination of matrix elements of representations from classes of the form $\alpha_1\btens\dotsm\btens\alpha_N$. By assumption each of the compact quantum groups appears in the sequence $(\GG_n)_{n\in\NN}$ infinitely many times, so there exists $k'>N$ such that $\GG_{k'}=\GG_n$. Since $\uprho_{\overline{\beta}}=\jmath(\uprho_\beta)^{-1}$ (here $\jmath$ denotes the mapping of an operator on $\sH_\beta$ to its adjoint acting on $\sH_\beta^*$, cf.~\cite[Proposition 1.4.7]{NeshveyevTuset}) there exist $i',j'\in\{1,\dotsc,\dim{\overline{\beta}}\}$ such that $\uprho_{\overline{\beta},i'}=(\uprho_{\beta,i})^{-1}$ and $\uprho_{\overline{\beta},j'}=(\uprho_{\beta,j})^{-1}$. Consequently the element of $\operatorname{Pol}(\GG)$
\[
z=\I^{\tens(k-1)}\tens{U^\beta_{i,j}}\tens\I^{\tens(k'-k-1)}\tens{U^{\bar{\beta}}_{i',j'}}\tens\I^{\tens\infty}
\]
belongs to $\Linf(\GG)^\sigma$ (cf.~\eqref{sigmatau}) and it follows that $z$ commutes with $x$. From this, the $\sigma^\bh$-invariance of $z$ and the fact that $z$ is a contraction we obtain\footnote{\label{sigmaft}We use the fact that all elements of $\operatorname{Pol}(\GG)$ are analytic for the modular group and for any such $a$ and any $b\in\Linf(\GG)$ we have $ba\Omega=(a^*b^*)^*\Omega=S_{\bh}a^*b^*\Omega=S_{\bh}a^*S_{\bh}b\Omega=J_\bh\modOp[\bh]^{\frac{1}{2}}a^*J_\bh\modOp[\bh]^{\frac{1}{2}}b\Omega=J_\bh\modOp[\bh]^{\frac{1}{2}}a^*\modOp[\bh]^{-\frac{1}{2}}J_\bh{b}\Omega=J_\bh\bigl(\modOp[\bh]^{-\frac{1}{2}}a\modOp[\bh]^{\frac{1}{2}}\bigr)^*J_\bh{b}\Omega=J_\bh\sigma_{\frac{\ii}{2}}^\bh(a)^*J_\bh{b}\Omega$.}
\begin{align*}
\|yz-&zy\|_2=\|yz\Omega-zy\Omega\|=\bigl\|J_\bh\sigma^\bh_{\frac{\ii}{2}}(z)^*J_\bh{y}\Omega-zy\Omega\bigr\|\\
&\leq\bigl\|J_\bh\sigma^\bh_{\frac{\ii}{2}}(z)^*J_\bh{y}\Omega-J_\bh\sigma^\bh_{\frac{\ii}{2}}(z)^*J_\bh{x}\Omega\bigr\|
+\bigl\|J_\bh\sigma^\bh_{\frac{\ii}{2}}(z)^*J_\bh{x}\Omega-zx\Omega\bigr\|
+\|zx\Omega-zy\Omega\|\\
&\leq\bigl\|\sigma^\bh_{\frac{\ii}{2}}(z)\bigr\|\|x\Omega-y\Omega\|
+\bigl\|J_\bh\sigma^\bh_{\frac{\ii}{2}}(z)^*J_\bh{x}\Omega-zx\Omega\bigr\|
+\|z\|\|x\Omega-y\Omega\|\\
&=\bigl\|\sigma^\bh_{\frac{\ii}{2}}(z)\bigr\|\|x-y\|_2
+\bigl\|J_\bh\sigma^\bh_{\frac{\ii}{2}}(z)^*J_\bh{x}\Omega-zx\Omega\bigr\|
+\|z\|\|x-y\|_2\\
&\leq\tfrac{\eps}{2}\Bigl(\bigl\|\sigma^\bh_{\frac{\ii}{2}}(z)\bigr\|+\|z\|\Bigr)
+\bigl\|J_\bh\sigma^\bh_{\frac{\ii}{2}}(z)^*J_\bh{x}\Omega-zx\Omega\bigr\|\\
&=\tfrac{\eps}{2}\bigl(\|z\|+\|z\|\bigr)
+\bigl\|xz\Omega-zx\Omega\bigr\|=\eps\|z\|+\|xz-zx\|_2\leq\eps.
\end{align*}

Next we write $y$ as $y_0\tens\I^{\tens\infty}$ for some $y_0\in\operatorname{Pol}(\GG_1\times\dotsm\times\GG_N)$. Since $k'>N$, we have
\[
yz-zy=
\Bigl(y_0\bigl(\I^{\tens(k-1)}\tens{U^{\beta}_{i,j}}\tens\I^{\tens(N-k)}\bigr)
-\bigl(\I^{\tens(k-1)}\tens{U^{\beta}_{i,j}}\tens\I^{\tens(N-k)}\bigr)y_0\Bigr)
\tens\I^{\tens(k'-N-1)}\tens{U^{\overline{\beta}}_{i',j'}}\tens\I^{\tens\infty}
\]
and consequently
\[
\Bigl\|
y_0\bigl(\I^{\tens(k-1)}\tens{U^{\beta}_{i,j}}\tens\I^{\tens(N-k)}\bigr)
-\bigl(\I^{\tens(k-1)}\tens{U^{\beta}_{i,j}}\tens\I^{\tens(N-k)}\bigr)y_0
\Bigr\|_2\bigl\|U^{\overline{\beta}}_{i',j'}\bigr\|_2=\|yz-zy\|_2\leq\eps.
\]
and from the orthogonality relations (\cite[Theorem 1.4.3]{NeshveyevTuset})
\[
\Bigl\|
y\bigl(\I^{\tens(k-1)}\tens{U^{\beta}_{i,j}}\tens\I^{\tens\infty}\bigr)
-\bigl(\I^{\tens(k-1)}\tens{U^{\beta}_{i,j}}\tens\I^{\tens\infty}\bigr)y
\Bigr\|_2\leq\eps\sqrt{\|\uprho_{\overline{\beta}}\|\dim_q{\beta}}.
\]

Now we go back to $x$: using the technique used to estimate $\|yz-zy\|_2$ above we obtain
\begin{align*}
&\Bigl\|x\bigl(\I^{\tens(k-1)}\tens{U^\beta_{i,j}}\tens\I^{\tens\infty}\bigr)-\bigl(\I^{\tens(k-1)}\tens{U^\beta_{i,j}}\tens\I^{\tens\infty}\bigr)x\Bigr\|_2\\
&\qquad\leq\tfrac{\eps}{2}\Bigl(\bigl\|\sigma^\bh_{\frac{\ii}{2}}\bigl(U^\beta_{i,j}\bigr)^*\bigr\|+\bigr\|U^\beta_{i,j}\bigl\|\Bigr)
+\Bigl\|y\bigl(\I^{\tens(k-1)}\tens{U^\beta_{i,j}}\tens\I^{\tens\infty}\bigr)-\bigl(\I^{\tens(k-1)}\tens{U^\beta_{i,j}}\tens\I^{\tens\infty}\bigr)y\Bigr\|_2\\
&\qquad\leq\Bigl(\|\uprho_{\overline{\beta}}\|+\sqrt{\|\uprho_{\overline{\beta}}\|\dim_q{\beta}}\Bigr)\eps
\end{align*}
again by the orthogonality relations and \eqref{sigmatau}.

Since $\eps>0$ in the above estimates is arbitrary and $\beta$ is fixed, we see that $x$ and $\I^{\tens(k-1)}\tens{U^\beta_{i,j}}\tens\I^{\tens\infty}$ commute. Furthermore, as this holds for any $\beta,i,j$, we find that
\[
x\in\Bigl(\I^{\tens(k-1)}\tens\Linf(\GG)\tens\I^{\tens\infty}\Bigr)'
\]
for any $k$. Consequently
\[
x\in\bigcap_{k=1}^{\infty}\Bigl(\I^{\tens(k-1)}\tens\Linf(\GG)\tens\I^{\tens\infty}\Bigr)'=\biggl(\bigvee_{k=1}^{\infty}\Bigl(\I^{\tens(k-1)}\tens\Linf(\GG)\tens\I^{\tens\infty}\Bigr)\biggr)'=\Linf(\GG)'.
\]

Statement \eqref{theCenters2} can be proved by an analogous argument.
\end{proof}

\begin{remark}
The difference between statements \eqref{theCenters1} and \eqref{theCenters2} stems from the fact that the scaling group can act non-trivially on the center of $\Linf(\GG)$. It is the case for $\operatorname{SU}_q(2)$ and hence for $\bigtimes\limits_{n=1}^{\infty}\operatorname{SU}_q(2)$.
\end{remark}

\section{Type \texorpdfstring{$\mathrm{III}$}{III} factors as \texorpdfstring{$\Linf(\GG)$}{L(G)}}\label{TheExamples}

Throughout this section $q$ and $q_n$ will denote numbers in $\left]-1,1\right[\setminus\{0\}$. The basic building block of our construction is the quantum group $\operatorname{SU}_q(2)$ constructed in \cite{su2}.

\subsection{\texorpdfstring{$\operatorname{SU}_q(2)$}{SUq(2)} and the bicrossed product \texorpdfstring{$\HH_{\nu,q}$}{Hvq}}\label{skladniki}\hspace*{\fill}

Let us first fix the notation related to the quantum group $\operatorname{SU}_q(2)$ and its dual $\hh{\operatorname{SU}_q(2)}$. In what follows $\bh_q$ will denote the Haar measure of $\operatorname{SU}_q(2)$ and $\Omega_{\bh_q}$ the canonical cyclic vector in $\Ltwo(\operatorname{SU}_q(2))$. Recall from \cite[Section 4.1]{KrajczokWasilewski} (cf.~also \cite{DesmedtPhD,modular}) that there is a unitary operator
\[
\cQ_{L,q}\colon\Ltwo\bigl(\operatorname{SU}_q(2)\bigr)\ni{a}\Omega_{\bh_q}\longmapsto{\Int_\TT}^\oplus\psi_\lambda(a)D_{\lambda,q}^{-1}\dd{\mu}(\lambda)\in
{\Int_\TT}^\oplus\HS(\sH_\lambda)\dd{\mu}(\lambda),
\]
where
\begin{itemize}
\item $\{\psi_\lambda\}_{\lambda\in\TT}$ are the irreducible representations of $\C(\operatorname{SU}_q(2))$ on $\sH_\lambda=\ell^2(\ZZ_+)$,
\item $\HS(\sH_\lambda)$ denotes the space of Hilbert-Schmidt operators on $\sH_\lambda$,
\item $D_{\lambda,q}$ is a positive self-adjoint operator on $\sH_\lambda$ such that
\[
D_{\lambda,q}^{-1}=\sqrt{1-q^2}\operatorname{diag}\bigl(1,|q|,|q|^2,\dotsc\bigr),\qquad\lambda\in\TT
\]
with respect to the standard basis of $\ell^2(\ZZ_+)$ (clearly $D_{\lambda,q}^{-1}$ is a positive Hilbert-Schmidt operator),
\item $\mu$ is the normalized Lebesgue measure on $\TT$.
\end{itemize}
The operator $\cQ_{L,q}$ induces an isomorphism
\[
\Linf\bigl(\operatorname{SU}_q(2)\bigr)\ni{a}\longmapsto\cQ_{L,q}a{\cQ_{L,q}^{\;*}}\in{\Int_\TT}^\oplus\bigl(\B(\sH_\lambda)\tens\I\bigr)\dd\mu(\lambda)
\]
and we will denote the composition of this isomorphism with the identification
\[
{\Int_\TT}^\oplus\bigl(\B(\sH_\lambda)\tens\I\bigr)\dd\mu(\lambda)\cong
{\Int_\TT}^\oplus\B(\sH_\lambda)\dd\mu(\lambda)
\]
by
\[
\Phi_q\colon\Linf\bigl(\operatorname{SU}_q(2)\bigr)\longrightarrow{\Int_\TT}^\oplus\B(\sH_\lambda)\dd\mu(\lambda).
\]

Observe that if $a\in\Linf(\operatorname{SU}_q(2))$ and $\Phi_q(a)={\Int_\TT}^{\oplus}a_\lambda\dd{\mu}(\lambda)$ then
\begin{align*}
\bh_q(a)=\is{\Omega_{\bh_q}}{a\Omega_{\bh_q}}&=\is{\cQ_{L,q}\Omega_{\bh_q}}{\bigl(\cQ_{L,q}a{\cQ_{L,q}^{\;*}}\bigr)\cQ_{L,q}\Omega_{\bh_q}}\\
&=\is{{\Int_\TT}^{\oplus}D_{\lambda,q}^{-1}\dd{\mu}(\lambda)}{{\Int_\TT}^{\oplus}{a_\lambda}D_{\lambda,q}^{-1}\dd{\mu}(\lambda)}
=\Int_\TT\operatorname{Tr}\bigl({a_\lambda}D_{\lambda,q}^{-2}\bigr)\dd{\mu}(\lambda),
\end{align*}
hence
\begin{equation}\label{HaarSUq2}
\bh_q=\biggl({\Int_\TT}^{\oplus}\operatorname{Tr}\bigl(\,{\cdot}\,D_{\lambda,q}^{-2}\bigr)\dd{\mu}(\lambda)\biggr)\comp\Phi_q.
\end{equation}

Our next ingredient will be the compact quantum group $\HH_{\nu,q}$ constructed in \cite[Section 4.2]{KrajczokWasilewski}. For $\nu\in\RR\setminus\{0\}$ the compact quantum group $\HH_{\nu,q}$ is defined to be the bicrossed product $\QQ\bowtie\operatorname{SU}_q(2)$ (\cite{FimaMukherjeePatri}) obtained from the action of the group $\QQ$ (with discrete topology) on $\Linf(\operatorname{SU}_q(2))$ via the action $\alpha^\nu$:
\[
\alpha^\nu_\gamma(x)=\tau^{\operatorname{SU}_q(2)}_{\nu\gamma}(x),\qquad{x}\in\Linf\bigl(\operatorname{SU}_q(2)\bigr),\:\gamma\in\QQ.
\]
In particular, $\Linf(\HH_{\nu,q})$ is by definition the crossed product $\QQ\ltimes_{\alpha^\nu}\Linf(\operatorname{SU}_q(2))$. The action $\alpha^\nu$ can be interpreted as a $*$-homomorphism
\[
\Linf(\operatorname{SU}_q(2))\longrightarrow\ell^\infty(\QQ)\vtens\Linf(\operatorname{SU}_q(2))
\]
whose range is (by definition) contained in the crossed product $\QQ\ltimes_{\alpha^\nu}\Linf(\operatorname{SU}_q(2))$. Thus
\[
\alpha^\nu\colon\Linf(\operatorname{SU}_q(2))\longrightarrow\QQ\ltimes_{\alpha^\nu}\Linf(\operatorname{SU}_q(2))=\Linf(\HH_{\nu,q}).
\]
It was shown in \cite{KrajczokWasilewski} that
\begin{itemize}
\item $\HH_{\nu,q}$ is coamenable, hence $\Linf(\HH_{\nu,q})$ is injective (\cite[Theorem 3.9]{Tomatsu}),
\item if $\nu\log|q|\not\in\pi\QQ$ then $\Linf(\HH_{\nu,q})$ is the injective factor of type $\mathrm{II}_\infty$.
\end{itemize}
In what follows we will always assume (or explicitly choose) the numbers $\nu$ and $q$ to satisfy $\nu\log|q|\not\in\pi\QQ$.

The unique (up to positive constant) trace $\btau[\nu,q]$ on $\Linf(\HH_{\nu,q})$ is given by
\begin{equation}\label{taunuq}
\btau[\nu,q]=\biggl({\Int_{\TT}}^\oplus\operatorname{Tr}_{\lambda}(\,\cdot\,)\dd{\mu}(\lambda)\biggr)\comp\Phi_q\comp\EE_{\nu,q},
\end{equation}
where $\operatorname{Tr}_\lambda$ denotes the trace on $\B(\sH_\lambda)$ and
\[
\EE_{\nu,q}\colon\Linf(\HH_{\nu,q})\longrightarrow\Linf\bigl(\operatorname{SU}_q(2)\bigr)
\]
is the canonical faithful normal conditional expectation of the crossed product onto the fixed points of the dual action. Furthermore, by \cite[Theorem 3.4$(1)$]{FimaMukherjeePatri} the Haar measure of $\HH_{\nu,q}$ is
\begin{equation}\label{HaarHnuq}
\bh_{\nu,q}=\bh_q\comp\EE_{\nu,q}.
\end{equation}
From the comparison of \eqref{HaarSUq2}, \eqref{taunuq} and \eqref{HaarHnuq} it follows that $\bh_{\nu,q}=\btau[\nu,q](\,\cdot\,k_{\nu,q})$ where the element $k_{\nu,q}\in\Linf(\HH_{\nu,q})$ is the image under $\alpha^\nu$ (see above) of
\[
\Phi_q^{-1}\biggl({\Int_\TT}^{\oplus}D_{\lambda,q}^{-1}\dd{\mu}(\lambda)\biggr).
\]

Finally let us recall from \cite[Theorem 6.1]{FimaMukherjeePatri} that for each $(\nu,q)\in(\RR\setminus\{0\})\times(\left]-1,1\right[\setminus\{0\})$ irreducible representations of $\HH_{\nu,q}$ are labeled by $\QQ\times\operatorname{Irr}{\operatorname{SU}_q(2)}=\QQ\times\tfrac{1}{2}\ZZ_+$. In this correspondence the representation corresponding to $(\gamma,s)\in\QQ\times\tfrac{1}{2}\ZZ_+$ is
\[
U^{\gamma,s}=(\I\tens{u_\gamma})(\id\tens\alpha^{\nu})(U^s)\in\B(\sH_s)\atens\operatorname{Pol}(\HH_{\nu,q}),
\]
where $U^s\in\B(\sH_s)\atens\operatorname{Pol}(\operatorname{SU}_q(2))$ is the spin-$s$ representation of $\operatorname{SU}_q(2)$, $u_\gamma$ is the canonical element of $\Linf(\HH_{\nu,q})=\QQ\ltimes_{\alpha^\nu}\Linf(\operatorname{SU}_q(2))$ implementing $\alpha^\nu_\gamma$ and $\alpha^\nu$ is understood as a $*$-homomorphism $\Linf(\operatorname{SU}_q(2))\to\Linf(\HH_{\nu,q})$ as above. The corresponding $\uprho$-operator is $\uprho_{U^{\gamma,s}}=\uprho_s$ (cf.~\cite[Section 4.2]{KrajczokWasilewski}). Using this and the first equation of \eqref{sigmatau} we can easily determine the action of the modular operator $\modOp[\bh_{\nu,q}]$ on the orthogonal basis of $\Ltwo(\HH_{\nu,q})$ consisting of vectors of the form $U^{\gamma,s}_{i,j}\Omega_{\bh_{\nu,q}}$:
\begin{equation}\label{nablaUgammas0}
\modOp[\bh_{\nu,q}]U^{\gamma,s}_{i,j}\Omega_{\bh_{\nu,q}}=\uprho_{s,i}\uprho_{s,j}U^{\gamma,s}_{i,j}\Omega_{\bh_{\nu,q}},\qquad(\gamma,s)\in\QQ\times\tfrac{1}{2}\ZZ_+,\:i,j\in\{1,\dotsc,2s+1\}
\end{equation}
(as usual the matrix elements of $U^{\gamma,s}$ are with respect to an orthonormal basis of $\sH_s$ diagonalizing $\uprho_s$). It is well known (see e.g.~\cite{Koornwinder}) that $\uprho_{s,i}=|q|^{2(i-s-1)}$, i.e.
\[
\uprho_s=\operatorname{diag}\bigl(|q|^{-2s},|q|^{-2s+2},\dotsc,|q|^{2s}\bigr),\qquad{s}\in\tfrac{1}{2}\ZZ_+,
\]
so that \eqref{nablaUgammas0} becomes
\[
\modOp[\bh_{\nu,q}]U^{\gamma,s}_{i,j}\Omega_{\bh_{\nu,q}}
=|q|^{2(i+j-2s-2)}
U^{\gamma,s}_{i,j}\Omega_{\bh_{\nu,q}},\qquad{s}\in\tfrac{1}{2}\ZZ_+,\:i,j\in\{1,\dotsc,2s+1\}.
\]
In particular,
\begin{equation}\label{SpUgammas}
\operatorname{Sp}{\modOp[\bh_{\nu,q}]}=\{0\}\cup|q|^{2\ZZ}.
\end{equation}

\subsection{The factors and their invariants}\label{sectTheFactors1}\hspace*{\fill}

We will now consider compact quantum groups of the form $\GG=\bigtimes\limits_{n=1}^\infty\HH_{\nu_n,q_n}$, where $\bigl((\nu_n,q_n)\bigr)_{n\in\NN}$ is a sequence of elements of $(\RR\setminus\{0\})\times(\left]-1,1\right[\setminus\{0\})$ such that $\nu_n\log|q_n|\not\in\pi\QQ$ for all $n$ and at least one pair is repeated infinitely many times in the sequence. Later we will impose further conditions on $\bigl((\nu_n,q_n)\bigr)_{n\in\NN}$.

\begin{theorem}\label{thmT}
$\Linf(\GG)$ is an injective factor of type $\mathrm{III}$ and the invariant $T(\Linf(\GG))$ is given by
\[
T\bigl(\Linf(\GG)\bigr)=\biggl\{t\in\RR\,\biggr|\biggl.\,\sum_{n=1}^{\infty}\Bigl(1-\tfrac{1-q_n^2}{\left|1-|q_n|^{2+2\ii{t}}\right|}\Bigr)<+\infty\biggr\}.
\]
\end{theorem}

\begin{proof}
$\Linf(\GG)$ is a factor of type $\mathrm{III}$ by Proposition \ref{typeIIIfactor} and it is injective by Proposition \ref{injTP}. From Proposition \ref{propStormer} we know that $t\in\RR$ belongs to $T(\Linf(\GG))$ if and only if
\[
\sum_{n=1}^{\infty}\Bigl(1-\bigl|\btau[\nu_n,q_n]\bigl(k_{\nu_n,q_n}^{1+\ii{t}}\bigr)\bigr|\Bigr)<+\infty.
\]
Finally, for any $n$ we have
\begin{align*}
\btau[\nu_n,q_n]\bigl(k_{\nu_n,q_n}^{1+\ii{t}}\bigr)&=\Int_\TT\operatorname{Tr}_\lambda\bigl((D_{\lambda,q_n}^{-2})^{1+\ii{t}}\bigr)\dd{\mu}(\lambda)\\
&=\bigl(1-q_n^2\bigr)^{1+\ii{t}}\Int_\TT\operatorname{Tr}_\lambda\Bigl(\operatorname{diag}\bigr(1,|q_n|^{2+2\ii{t}},|q_n|^{2(2+2\ii{t})},|q_n|^{3(2+2\ii{t})},\dotsc\bigr)\Bigr)\dd{\mu}(\lambda)\\
&=\bigl(1-q_n^2\bigr)^{1+\ii{t}}\sum_{k=0}^{\infty}|q_n|^{k(2+2\ii{t})}=\frac{(1-q_n^2)^{1+\ii{t}}}{1-|q_n|^{2+2\ii{t}}}.
\end{align*}
\end{proof}

Theorem \ref{thmT} allows us to identify the invariant $T(\Linf(\GG))$ in some special cases.

\begin{corollary}
\noindent
\begin{enumerate}
\item\label{corT1} If for all $n\in\NN$ we have $q_n=q$ for some $q\in\left]-1,1\right[\setminus\{0\}$ then $T(\Linf(\GG))=\tfrac{\pi}{\log|q|}\ZZ$.
\item\label{corT2} Assume that there are two subsequences $(q_{n_{1,p}})_{p\in\NN}$ and $(q_{n_{2,p}})_{p\in\NN}$ such that
\[
\bigl\{n_{1,p}\,\bigr|\bigl.\,p\in\NN\bigr\}\cap\bigl\{n_{2,p}\,\bigr|\bigl.\,p\in\NN\bigr\}=\emptyset
\]
and
\[
q_{n_{1,p}}\xrightarrow[p\to\infty]{}r_1,\quad
q_{n_{2,p}}\xrightarrow[p\to\infty]{}r_2
\]
for some $r_1,r_2\in\left]-1,1\right[\setminus\{0\}$ such that $\tfrac{\pi}{\log|r_1|}\ZZ\cap\tfrac{\pi}{\log|r_2|}\ZZ=\{0\}$. Then $T(\Linf(\GG))=\{0\}$.
\end{enumerate}
\end{corollary}

\begin{proof}
Ad \eqref{corT1}. When all $q_n$ are equal, by Theorem \ref{thmT} we have $t\in{T(\Linf(\GG))}$ if and only if $1-\tfrac{1-q^2}{\left|1-|q|^{2+2\ii{t}}\right|}=0$ which is easily seen to be equivalent to $|q|^{2\ii{t}}=1$, i.e.~$2\ii{t}\log|q|\in2\pi\ii\ZZ$.

Ad \eqref{corT2}. Clearly $0\in{T(\Linf(\GG))}$, so it is enough to consider a non-zero $t\in{T(\Linf(\GG))}$. Write
\[
q_{n_{1,p}}=r_1+\eps_{1,p}\quad\text{and}\quad
q_{n_{2,p}}=r_2+\eps_{2,p}.
\]
Since
\[
\sum_{n=1}^{\infty}\Bigl(1-\tfrac{1-q_n^2}{\left|1-|q_n|^{2+2\ii{t}}\right|}\Bigr)\geq
\sum_{p=1}^{\infty}\Bigl(1-\tfrac{1-(r_1+\eps_{1,p})^2}{\left|1-|r_1+\eps_{1,p}|^{2+2\ii{t}}\right|}\Bigr)
+
\sum_{p=1}^{\infty}\Bigl(1-\tfrac{1-(r_2+\eps_{2,p})^2}{\left|1-|r_2+\eps_{2,p}|^{2+2\ii{t}}\right|}\Bigr)
\]
the convergence of $\sum\limits_{n=1}^{\infty}\Bigl(1-\tfrac{1-q_n^2}{\left|1-|q_n|^{2+2\ii{t}}\right|}\Bigr)$ forces
\[
\tfrac{1-(r_1+\eps_{1,p})^2}{\left|1-|r_1+\eps_{1,p}|^{2+2\ii{t}}\right|}\xrightarrow[p\to\infty]{}1\quad\text{and}\quad
\tfrac{1-(r_2+\eps_{2,p})^2}{\left|1-|r_2+\eps_{2,p}|^{2+2\ii{t}}\right|}\xrightarrow[p\to\infty]{}1.
\]
This happens if and only if $|r_1|^{2\ii{t}}=1=|r_2|^{2\ii{t}}$, i.e.~when
\begin{equation}\label{nonzerot}
2\ii{t}\log|r_1|\in{2\pi\ii}\ZZ\quad\text{and}\quad
2\ii{t}\log|r_2|\in{2\pi\ii}\ZZ.
\end{equation}
But by assumption $\tfrac{\pi}{\log|r_1|}\ZZ\cap\tfrac{\pi}{\log|r_2|}\ZZ=\{0\}$, so \eqref{nonzerot} is impossible for non-zero $t$.
\end{proof}

Next we identify the invariant $S(\Linf(\GG))$ in some special cases.

\begin{theorem}\label{theSthm}
Assume that the sequence $\bigl((\nu_n,q_n)\bigr)_{n\in\NN}$ is such that each value $(\nu_n,q_n)$ appears in the sequence infinitely many times. Then $\Linf(\GG)_\bh=\Linf(\GG)^\sigma$ is a factor and
\begin{equation}\label{theS}
S\bigl(\Linf(\GG)\bigr)=\overline{\Bigl\{\lambda_1\dotsm\lambda_N\,\Bigr|\Bigl.\,N\in\NN,\:\lambda_i\in\bigl(\{0\}\cup|q_i|^{2\ZZ}\bigr),\:i\in\{1,\dotsc,N\}\Bigr\}}.
\end{equation}
\end{theorem}

\begin{proof}
From Theorem \ref{thmT} we know that $\Linf(\GG)$ is a factor and hence, by Theorem \ref{theCenters}\eqref{theCenters1} the centralizer $\Linf(\GG)_\bh=\Linf(\GG)^\sigma$ of $\bh$ is a factor as well. As we mentioned at the end of Section \ref{ConnesTandS}, this implies that $S(\Linf(\GG))=\operatorname{Sp}{\modOp[\bh]}$ which by Proposition \ref{propPsNabla} is
\[
S\bigl(\Linf(\GG)\bigr)
=\operatorname{Sp}{\modOp[\bh]}
=\overline{\bigl\{\lambda_1\dotsm\lambda_N\,\bigr|\bigl.\,N\in\NN,\:\lambda_i\in\operatorname{Sp}(\modOp[\bh_{\nu_i,q_i}]),\:i\in\{1,\dotsc,N\}\bigr\}}.
\]
In view of \eqref{SpUgammas}, this proves \eqref{theS}.
\end{proof}

\begin{corollary}\label{corS}
\noindent
\begin{enumerate}
\item\label{corS1} If $(\nu_n,q_n)=(\nu,q)$ for all $n$ and some $(\nu,q)\in(\RR\setminus\{0\})\times(\left]-1,1\right[\setminus\{0\})$ then $S(\Linf(\GG))=\{0\}\cup|q|^{2\ZZ}$. In particular, $\Linf(\GG)$ is the injective factor of type $\mathrm{III}_{|q|^2}$.
\item\label{corS2} Assume that there are two subsequences $(q_{n_{1,p}})_{p\in\NN}$ and $(q_{n_{2,p}})_{p\in\NN}$ such that
\[
\bigl\{n_{1,p}\,\bigr|\bigl.\,p\in\NN\bigr\}\cap\bigl\{n_{2,p}\,\bigr|\bigl.\,p\in\NN\bigr\}=\emptyset
\]
and
\[
q_{n_{1,p}}\xrightarrow[p\to\infty]{}r_1,\quad
q_{n_{2,p}}\xrightarrow[p\to\infty]{}r_2
\]
for some $r_1,r_2\in\left]-1,1\right[\setminus\{0\}$ such that $\tfrac{\pi}{\log|r_1|}\ZZ\cap\tfrac{\pi}{\log|r_2|}\ZZ=\{0\}$. Then
$S(\Linf(\GG))=\RR_{\geq{0}}$. In particular, $\Linf(\GG)$ is the injective factor of type $\mathrm{III}_1$.
\end{enumerate}
\end{corollary}

\begin{proof}
Statement \eqref{corS1} is clear from Theorem \ref{theSthm} and the fact that the $S$-invariant of a type $\mathrm{III}_\lambda$ factor is $\{0\}\cup\lambda^\ZZ$.

Ad \eqref{corS2}. By Theorem \ref{theSthm} we have
\[
S\bigl(\Linf(\GG)\bigr)=\overline{\Bigl\{\lambda_1\dotsm\lambda_N\,\Bigr|\Bigl.\,N\in\NN,\:\lambda_i\in\bigl(\{0\}\cup|q_i|^{2\ZZ}\bigr),\:i\in\{1,\dotsc,N\}\Bigr\}}.
\]
In particular, $S(\Linf(\GG))\setminus\{0\}$ is a closed multiplicative subgroup of
$\RR_{>0}$. By assumption it contains both the subgroups $|r_1|^{2\ZZ}$ and $|r_2|^{2\ZZ}$ whose intersection is $\{1\}$, because $|r_1|^{2k}=|r_2|^{2l}$ for some $k,l\in\ZZ\setminus\{0\}$ implies $2k\log|r_1|=2l\log|r_2|$ and hence $\tfrac{\pi}{\log|r_1|}\ZZ\cap\tfrac{\pi}{\log|r_2|}\ZZ$ contains the non-zero number $\tfrac{2\pi{l}}{\log|r_1|}=\tfrac{2\pi{k}}{\log|r_2|}$ contrary to our assumption.

Subgroups of
$\RR_{>0}$ are either dense or discrete and the discrete ones are of the form $r^\ZZ$ for some $r>0$. If $S(\Linf(\GG))\setminus\{0\}$ were discrete we would have $|r_1|=r^k$ and $|r_2|=r^l$ for some $k,l\in\ZZ$. But then $|r_1|^l$ would be equal to $|r_2|^k$ and hence $|r_1|^\ZZ\cap|r_2|^\ZZ$ would not be trivial. It follows that $S(\Linf(\GG))\setminus\{0\}$ is dense in
$\RR_{>0}$, and since it is closed, we obtain
$S(\Linf(\GG))=\RR_{\geq{0}}$.
\end{proof}

\subsection{Type \texorpdfstring{$\mathrm{III}_0$}{III0}}\label{sectIII0}\hspace*{\fill}

In this section we will construct the sequence $\bigl((\nu_n,q_n)\bigr)_{n\in\NN}$ in such a way that $\Linf(\GG)$ with $\GG=\bigtimes\limits_{n=1}^{\infty}\HH_{\nu_n,q_n}$ is an injective factor of type $\mathrm{III}_0$.

For $s\in\left]0,1\right[$ let
\[
t_s=\sum_{p=1}^{\infty}\tfrac{\left\lfloor{p^{1-s}}\right\rfloor}{p!},
\]
where $\lfloor{x}\rfloor$ denotes the integer part of $x\in\RR_{\geq{0}}$. Note that the mapping
$\left]0,1\right[\ni{s}\mapsto{t_s}\in\RR_{>0}$ is injective.

\begin{theorem}\label{thmTypeIII0}
For each $s\in\left]0,1\right[$ there exists a compact quantum group $\GG_s$ such that $\Linf(\GG_s)$ is an injective factor and its invariant $T$ satisfies
\begin{enumerate}
\item\label{QinT} $\QQ\subset{T(\Linf(\GG_s))}$,
\item\label{sprimT} $t_{s'}\in{T(\Linf(\GG_s))}$ if and only if $s'>s$.
\end{enumerate}
In particular, $\bigl\{\Linf(\GG_s)\bigr\}_{s\in\left]0,1\right[}$ is an uncountable family of pairwise non-isomorphic injective type $\mathrm{III}_0$ factors.
\end{theorem}

For the proof of Theorem \ref{thmTypeIII0} we need the next lemma in whose statement we write $\{x\}$ for the fractional part $x-\lfloor{x}\rfloor$ of $x\in\RR_{\geq{0}}$.

\begin{lemma}\label{weNeed}
For any $s\in\left]0,1\right[$ we have $\{t_sk!\}-\tfrac{1}{k^s}=\cO\bigl(\tfrac{1}{k}\bigr)$ as $k\to\infty$.
\end{lemma}

\begin{proof}
Observe that
\[
t_sk!=\sum_{p=1}^{k}\tfrac{k!}{p!}\bigl\lfloor{p^{1-s}}\bigr\rfloor+\tfrac{1}{k+1}\bigl\lfloor{(k+1)^{1-s}}\bigr\rfloor+\sum_{p=k+2}^{\infty}\tfrac{1}{p(p-1)\dotsm(k+1)}\bigl\lfloor{p^{1-s}}\bigr\rfloor
\]
The first term on the right-hand side is an integer which we can disregard when computing $\{t_sk!\}$. Moreover the last term can be estimated as
\begin{align*}
\tfrac{1}{(k+1)^s}\sum_{p=k+2}^{\infty}\tfrac{(k+1)^s\left\lfloor{p^{1-s}}\right\rfloor}{p(p-1)\dotsm(k+1)}
&<\tfrac{1}{(k+1)^s}\sum_{p=k+2}^{\infty}\tfrac{p^s\left\lfloor{p^{1-s}}\right\rfloor}{p(p-1)\dotsm(k+1)}\\
&\leq\tfrac{1}{(k+1)^s}\sum_{p=k+2}^{\infty}\tfrac{1}{(p-1)\dotsm(k+1)}\\
&\leq\tfrac{1}{(k+1)^s}\sum_{p=k+2}^{\infty}\tfrac{1}{(k+1)^{p-k-1}}=\tfrac{1}{(k+1)^s}\tfrac{1}{k}
\end{align*}
which is arbitrarily small for sufficiently large $k$. Since so is $\tfrac{1}{k+1}\bigl\lfloor{(k+1)^{1-s}}\bigr\rfloor$, we find that for sufficiently large $k$
\[
\{t_sk!\}=\tfrac{1}{k+1}\bigl\lfloor{(k+1)^{1-s}}\bigr\rfloor+\sum_{p=k+2}^{\infty}\tfrac{1}{p(p-1)\dotsm(k+1)}\bigl\lfloor{p^{1-s}}\bigr\rfloor
\]

Hence
\[
\{t_sk!\}-\tfrac{1}{k^s}=\Bigl(\tfrac{1}{k+1}\bigl\lfloor{(k+1)^{1-s}}\bigr\rfloor-\tfrac{1}{k^s}\Bigr)+\sum_{p=k+2}^{\infty}\tfrac{1}{p(p-1)\dotsm(k+1)}\bigl\lfloor{p^{1-s}}\bigr\rfloor
\]
and we already know that the second term is $\cO\bigl(\tfrac{1}{k^{s+1}}\bigr)$. As for the first term, it is clearly negative ($\tfrac{1}{k+1}\bigl\lfloor{(k+1)^{1-s}}\bigr\rfloor\leq\tfrac{1}{(k+1)^s}\leq\tfrac{1}{k^s}$), so
\begin{align*}
\Bigl|\tfrac{1}{k+1}\bigl\lfloor{(k+1)^{1-s}}\bigr\rfloor-\tfrac{1}{k^s}\Bigr|
&=\tfrac{1}{k^s}-\tfrac{1}{k+1}\bigl\lfloor{(k+1)^{1-s}}\bigr\rfloor\\
&\leq\tfrac{1}{k^s}-\tfrac{1}{k+1}\bigl((k+1)^{1-s}-1\bigr)\\
&=\tfrac{1}{k^s}-\tfrac{1}{(k+1)^s}+\tfrac{1}{k+1}\\
&=\tfrac{1}{k(k+1)^s}\Bigl(k\tfrac{(k+1)^s-k^s}{k^s}\Bigr)+\tfrac{1}{k+1}\\
&=\tfrac{1}{k(k+1)^s}k\Bigl(\bigl(1+\tfrac{1}{k}\bigr)^s-1\Bigr)+\tfrac{1}{k+1}\\
&=\tfrac{s}{k(k+1)^s}+\tfrac{1}{k+1}=\cO\bigl(\tfrac{1}{k}\bigr)
\end{align*}
by the elementary inequality $\bigl(1+\tfrac{1}{k}\bigr)^s\leq{1}+\tfrac{s}{k}$.
\end{proof}

\begin{proof}[Proof of Theorem \ref{thmTypeIII0}]
Let us fix $s\in\left]0,1\right[$ and define
\[
l_k=\bigl\lfloor\exp(2\pi{k!})k^{2s-1}\bigr\rfloor,\qquad{k}\in\NN.
\]
Next let $(q_n)_{n\in\NN}$ be the sequence
\[
\bigl(
\underbrace{\exp(-\pi{1!}),\dotsc,\exp(-\pi{1!})}_{l_1\text{ times}},
\underbrace{\exp(-\pi{2!}),\dotsc,\exp(-\pi{2!})}_{l_2\text{ times}},\dotsc
\bigr)
\]
and for each $n$ choose $\nu_n\in\RR\setminus\{0\}$ such that $\nu_n\log|q_n|\not\in\QQ$. Define also $k(n)$ to be the positive integer such that
\[
q_n=\exp\bigl(-\pi{k(n)!}\bigr),\qquad{n}\in\NN.
\]
Clearly $n\mapsto{k(n)}$ is an unbounded non-decreasing function.

The compact quantum group $\GG_s$ is defined as
\[
\GG_s=\bigtimes_{n=1}^{\infty}\HH_{\nu_n,q_n}.
\]
Then $\Linf(\GG_s)$ is an infinite tensor product of injective factors (of type $\mathrm{II}_\infty$) and hence it is an injective factor (Proposition \ref{injTP} and remarks before Proposition \ref{typeIIIfactor}). Next we will study the invariant $T(\Linf(\GG_s))$.

Take $r\in\QQ$. Then for $n$ large enough the number
\[
2\ii{r}\log|q_n|=-2\pi\ii{r}k(n)!
\]
belongs to $2\pi\ii\ZZ$ and hence, for such $n$,
\[
1-\tfrac{1-q_n^2}{\left|1-|q_n|^{2+2\ii{r}}\right|}=1-\tfrac{1-q_n^2}{\left|1-|q_n|^2\right|}=0.
\]
Consequently $r\in{T(\Linf(\GG_s))}$ by a reasoning analogous to the proof of Theorem \ref{thmT}. This proves statement \eqref{QinT}.

We now pass to the proof of statement \eqref{sprimT}. Let $t=t_{s'}$ for some $s'\in\left]0,1\right[$. Our aim is to show that $t\in{T(\Linf(\GG_s))}$ if and only if $s'>s$. Denote
\[
S_t=\sum_{n=1}^\infty\Bigl(1-\tfrac{1-q_n^2}{\left|1-|q_n|^{2+2\ii{r}}\right|}\Bigr)
=\sum_{k=1}^{\infty}l_k\Bigl(1-\tfrac{1-\exp(-2\pi{k!})}{\left|1-\exp(-(2+2\ii{t})\pi{k!})\right|}\Bigr).
\]

Fix $k\in\NN$. Writing $x_k$ for $\bigl|1-\exp\bigl(-(2+2\ii{t})\pi{k!}\bigr)\bigr|$ we have
\begin{align*}
1-\tfrac{1-\exp(-2\pi{k!})}{\left|1-\exp(-(2+2\ii{t})\pi{k!})\right|}&=1-\tfrac{1-\exp(-2\pi{k!})}{x_k}\\
&=\tfrac{x_k-\left(1-\exp(-2\pi{k!})\right)}{x_k}\\
&=\tfrac{x_k^2-\left(1-\exp(-2\pi{k!})\right)^2}{x_k\left(x_k+\left(1-\exp(-2\pi{k!})\right)\right)}\\
&=\tfrac{2}{x_k\left(x_k+\left(1-\exp(-2\pi{k!})\right)\right)}\exp(-2\pi{k!})\bigl(1-\cos(2\pi{t}k!)\bigr).
\end{align*}
Since $x_k\xrightarrow[k\to\infty]{}1$, the series $S_t$ converges if and only if
\[
S_t'=\sum_{k=1}^{\infty}l_k\exp(-2\pi{k!})\bigl(1-\cos(2\pi{t}k!)\bigr)
\]
converges. To analyze this series we note the obvious fact that
\[
\tfrac{l_k\exp(-2\pi{k!})}{k^{2s-1}}\xrightarrow[k\to\infty]{}1
\]
and study the factor $\bigl(1-\cos(2\pi{t}k!)\bigr)$:
\[
1-\cos(2\pi{t}k!)=1-\cos\bigl(2\pi\{tk!\}\bigr)
=1-\cos\Bigl(\tfrac{2\pi}{k^{s'}}+2\pi\bigl(\{tk!\}-\tfrac{1}{k^{s'}}\bigr)\Bigr)
=1-\cos\Bigl(\tfrac{2\pi}{k^{s'}}+\cO\bigl(\tfrac{1}{k}\bigr)\Bigr)
\]
by Lemma \ref{weNeed}. Using this we immediately arrive at
\[
\lim_{k\to\infty}\tfrac{1-\cos(2\pi{t}k!)}{\frac{2\pi^2}{k^{2s'}}}=1
\]
which shows that $S_t'$ converges if and only if
\[
S_t''=\sum_{k=1}^{\infty}k^{2s-1}\tfrac{2\pi^2}{k^{2s'}}=2\pi^2\sum_{k=1}^{\infty}\tfrac{k^{2(s-s')}}{k}
\]
converges which happens if and only if $s'>s$.

Since $T(\Linf(\GG_s))$ is not equal to $\RR$, the injective factor $\Linf(\GG_s)$ is of type $\mathrm{III}$. Furthermore, since $\QQ\subset{T(\Linf(\GG_s))}$ it is not of type $\mathrm{III}_\lambda$ for any $\lambda>0$ and consequently it is of type $\mathrm{III}_0$. The fact that $\bigl\{\Linf(\GG_s)\bigr\}_{s\in\left]0,1\right[}$ are pairwise non-isomorphic is clear.
\end{proof}

\begin{remark}\label{remninj0}
Let $\{\GG_s\}_{s\in\left]0,1\right[}$ be the family of compact quantum groups constructed in Theorem \ref{thmTypeIII0}. Consider compact quantum groups $\GG_s\times\hh{\mathbb{F}_2}$ where $\mathbb{F}_2$ is the free group on two generators. Then $\bigl\{\GG_{s}\times\hh{\mathbb{F}_2}\bigr\}_{s\in\left]0,1\right[}$ is an uncountable family of second countable compact quantum groups such that $\bigl\{\Linf(\GG_s\times\hh{\mathbb{F}_2})\bigr\}_{s\in\left]0,1\right[}$ are pairwise non-isomorphic, non-injective type $\mathrm{III}_0$ factors. Indeed, $\Linf(\GG_s\times\hh{\mathbb{F}_2})=\Linf(\GG_s)\vtens\Linf(\hh{\mathbb{F}_2})$ is non-injective, since $\Linf(\hh{\mathbb{F}_2})$ is non-injective, while the Connes invariant $T$ of $\Linf(\GG_s\times\hh{\mathbb{F}_2})$ is equal to $T(\Linf(\GG_s))$ (\cite[Proposition 27.3]{Stratila}).
\end{remark}

\section{Invariants \texorpdfstring{$\Ttau(\GG)$}{TtG}, \texorpdfstring{$\TtauInn(\GG)$}{TtinnG} and \texorpdfstring{$\TtauAInn(\GG)$}{TtinnG}}\label{sectTtau}

Let $\rM$ be a von Neumann algebra. The natural topology on the group $\operatorname{Aut}(\rM)$ is the restriction to $\operatorname{Aut}(\rM)$ of the topology on the vector space $\B_{\sigma\text{-w}}(\rM)$ of all $\sigma$-weakly continuous linear maps $\rM\to\rM$ defined by the seminorms
\[
\B_{\sigma\text{-w}}(\rM)\ni\Theta\longmapsto\|\theta\comp\Theta\|\in\RR_{\geq{0}},\qquad\theta\in\rM_*.
\]
(often called the \emph{$u$-topology}, cf.~\cite{HaagerupStandardForm,Stratila}). It is known that $\operatorname{Aut}(\rM)$ is a topological group with the $u$-topology (see \cite[Section I.2]{Stratila}). We will denote the closure of the subgroup $\operatorname{Inn}(\rM)\subset\operatorname{Aut}(\rM)$ in $u$-topology by $\overline{\operatorname{Inn}}(\rM)$ and its elements will be referred to as \emph{approximately inner} automorphisms of $\rM$.

\begin{definition}\label{def:invariants}
For a locally compact quantum group $\GG$ define
\begin{align*}
\Ttau(\GG)&=\bigl\{t\in\RR\,\bigr|\bigl.\,\tau^\GG_t=\id\bigr\},\\
\TtauInn(\GG)&=\bigl\{t\in\RR\,\bigr|\bigl.\,\tau^\GG_t\in\operatorname{Inn}(\Linf(\GG))\bigr\},\\
\TtauAInn(\GG)&=\bigl\{t\in\RR\,\bigr|\bigl.\,\tau^\GG_t\in\overline{\operatorname{Inn}}(\Linf(\GG))\bigr\}.
\end{align*}
\end{definition}

\begin{remark}\label{remark1}
\noindent
\begin{enumerate}
\item If $\GG$ and $\HH$ are isomorphic locally compact quantum groups then $\Ttau(\GG)=\Ttau(\HH)$ and similarly for $\TtauInn$, $\TtauAInn$.
\item Obviously $\Ttau(\GG)\subset{\TtauInn(\GG)}\subset\TtauAInn(\GG)$ and each invariant is a subgroup of $\RR$.
\item For any locally compact quantum group $\GG$ we have $\Ttau(\GG)=\Ttau(\hh{\GG})$. This follows easily from the formula $\bigl(\tau^{\hh{\GG}}_t\tens\tau^{\GG}_t\bigr)W=W$ for all $t\in\RR$, where $W$ is the Kac-Takesaki operator (\cite[Formula $(5.28)$]{SoltanWoronowicz}).
\item If $\GG$ is compact then $\GG$ is of Kac type if and only of $\Ttau(\GG)=\RR$ (cf.~\cite[Section 1.7]{NeshveyevTuset}).
\item An invariant similar to $\TtauInn(\GG)$ was studied in \cite[Definition 3.4]{Vaes}, namely $T(\Linf(\GG),\Delta)=\bigl\{t\in\RR\,\bigr|\bigl.\,\exists\:u\in\Linf(\GG),\:\tau^{\GG}_t=\operatorname{Ad}(u)\text{ and }\Delta(u)=u\tens{u}\bigr\}$. Obviously this set is contained in $\TtauInn(\GG)$, but it can be shown that the two invariants are different. More precisely, for $q\in\left]0,1\right[$ the scaling automorphism $\tau^{\SU_q(3)}_{t}\in\operatorname{Aut}(\Linf(\SU_q(3)))$ with $t=\frac{\pi}{2\log{q}}$ is inner, but not implemented by any group-like unitary. This fact relies on the study of the Plancherel objects of $\widehat{\SU_q(3)}$, and a formula expressesing scaling automorphisms of $\SU_q(3)$ at the level of direct integrals. For more details, see \cite[Section 4]{invariants}.
\item In this paper we use invariants introduced in Definition \ref{def:invariants} (in particular $\TtauInn(\GG)$) to distinguish between different quantum groups with the same von Neumann algebra $\Linf(\GG)$. Alternatively, as observed by the referee, rather than using $\TtauInn(\GG)$, we could use invariant $T(\Linf(\GG),\Delta)$ of Vaes to achieve this goal (see remarks \ref{remark2}, \ref{remark3}). Its usage would shorten some of the arguments, however we prefer to work with invariant $\TtauInn(\GG)$. We believe that it is more natural and in itself carries interesting information about quantum groups we construct.
\item One can find compact quantum groups $\GG,\HH$ which share the same $\Ttau$ and $\TtauInn$ invariants, while $\TtauAInn(\GG)\neq\TtauAInn(\HH)$ (see Example \ref{example1}). This shows that all the three invariants of Definition \ref{def:invariants} can be useful in telling quantum groups apart.
\end{enumerate}
\end{remark}

\begin{proposition}\label{TtauProp}
Let $\GG$ be a locally compact quantum group. Then $\Ttau(\GG)$ and $\TtauAInn(\GG)$ are closed subgroups of $\RR$.
\end{proposition}

\begin{proof}
Clearly all the subsets in question are subgroups and $\Ttau(\GG)$ is obviously closed. Moreover, since $\Linf(\GG)$ is in standard form on $\Ltwo(\GG)$ and the scaling group is implemented by a strongly continuous one parameter group of unitary operators on $\Ltwo(\GG)$, by the theorem of Haagerup \cite[Proposition 3.5]{HaagerupStandardForm} the map $\RR\ni{t}\mapsto\tau^\GG_t\in\operatorname{Aut}(\Linf(\GG))$ is continuous. This immediately implies that $\TtauAInn(\GG)$ is also closed.
\end{proof}

\begin{examples}\label{exTtau}
\noindent
\begin{enumerate}
\item Let $q\in\left]-1,1\right[\setminus\{0\}$. Then $\Ttau\bigl(\operatorname{SU}_q(2)\bigr)=\TtauInn\bigl(\operatorname{SU}_q(2)\bigr)=\TtauAInn\bigl(\operatorname{SU}_q(2)\bigr)=\tfrac{\pi}{\log|q|}\ZZ$ (\cite[Proposition 7.3]{modular}).
\item\label{exTtau2} Let $q\in\left]-1,1\right[\setminus\{0\}$, $\nu\in\RR\setminus\{0\}$ and let $\HH_{\nu,q}=\QQ\bowtie\operatorname{SU}_q(2)$ as in Section \ref{TheExamples}. Then by \cite[Proposition 4.8]{KrajczokWasilewski} we have $\Ttau(\HH_{\nu,q})=\tfrac{\pi}{\log|q|}\ZZ$ and $\TtauInn(\HH_{\nu,q})=\nu\QQ+\tfrac{\pi}{\log|q|}\ZZ$. In particular, $\TtauInn(\HH_{\nu,q})$ is not closed and $\TtauAInn(\HH_{\nu,q})=\RR$ by Proposition \ref{TtauProp}.
\end{enumerate}
\end{examples}

\begin{lemma}
\noindent
\begin{enumerate}
\item\label{tauProd1} Let $\GG$ and $\HH$ be locally compact quantum groups. Then the scaling group $\bigl(\tau^{\GG\times\HH}_t\bigr)_{t\in\RR}$ is the product of the scaling groups $\bigl(\tau^\GG_t\bigr)_{t\in\RR}$ and $\bigl(\tau^\HH_t\bigr)_{t\in\RR}$ of $\GG$ and $\HH$ in the sense that for each $t\in\RR$ we have $\tau^{\GG\times\HH}_t=\tau^\GG_t\tens\tau^\HH_t$.
\item\label{tauProd2} Let $\{\GG_n\}_{n\in\NN}$ be a family of compact quantum groups and let $\GG=\bigtimes\limits_{n=1}^\infty\GG_n$. Then for each $t\in\RR$ we have
\[
\tau^\GG_t=\bigotimes_{n=1}^{\infty}\tau^{\GG_n}_t,
\]
where $\bigl(\tau^{\GG_n}_t\bigr)_{t\in\RR}$ is the scaling group of $\GG_n$.
\end{enumerate}
\end{lemma}

\begin{proof}
Statement \eqref{tauProd1} is well known (see e.g.~\cite[Section 13.5]{KrajczokTypeI}), while \eqref{tauProd2} follows from the fact that $\bh=\bigotimes\limits_{n=1}^{\infty}\bh_n$ and that the irreducible representations of $\GG$ are given by \eqref{IrrProdG}.
\end{proof}

\begin{proposition}\label{propTprod}
\noindent
\begin{enumerate}
\item\label{propTprod1} For any locally compact quantum groups $\GG$ and $\HH$ we have
\begin{align*}
\Ttau(\GG\times\HH)&=\Ttau(\GG)\cap\Ttau(\HH),\\
\TtauInn(\GG\times\HH)&=\TtauInn(\GG)\cap\TtauInn(\HH),\\
\TtauAInn(\GG\times\HH)&\supset\TtauAInn(\GG)\cap\TtauAInn(\HH).
\end{align*}
\item\label{propTprod2} If $\{\GG_n\}_{n\in\NN}$ is a family of compact quantum groups and $\GG=\bigtimes\limits_{n=1}^\infty\GG_n$ then
\[
\Ttau(\GG)=\bigcap_{n=1}^{\infty}\Ttau(\GG_n).
\]
\end{enumerate}
\end{proposition}

\begin{proof}
The only non-trivial assertion is that $\TtauInn(\GG\times\HH)\subset\TtauInn(\GG)\cap\TtauInn(\HH)$ which follows immediately from \cite[Proposition 17.6]{Stratila}.
\end{proof}

\begin{remark}
In the situation of Proposition \ref{propTprod}\eqref{propTprod1} it is not clear if $\TtauAInn(\GG\times\HH)\subset\TtauAInn(\GG)\cap\TtauAInn(\HH)$. Furthermore if $\{\GG_n\}_{n\in\NN}$ is as in Proposition \ref{propTprod}\eqref{propTprod2} then $\tau^\GG_t$ might not be inner even when all $\tau^{\GG_n}_t$ are inner. The general result on when $\tau^\GG_t$ is inner is \cite[Theorem XIV.1.13]{Takesaki3}.
\end{remark}

We will now continue our study of infinite products of quantum groups $\HH_{\nu,q}$ as in Section \ref{TheExamples}. Thus from now on $\bigl((\nu_n,q_n)\bigr)_{n\in\NN}$ is a sequence of elements of $(\RR\setminus\{0\})\times(\left]-1,1\right[\setminus\{0\})$ such that $\nu_n\log|q_n|\not\in\pi\QQ$ for all $n$ and $\HH_{\nu_n,q_n}$ is the bicrossed product compact quantum group $\QQ\bowtie\operatorname{SU}_{q_n}(2)$ described in Section \ref{skladniki}. Finally we put $\GG=\bigtimes\limits_{n=1}^{\infty}\HH_{\nu_n,q_n}$.

\begin{theorem}\label{TauGthm}
We have
\begin{subequations}
\begin{align}
\Ttau(\GG)&=\bigcap_{n=1}^{\infty}\tfrac{\pi}{\log|q_n|}\ZZ,\label{TtauG}\\
\TtauInn(\GG)&=\biggl\{t\in\bigcap_{n=1}^{\infty}\Bigl(\nu_n\QQ+\tfrac{\pi}{\log|q_n|}\ZZ\Bigr)\,\biggr|\biggl.\,t\in\tfrac{\pi}{\log|q_n|}\ZZ\text{ for all but finitely many }n\in\NN\biggr\},\label{TtauInnG}\\
\TtauAInn(\GG)&=\RR.\label{TtauAInnG}
\end{align}
\end{subequations}
\end{theorem}

\begin{proof}
To lighten the notation let us write $\HH_n$ for $\HH_{\nu_n,q_n}$. We already know that $\Ttau(\HH_n)=\tfrac{\pi}{\log|q_n|}\ZZ$ and $\TtauInn(\HH_n)=\nu_n\QQ+\tfrac{\pi}{\log|q_n|}\ZZ$ (Example \ref{exTtau}\eqref{exTtau2}). This together with Proposition \ref{propTprod}\eqref{propTprod2} proves \eqref{TtauG}.

Let $t\in\bigcap\limits_{n=1}^{\infty}\bigl(\nu_n\QQ+\tfrac{\pi}{\log|q_n|}\ZZ\bigr)$ and assume that $t$ belongs to $\tfrac{\pi}{\log|q_n|}\ZZ$ for all $n\in\NN\setminus\{n_1,\dotsc,n_N\}$. It follows that for each $k\in\{1,\dotsc,N\}$ the scaling automorphism $\tau^{\HH_{n_k}}_t$ is implemented by some unitary $u_{n_k}\in\Linf(\HH_{n_k})$ and $\tau^{\HH_n}_t=\id$ for $n\in\NN\setminus\{n_1,\dotsc,n_N\}$. Setting $u_n=\I$ we obtain $\tau^{\HH_n}_t=\operatorname{Ad}(u_n)$ for all $n$ and since only finitely many of them are non-trivial we immediately see that
\[
\tau^{\GG}_t=\bigotimes_{n=1}^{\infty}\tau^{\HH_n}_t\in\operatorname{Inn}\bigl(\Linf(\GG)\bigr).
\]

Now let $t\in\TtauInn(\GG)$, so that the automorphism $\tau^\GG_t=\bigotimes\limits_{n=1}^{\infty}\tau^{\HH_n}_t$ is inner. By \cite[Theorem XIV.1.13]{Takesaki3} each $\tau^{\HH_n}_t$ is inner, and hence $t\in\bigcap\limits_{n=1}^{\infty}\bigl(\nu_n\QQ+\tfrac{\pi}{\log|q_n|}\ZZ\bigr)$. Furthermore, the same result says also that if $\tau^{\HH_n}_t=\operatorname{Ad}(u_n)$ then
\begin{equation}\label{sumhnun}
\sum_{n=1}^\infty\Bigl(1-\bigl|\bh_n(u_n)\bigr|\Bigr)<+\infty.
\end{equation}
As the implementing unitaries $\{u_n\}_{n\in\NN}$ are unique up to a constant in $\TT$ (since $\Linf(\HH_n)$ is a factor), the convergence of the series is not affected if we choose $u_n$ to be the operators considered in \cite[Proof of Proposition 4.8]{KrajczokWasilewski}, i.e.~the canonical unitaries in $\Linf(\HH_n)=\QQ\ltimes_{\alpha^{\nu_n}}\Linf(\operatorname{SU}_q(2))$ implementing the action. These unitaries are characters of irreducible representations, so for each $n$ we have two possibilities:
\begin{itemize}
\item $t\in\tfrac{\pi}{\log|q_n|}\ZZ$ and then $\tau^{\HH_n}_t=\id$,
\item $t\not\in\tfrac{\pi}{\log|q_n|}\ZZ$ in which case $u_n$ is non-trivial and consequently $\bh_n(u_n)=0$.
\end{itemize}
The fact that the series \eqref{sumhnun} converges implies that the latter possibility arises for only finitely many $n$ which proves \eqref{TtauInnG}.

Now let us show \eqref{TtauAInnG}. Take $t\in\RR$ and $\omega\in\Lone(\GG)$. Since $\Linf(\GG)$ is in standard form on $\Ltwo(\GG)=\bigotimes\limits_{n=1}^{\infty}(\Ltwo(\HH_n),\Omega_n)$, functionals of the form
\[
\omega=\omega_1\tens\dotsm\tens\omega_N\tens\bh_{N+1}\tens\bh_{N+2}\tens\dotsm
\]
with $\omega_i\in\Lone(\HH_i)$ ($i\in\{1,\dotsc,N\}$) form a linearly dense subset of $\Lone(\GG)$. Furthermore, since for each $n$ the automorphism $\tau^{\HH_n}_t$ is approximately inner (Example \ref{exTtau}\eqref{exTtau2}), given $\omega$ as above and $\eps>0$ we can find a unitaries $\{u_{i,\eps}\}_{i\in\{1,\dotsc,N\}}$ with $u_{i,\eps}\in\Linf(\HH_i)$ such that
\[
\bigl\|\omega_i\comp\tau^{\HH_i}_t-\omega_i\comp\operatorname{Ad}(u_{i,\eps})\bigr\|\leq\tfrac{\eps}{N},\qquad{i}\in\{1,\dotsc,N\}.
\]
Consequently
\[
\Bigl\|
\omega\comp\tau^\GG_t-\omega\comp\operatorname{Ad}\bigl(u_{1,\eps}\tens\dotsm\tens{u_{N,\eps}}\tens\I^{\tens\infty}\bigr)
\Bigr\|
=\biggl\|\bigotimes_{i=1}^N\bigl(\omega_i\comp\tau^{\HH_i}_t\bigr)-\bigotimes_{i=1}^N\bigl(\omega_i\comp\operatorname{Ad}(u_{i,\eps})\bigr)\biggr\|\leq{N}\tfrac{\eps}{N}=\eps.
\]
\end{proof}

\begin{examples}\label{exTtauG}
\noindent
\begin{enumerate}
\item\label{exTtauG1} Consider the sequence $\bigl((\nu_n,q_n)\bigr)_{n\in\NN}$ such that for all $n$ we have $(\nu_n,q_n)=(\nu,q)$ for some $\nu\in\RR\setminus\{0\}$ and $q\in\left]-1,1\right[\setminus\{0\}$ (so that $\Linf(\GG)$ is the injective factor of type $\mathrm{III}_{|q|^2}$, cf.~Corollary \ref{corS}\eqref{corS1}). Then Theorem \ref{TauGthm} yields $\Ttau(\GG)=\TtauInn(\GG)=\tfrac{\pi}{\log|q|}\ZZ$ and $\TtauAInn{(\GG)}=\RR$. Note, in particular, that $\TtauAInn{(\GG)}$ is not the closure of $\TtauInn{(\GG)}$.
\item Now consider $\bigl((\nu_n,q_n)\bigr)_{n\in\NN}$ as in Corollary \ref{corS}\eqref{corS2} (so that $\Linf(\GG)$ is the injective factor of type $\mathrm{III}_1$). Then $\Ttau(\GG)=\TtauInn(\GG)=\{0\}$ and $\TtauAInn(\GG)=\RR$. In fact, the last property is automatic by \cite[Theorem 1]{KawahigashiSutherlandTakesaki} which, in particular, says that any automorphism of the injective factor of type $\mathrm{III}_1$ with separable predual is approximately inner (see also \cite{Marrakchi}).
\item\label{exTtauG3} Let $\bigl((\nu_n,q_n)\bigr)_{n\in\NN}$ be defined as in Section \ref{sectIII0}, i.e.~$(q_n)_{n\in\NN}$ is the sequence
\[
\bigl(
\underbrace{\exp(-\pi{1!}),\dotsc,\exp(-\pi{1!})}_{l_1\text{ times}},
\underbrace{\exp(-\pi{2!}),\dotsc,\exp(-\pi{2!})}_{l_2\text{ times}},\dotsc
\bigr),
\]
where
\[
l_k=\bigl\lfloor\exp(2\pi{k!})k^{2s-1}\bigr\rfloor,\qquad{k}\in\NN
\]
while $(\nu_n)_{n\in\NN}$ are chosen so that $\nu_n\log|q_n|\not\in\pi\QQ$ for all $n$. By Theorem \ref{thmTypeIII0} the algebra $\Linf(\GG)$ is then a factor of type $\mathrm{III}_0$. Moreover, it is immediate that
\[
\Ttau(\GG)=\bigcap_{n=1}^{\infty}\tfrac{1}{n!}\ZZ=\ZZ\quad\text{and}\quad\TtauAInn(\GG)=\RR.
\]
Again by Theorem \ref{TauGthm}
\[
\TtauInn(\GG)=\biggl\{
t\in\bigcap_{n=1}^{\infty}\Bigl(\nu_n\QQ+\tfrac{1}{k(n)!}\ZZ\Bigr)\,\biggr|\biggl.\,t\in\tfrac{1}{k(n)!}\ZZ\text{ for all but finitely many }n\in\NN
\biggr\},
\]
where $k$ is the function defined by
\[
q_n=\exp\bigl(-\pi{k(n)!}\bigr),\qquad{n}\in\NN.
\]
We claim that in fact $\TtauInn(\GG)=\ZZ$.

To see this note first that $\ZZ\subset\TtauInn(\GG)$. Furthermore, for any $n\in\NN$ we have $\nu_n\not\in\QQ$, since $\nu_n(-\pi{k(n)}!)\not\in\pi\QQ$. Taking $n=1$ shows that for any $t\in\TtauInn(\GG)$ we have $t\in\nu_1\QQ+\ZZ$, so that $t=\nu_1r+p$ for some $r\in\QQ$ and $p\in\ZZ$. On the other hand $t\in\tfrac{1}{k(N)!}\ZZ$ for some $N\in\NN$, which forces $r=0$ and consequently $t\in\ZZ$.
\end{enumerate}
\end{examples}

\begin{remark}\label{remManyIII0}
If we shift the sequence $\bigl((\nu_n,q_n)\bigr)_{n\in\NN}$ considered in Example \ref{exTtauG}\eqref{exTtauG3} by $p\in\NN$, i.e.~let $(q_n)_{n\in\NN}$ be the sequence
\[
\Bigl(
\underbrace{\exp\bigl(-\pi{(1+p)!}\bigr),\dotsc,\exp\bigl(-\pi{(1+p)!}\bigr)}_{l_1\text{ times}},
\underbrace{\exp\bigl(-\pi{(2+p)!}\bigr),\dotsc,\exp\bigl(-\pi{(2+p)!}\bigr)}_{l_2\text{ times}},\dotsc
\Bigr),
\]
with $(l_k)_{k\in\NN}$ as before, then the invariant $T(\Linf(\GG))$ does not change. In particular, $\Linf(\GG)$ is a factor of type $\mathrm{III}_0$. However the resulting compact quantum group $\GG$ changes, since after the shift $\Ttau(\GG)=\TtauInn(\GG)=\tfrac{1}{p!}\ZZ$. It follows that each of the uncountably many type $\mathrm{III}_0$ factors obtained as $\Linf(\GG)$ in Theorem \ref{thmTypeIII0} arises from at least countably infinite family of pairwise non-isomorphic compact quantum groups.
\end{remark}

Similarly to the case of factors of type $\mathrm{III}_0$ discussed in Remark \ref{remManyIII0} above we can use the invariant $\TtauInn$ to show that each injective factor of type $\mathrm{III}_\lambda$ with $\lambda\in\left]0,1\right[$ arises as $\Linf(\GG)$ for at least countably infinitely many pairwise non-isomorphic compact quantum groups. This is done in Corollary \ref{corManyIIIlambda} below. These results will be considerably strengthened in Section \ref{sectBicrossed}, but we include them here because they can be obtained with much fewer ingredients and fewer technical complications.

\begin{theorem}\label{thmManyIIIlambda}
Fix $\lambda\in\left]0,1\right]$ and let $\GG$ be a compact quantum group with $\Linf(\GG)$ the injective type $\mathrm{III}_\lambda$ factor with separable predual. Furthermore let $\nu\in\RR\setminus\{0\}$ and $q\in\left]-1,1\right[\setminus\{0\}$ be such that $\nu\log|q|\not\in\pi\QQ$. Then $\Linf(\GG\times\HH_{\nu,q})\cong\Linf(\GG)$ and
\begin{align*}
\Ttau(\GG\times\HH_{\nu,q})&=\Ttau(\GG)\cap\tfrac{\pi}{\log|q|}\ZZ,\\
\TtauInn(\GG\times\HH_{\nu,q})&=\TtauInn(\GG)\cap\bigl(\nu\QQ+\tfrac{\pi}{\log|q|}\ZZ\bigr),\\
\TtauAInn(\GG\times\HH_{\nu,q})&\supset\TtauAInn(\GG).
\end{align*}
\end{theorem}

\begin{corollary}\label{corManyIIIlambda}
Let $\lambda\in\left]0,1\right[$, put $q_1=\sqrt{\lambda}$ and $\nu_1=\tfrac{2\pi^2}{\log{\lambda}}$ and define $\GG$ to be the infinite product $\bigtimes\limits_{n=1}^{\infty}\HH_{\nu_1,q_1}$. Furthermore for $k\in\NN$ let $\GG^\lambda_k=\GG\times\HH_{\nu,q}$ with $q=\lambda^{\frac{\pi}{2}}$ and $\nu=\tfrac{2}{\log{\lambda}}(\pi{k}-1)$. Then
\[
\begin{array}{r@{\;}c@{\;}l}
\Ttau(\GG^\lambda_k)&=&\{0\},\\[5pt]
\TtauInn(\GG^\lambda_k)&=&\tfrac{2\pi{k}}{\log{\lambda}}\ZZ,\\[5pt]
\TtauAInn(\GG^\lambda_k)&=&\RR,
\end{array}
\qquad\quad{k}\in\NN
\]
and setting $\GG^\lambda_0=\GG$ we obtain an infinite family of pairwise non-isomorphic compact quantum groups $\bigl\{\GG^\lambda_k\bigr\}_{k\in\ZZ_+}$ with $\Linf(\GG^\lambda_k)$ the injective factor of type $\mathrm{III}_\lambda$ with separable predual.
\end{corollary}

\begin{proof}[Proof of Theorem \ref{thmManyIIIlambda} and Corollary \ref{corManyIIIlambda}]
We first address Theorem \ref{thmManyIIIlambda}. The algebra $\Linf(\HH_{\nu,q})$ is semifinite (it is a factor of type $\mathrm{II}_\infty$). By \cite[Proposition 28.4]{Stratila} we have $S(\Linf(\GG\times\HH_{\nu,q}))=S(\Linf(\GG))$, hence the uniqueness of the injective type $\mathrm{III}_\lambda$ factor with separable predual implies $\Linf(\GG\times\HH_{\nu,q})\cong\Linf(\GG)$. Consequently the assertions about invariants $\Ttau(\GG\times\HH_{\nu,q})$, $\TtauInn(\GG\times\HH_{\nu,q})$ and $\TtauAInn(\GG\times\HH_{\nu,q})$ follow from Proposition \ref{propTprod}\eqref{propTprod1} and Example \ref{exTtau}\eqref{exTtau2}.

We now pass to the proof of Corollary \ref{corManyIIIlambda}. First note that $\tfrac{2\pi^2}{\log{\lambda}}\log{\sqrt{\lambda}}\not\in\pi\QQ$ and similarly $\tfrac{2}{\log{\lambda}}(\pi{k}-1)\log\bigl(\lambda^{\frac{\pi}{2}}\bigr)\not\in\pi\QQ$ for all $k\in\NN$. Therefore, by Theorem \ref{thmManyIIIlambda} and Proposition \ref{propTprod}\eqref{propTprod1}, we have
\[
\Ttau(\GG^\lambda_k)
=\tfrac{2\pi}{\log{\lambda}}\ZZ\cap\tfrac{\pi}{\log\bigl(\lambda^{\frac{\pi}{2}}\bigr)}\ZZ
=\tfrac{2\pi}{\log{\lambda}}\ZZ\cap\tfrac{2}{\log{\lambda}}\ZZ=\{0\}
\]
and
\begin{align*}
\TtauInn(\GG^\lambda_k)
&=\tfrac{2\pi}{\log{\lambda}}\ZZ\cap
\Bigl(
\tfrac{2}{\log{\lambda}}(\pi{k}-1)\QQ+
\tfrac{\pi}{\log\bigl(\lambda^{\frac{\pi}{2}}\bigr)}\ZZ
\Bigr)\\
&=\tfrac{2\pi}{\log{\lambda}}\ZZ\cap
\Bigl(
\tfrac{2}{\log{\lambda}}(\pi{k}-1)\QQ+
\tfrac{2}{\log{\lambda}}\ZZ
\Bigr).
\end{align*}

Now let $t\in\TtauInn(\GG^\lambda_k)$. Then $t=\tfrac{2\pi{n}}{\log{\lambda}}$ for some $n\in\ZZ$ and $t=\tfrac{2}{\log{\lambda}}(\pi{k}-1)r+\tfrac{2n'}{\log{\lambda}}$ for some $r\in\QQ$ and $n'\in\ZZ$. This implies $\tfrac{2\pi}{\log{\lambda}}(n-kr)=\tfrac{2}{\log{\lambda}}(n'-r)$ and consequently $n=kr$ and $r=n'$ (as $\pi\not\in\QQ$). It follows that $t\in\tfrac{2\pi{k}}{\log{\lambda}}\ZZ$.

The inclusion $\tfrac{2\pi{k}}{\log{\lambda}}\ZZ\subset\TtauInn(\GG^\lambda_k)$ can be seen as follows: if $t\in\tfrac{2\pi{kn}}{\log{\lambda}}$ for some $n\in\ZZ$ then $t\in\tfrac{2\pi}{\log{\lambda}}\ZZ$ and
\[
t=\tfrac{2}{\log{\lambda}}(\pi{k}-1)n+\tfrac{2n}{\log{\lambda}}\in\tfrac{2}{\log{\lambda}}(\pi{k}-1)\ZZ+\tfrac{2}{\log{\lambda}}\ZZ.
\]

Finally $\TtauAInn(\GG^\lambda_k)=\RR$ because $\TtauAInn(\GG^\lambda_k)\supset\TtauAInn(\GG)=\RR$. Note that $\GG^\lambda_0$ is not isomorphic to $\GG^\lambda_k$ for $k>0$ because $\Ttau(\GG^\lambda_0)=\tfrac{2\pi}{\log{\lambda}}\ZZ$ (cf.~Example \ref{exTtauG}\eqref{exTtauG1}).
\end{proof}

\subsection{Bicrossed products by subgroups of the scaling group}\label{sectBicrossed}\hspace*{\fill}

In this section we will generalize the construction used to obtain the examples presented in Sections \ref{sectTheFactors1} and \ref{sectIII0}. The idea is to introduce an additional step of taking a bicrossed product by an action of a countable subgroup $\Gamma$ acting by the scaling automorphisms in such a way that the desired properties of the von Neumann algebra are not changed, but the group $\Gamma$ makes it possible to distinguish our examples by means of the invariants introduced at the beginning of Section \ref{sectTtau}.

Let us fix a non-trivial subgroup $\Gamma\subset\RR$ (not necessarily dense) which we equip with the discrete topology. For any compact quantum group $\GG$ we let $\Gamma$ act on $\Linf(\GG)$ via the scaling automorphisms. By a slight abuse of notation we will write $\tau^\GG$ also for the action $\Gamma\ni\gamma\mapsto\tau^\GG_\gamma\in\operatorname{Aut}(\Linf(\GG))$. The associated bicrossed product (compact) quantum group will be denoted by $\KK=\Gamma\bowtie\GG$ and the canonical inclusion of $\Linf(\GG)$ into $\Linf(\KK)=\Gamma\ltimes_{\tau^{\GG}}\Linf(\GG)$ by $\iota\colon\Linf(\GG)\hookrightarrow\Linf(\KK)$.

Recall from \cite[Theorem 6.1]{FimaMukherjeePatri} (cf.~also Section \ref{skladniki}) that the canonical unitaries $u_\gamma\in\Linf(\KK)$ ($\gamma\in\Gamma$) implementing the action of $\Gamma$ on $\Linf(\GG)$ are one-dimensional representations (i.e.~they are group-like elements) and hence are invariant under the modular group of the Haar measure of $\KK$ (cf.~\eqref{sigmatau}).

Next, let us recall that an automorphism $\theta$ of a von Neumann algebra $\rM$ is \emph{free} if for any $x\in\rM$ the condition that $xy=\theta(y)x$ for all $y\in\rM$ implies $x=0$ (\cite[Definition 1.3]{Kallman}, \cite[Definition 1.4.2]{SunderJones}) and an action $\alpha$ of a group $G$ on $\rM$ is \emph{free} if $\alpha_g$ is free for all $g\in{G}\setminus\{e\}$. Recall also that an automorphism $\theta$ of $\rM$ is \emph{outer} if $\theta\in\operatorname{Aut}(\rM)\setminus\operatorname{Inn}(\rM)$.

We have
\begin{itemize}
\item if $\alpha$ is a free action of a discrete group $G$ on $\rM$ and $\iota\colon\rM\to{G}\ltimes_\alpha\rM$ denotes the canonical inclusion then $\iota(\rM)'\cap(G\ltimes_\alpha\rM)=\iota(\cZ(\rM))$ (\cite[Proposition 1.4.4$(i)$]{SunderJones}),
\item If $\rM$ is a factor then an automorphism $\theta$ of $\rM$ is free if and only if it is outer (\cite[Remark 1.7]{Kallman}).
\end{itemize}

As an immediate consequence of these facts we obtain

\begin{proposition}\label{propCrossed}
Assume that $\Linf(\GG)$ is a factor. If for every $\gamma\in\Gamma\setminus\{0\}$ the automorphism $\tau^\GG_\gamma$ is outer then $\iota(\Linf(\GG))'\cap\Linf(\KK)=\CC\I$. In particular, $\Linf(\KK)$ is a factor.
\end{proposition}

Note that the scaling and modular automorphisms commute, so the scaling group $(\tau^\GG_t)_{t\in\RR}$ restricts to a one-parameter group of automorphisms of $\Linf(\GG)^\sigma$. In this context we have

\begin{proposition}\label{LKsigma}
If $\Linf(\GG)^\sigma$ is a factor and $\bigl.\tau^\GG_\gamma\bigr|_{\Linf(\GG)^\sigma}$ is outer for all $\gamma\in\Gamma\setminus\{0\}$ then $\Linf(\KK)^\sigma$ is a factor.
\end{proposition}

\begin{proof}
Let $x\in\cZ(\Linf(\KK)^\sigma)$. As $\Linf(\KK)=\Gamma\ltimes_{\tau^\GG}\Linf(\GG)$, we can\footnote{Define $x_{\gamma}=\mathbb{E}(u_{\gamma}^*x)\in\Linf(\GG)$, where $\mathbb{E}\colon\Gamma\ltimes_{\tau^{\GG}}\Linf(\GG)\rightarrow\Linf(\GG)$ is the canonical normal conditional expectation. Then an elementary calculation shows that $x(\delta_0\tens\Omega_{\GG})=\sum\limits_{\gamma\in\Gamma}\delta_\gamma\tens{x_\gamma}\Omega_{\GG}=\sum\limits_{\gamma\in\Gamma}u_\gamma\iota(x_\gamma)(\delta_0\tens\Omega_{\GG})$ in $\Ltwo(\KK)$.} write $x$ as a $\|\cdot\|_2$ convergent series $x=\sum\limits_{\gamma\in\Gamma}u_\gamma\iota(x_\gamma)$ with $x_\gamma\in\Linf(\GG)$ ($\gamma\in\Gamma$). Since the unitaries $u_\gamma$ are invariant under the modular group (of $\KK$) which is $\|\cdot\|_2$-continuous, the invariance of $x$ implies
\[
x_\gamma\in\Linf(\GG)^\sigma,\qquad\gamma\in\Gamma.
\]

For any $y\in\Linf(\GG)^\sigma$ we have $\iota(y)\in\Linf(\KK)^\sigma$ and hence
\[
\sum_{\gamma\in\Gamma}u_\gamma\iota\bigl(\tau^\GG_{-\gamma}(y)x_\gamma\bigr)
=\sum_{\gamma\in\Gamma}u_\gamma{u_\gamma^*}\iota(y)u_\gamma\iota(x_\gamma)
=\iota(y)x=x\iota(y)=\sum_{\gamma\in\Gamma}u_\gamma\iota(x_{\gamma}y)
\]
and consequently
\begin{equation}\label{freeK}
\tau^\GG_{-\gamma}(y)x_\gamma=x_{\gamma}y,\qquad\gamma\in\Gamma,\:y\in\Linf(\KK)^\sigma.
\end{equation}
As for $\gamma\neq{0}$ the map $\tau^\GG_\gamma$ is an outer automorphism of the factor $\Linf(\GG)^\sigma$, it is free and consequently \eqref{freeK} implies $x_\gamma=0$ for $\gamma\neq{0}$. It follows that $x=\iota(x_0)$ for some $x_0\in\Linf(\GG)^\sigma$ which by \eqref{freeK} belongs to $\cZ(\Linf(\GG)^\sigma)=\CC\I$.
\end{proof}

It may be hard to check that $\bigl.\tau^\GG_\gamma\bigr|_{\Linf(\GG)^\sigma}$ is outer for all $\gamma\neq{0}$. However, under a condition of symmetry of the spectrum of the $\uprho$-operators the situation becomes more manageable as will be seen in Theorem \ref{thm5}. Before stating and proving the theorem we need two preliminary results.

\begin{lemma}\label{tau2t}
Let $\HH$ be a compact quantum group and assume that $\operatorname{Sp}(\uprho_\alpha)=\operatorname{Sp}\bigl(\uprho_\alpha^{-1}\bigr)$ for all $\alpha\in\operatorname{Irr}{\HH}$. If $\bigl.\tau^\HH_t\bigr|_{\Linf(\HH)^\sigma}=\id$ for some $t\in\RR$ then $\tau^\HH_{2t}=\id$.
\end{lemma}

\begin{proof}
Take $\alpha\in\operatorname{Irr}{\GG}$, a unitary representation $U^\alpha\in\alpha$ and $i\in\{1,\dotsc,\dim{\alpha}\}$. By assumption there exists $k\in\{1,\dotsc,\dim{\alpha}\}$ such that $\uprho_{\alpha,k}=\uprho_{\alpha,i}^{-1}$. Thus $U^\alpha_{i,k}\in\Linf(\HH)^\sigma$ and consequently
\[
U^\alpha_{i,k}=\tau^\HH_t(U^\alpha_{i,k})=\uprho_{\alpha,i}^{\ii{t}}\uprho_{\alpha,k}^{-\ii{t}}U^\alpha_{i,k}
\]
which implies that $\uprho_{\alpha,i}^{2\ii{t}}=1$ because $U^\alpha_{i,k}\neq{0}$ (cf.~the orthogonality relations \cite[Theorem 1.4.3]{NeshveyevTuset}). Since this is true for any $i$, given $i,j\in\{1,\dotsc,\dim{\alpha}\}$ we compute
\[
\tau^\HH_{2t}(U^\alpha_{i,j})=\uprho_{\alpha,i}^{2\ii{t}}\uprho_{\alpha,j}^{-2\ii{t}}U^\alpha_{i,j}=U^\alpha_{i,j}.
\]
As $\alpha$ is arbitrary, we obtain $\tau^\HH_{2t}=\id$.
\end{proof}

\begin{remark}\label{remSym}
The symmetry assumption on the spectrum of the $\uprho$-operators in Lemma \ref{tau2t} (cf.~also Section \ref{TtauandT}) has been studied in \cite{symmetry}. In particular, it was proved there that this condition holds automatically for compact quantum groups of subexponential growth.
\end{remark}

\begin{proposition}\label{propAdv}
Let $\{\HH_n\}_{n\in\NN}$ be a family of compact quantum groups with the property that each $\HH_n$ appears in the sequence $(\HH_n)_{n\in\NN}$ infinitely many times and let $\GG=\bigtimes\limits_{n=1}^{\infty}\HH_n$. Assume that for some $t\in\RR$ and a unitary $v\in\Linf(\GG)$ we have $\bigl.\tau^\GG_t\bigr|_{\Linf(\GG)^\sigma}=\bigl.\operatorname{Ad}(v)\bigr|_{\Linf(\GG)^\sigma}$. Then $v\in\cZ(\Linf(\GG))$ and hence $\operatorname{Ad}(v)=\id$. In particular, $\TtauInn(\GG)=\Ttau(\GG)$.
\end{proposition}

Note that the statement of Proposition \ref{propAdv} is not that a scaling automorphism which coincides with an inner automorphism on $\Linf(\GG)^\sigma$ is trivial, but that the inner automorphism which coincides with the scaling automorphism $\tau^\GG_t$ on $\Linf(\GG)^\sigma$ is trivial, while $\tau^\GG_t$ itself might act non-trivially on some elements of $\Linf(\GG)\setminus\Linf(\GG)^\sigma$.

\begin{proof}[Proof of Proposition \ref{propAdv}]
Fix $N\in\NN$, $\alpha\in\operatorname{Irr}{\HH_N}$, $U^\alpha\in\alpha$ and $i,j\in\{1,\dotsc,\dim{\alpha}\}$. By assumption there is a strictly increasing sequence $(a_k)_{k\in\NN}$ of natural numbers such that $\HH_{a_k+1+N}=\HH_N$ for all $k$. Consequently with $i',j'\in\{1,\dotsc,\dim{\alpha}\}$ such that $\uprho_{\overline{\alpha},i'}=\uprho_{\alpha,i}^{-1}$ and $\uprho_{\overline{\alpha},j'}=\uprho_{\alpha,j}^{-1}$ (cf.~\cite[Proposition 1.4.7]{NeshveyevTuset}) we can define
\[
y_k=\I^{\tens(N-1)}\tens{U^\alpha_{i,j}}\tens\I^{\tens{a_k}}\tens{U^{\overline{\alpha}}_{i',j'}}\tens\I^{\tens\infty}\in\Linf(\GG),\qquad{k}\in\NN.
\]
It follows immediately from \eqref{sigmatau} that $y_k\in\Linf(\GG)^\sigma\cap\Linf(\GG)^\tau$ for all $k$. This means that we have
\begin{align*}
\I^{\tens(N-1)}\tens{U^\alpha_{i,j}}\tens\I^{\tens{a_k}}\tens{U^{\overline{\alpha}}_{i',j'}}&\tens\I^{\tens\infty}=y_k=\tau^\GG_t(y_k)=vy_kv^*\\
&=v(\I^{\tens(N-1)}\tens{U^\alpha_{i,j}}\tens\I^{\tens\infty})v^*v(\I^{\tens(a_k+N)}\tens{U^{\overline{\alpha}}_{i',j'}}\tens\I^{\tens\infty})v^*
\end{align*}
for all $k\in\NN$. It is easy to see that given $\eps>0$ there exists $k_\eps$ such that
\[
\bigl\|
\I^{\tens(a_k+N)}\tens{U^{\overline{\alpha}}_{i',j'}}\tens\I^{\tens\infty}
-v(\I^{\tens(a_k+N)}\tens{U^{\overline{\alpha}}_{i',j'}}\tens\I^{\tens\infty})v^*
\bigr\|_2\leq\eps
\]
for all $k\geq{k_\eps}$ and consequently
\begin{align*}
&\bigl\|y_k-v(\I^{\tens(N-1)}\tens{U^\alpha_{i,j}}\tens\I^{\tens\infty})v^*(\I^{\tens(a_k+N)}\tens{U^{\overline{\alpha}}_{i',j'}}\tens\I^{\tens\infty})\bigr\|_2\\
&\quad=\Bigl\|v(\I^{\tens(N-1)}\tens{U^\alpha_{i,j}}\tens\I^{\tens\infty})v^*\bigl(v(\I^{\tens(a_k+N)}\tens{U^{\overline{\alpha}}_{i',j'}}\tens\I^{\tens\infty})v^*-(\I^{\tens(a_k+N)}\tens{U^{\overline{\alpha}}_{i',j'}}\tens\I^{\tens\infty})\bigr)\Bigr\|_2\\
&\quad\leq\bigl\|v(\I^{\tens(N-1)}\tens{U^\alpha_{i,j}}\tens\I^{\tens\infty})v^*\bigr\|
\bigl\|
\I^{\tens(a_k+N)}\tens{U^{\overline{\alpha}}_{i',j'}}\tens\I^{\tens\infty}
-v(\I^{\tens(a_k+N)}\tens{U^{\overline{\alpha}}_{i',j'}}\tens\I^{\tens\infty})v^*
\bigr\|_2\\
&\quad\leq\bigl\|
\I^{\tens(a_k+N)}\tens{U^{\overline{\alpha}}_{i',j'}}\tens\I^{\tens\infty}
-v(\I^{\tens(a_k+N)}\tens{U^{\overline{\alpha}}_{i',j'}}\tens\I^{\tens\infty})v^*
\bigr\|_2\leq\eps
\end{align*}
which we will write as
\[
y_k\approx_{\eps}v(\I^{\tens(N-1)}\tens{U^\alpha_{i,j}}\tens\I^{\tens\infty})v^*(\I^{\tens(a_k+N)}\tens{U^{\overline{\alpha}}_{i',j'}}\tens\I^{\tens\infty}).
\]
Now let us apply both these operators to the cyclic vector $\Omega$ and act on the resulting vectors with the contraction $\I^{\tens(N+a_k)}\tens\ket{\bigl.\bigr.\Omega_{N+a_k+1}}\bra{U^{\overline{\alpha}}_{i',j'}\Omega_{N+a_k+1}}\tens\I^{\tens\infty}$. As a result we obtain
\begin{align*}
&\bigl\|U^{\overline{\alpha}}_{i',j'}\bigr\|_2^2\bigl(\Omega_1\tens\dotsm\tens\Omega_{N-1}\tens{U^\alpha_{i,j}\Omega_N}\tens\Omega_{N+1}\tens\dotsm\bigr)\\
&\qquad\approx_\eps
\bigl(\I^{\tens(N+a_k)}\tens\ket{\bigl.\bigr.\Omega_{N+a_k+1}}\bra{U^{\overline{\alpha}}_{i',j'}\Omega_{N+a_k+1}}\tens\I^{\tens\infty}\bigr)
v\bigl(\I^{\tens(N-1)}\tens{U^\alpha_{i,j}}\tens\I^{\tens\infty}\bigr)v^*\\
&\qquad\qquad\qquad\qquad\qquad\qquad\qquad\qquad\qquad\qquad\qquad\qquad\quad\quad\quad\bigl(\I^{\tens(a_k+N)}\tens{U^{\overline{\alpha}}_{i',j'}}\tens\I^{\tens\infty}\bigr)\Omega\\
&\qquad=\bigl(\I^{\tens(N+a_k)}\tens\ket{\bigl.\bigr.\Omega_{N+a_k+1}}\bra{U^{\overline{\alpha}}_{i',j'}\Omega_{N+a_k+1}}\tens\I^{\tens\infty}\bigr)
J_\bh\sigma^\bh_{\frac{\ii}{2}}\bigl(\I^{\tens(a_k+N)}\tens{U^{\overline{\alpha}}_{i',j'}}\tens\I^{\tens\infty}\bigr)^*J_\bh\\
&\qquad\qquad\qquad\qquad\qquad\qquad\qquad\qquad\qquad\qquad\qquad\qquad\quad\quad\quad{v}\bigl(\I^{\tens(N-1)}\tens{U^\alpha_{i,j}}\tens\I^{\tens\infty}\bigr)v^*\Omega
\end{align*}
(see Footnote \ref{sigmaft} in the proof of Theorem \ref{theCenters}). Now we can find $M$ sufficiently large and $z\in\operatorname{Pol}(\HH_1)\atens\dotsm\atens\operatorname{Pol}(\HH_M)\tens\I^{\tens\infty}$ such that
\[
\bigl\|v\bigl(\I^{\tens(N-1)}\tens{U^\alpha_{i,j}}\tens\I^{\tens\infty}\bigr)v^*-z\bigr\|_2\leq\eps\min\Bigl\{
\bigl\|U^{\overline{\alpha}}_{i',j'}\bigr\|_2^{-2},\bigl\|\sigma^\bh_{\frac{\ii}{2}}\bigl(U^{\overline{\alpha}}_{i',j'}\bigr)\bigr\|^{-1}\Bigr\}
\]
and consequently
\begin{align*}
&\bigl\|U^{\overline{\alpha}}_{i',j'}\bigr\|_2^2\bigl(\Omega_1\tens\dotsm\tens\Omega_{N-1}\tens{U^\alpha_{i,j}\Omega_N}\tens\Omega_{N+1}\tens\dotsm\bigr)\\
&\qquad\approx_{2\eps}
\bigl(\I^{\tens(N+a_k)}\tens\ket{\bigl.\bigr.\Omega_{N+a_k+1}}\bra{U^{\overline{\alpha}}_{i',j'}\Omega_{N+a_k+1}}\tens\I^{\tens\infty}\bigr)\\
&\qquad\qquad\qquad\qquad\qquad\qquad\qquad\qquad\qquad\quad\quad\quad
J_\bh\sigma^\bh_{\frac{\ii}{2}}\bigl(\I^{\tens(a_k+N)}\tens{U^{\overline{\alpha}}_{i',j'}}\tens\I^{\tens\infty}\bigr)^*J_\bh{z}\Omega\\
&\qquad=\bigl(\I^{\tens(N+a_k)}\tens\ket{\bigl.\bigr.\Omega_{N+a_k+1}}\bra{U^{\overline{\alpha}}_{i',j'}\Omega_{N+a_k+1}}\tens\I^{\tens\infty}\bigr)
z\bigl(\I^{\tens(a_k+N)}\tens{U^{\overline{\alpha}}_{i',j'}}\tens\I^{\tens\infty}\bigr)\Omega.
\end{align*}
Next note that for $k$ sufficiently large (such that $N+a_k\geq{M}$, e.g.~$k\geq{M}$) the operators $z$ and $\I^{\tens(N+a_k)}\tens\ket{\bigl.\bigr.\Omega_{N+a_k+1}}\bra{U^{\overline{\alpha}}_{i',j'}\Omega_{N+a_k+1}}\tens\I^{\tens\infty}$ commute, hence
\[
\bigl(\I^{\tens(N+a_k)}\tens\ket{\bigl.\bigr.\Omega_{N+a_k+1}}\bra{U^{\overline{\alpha}}_{i',j'}\Omega_{N+a_k+1}}\tens\I^{\tens\infty}\bigr)
z\bigl(\I^{\tens(a_k+N)}\tens{U^{\overline{\alpha}}_{i',j'}}\tens\I^{\tens\infty}\bigr)\Omega=
\bigl\|U^{\overline{\alpha}}_{i',j'}\bigr\|_2^2z\Omega,
\]
and so
\begin{align*}
\bigl\|U^{\overline{\alpha}}_{i',j'}\bigr\|_2^2\bigl(\Omega_1\tens\dotsm\tens\Omega_{N-1}\tens{U^\alpha_{i,j}\Omega_N}\tens\Omega_{N+1}\tens\dotsm\bigr)
&\approx_{2\eps}\bigl\|U^{\overline{\alpha}}_{i',j'}\bigr\|_2^2z\Omega\\
&\approx_{\eps}
\bigl\|U^{\overline{\alpha}}_{i',j'}\bigr\|_2^2
v\bigl(\I^{\tens(N-1)}\tens{U^\alpha_{i,j}}\tens\I^{\tens\infty}\bigr)v^*\Omega.
\end{align*}

As $\eps$ is arbitrary and $\Omega$ is separating for $\Linf(\GG)$, we obtain
\[
\I^{\tens(N-1)}\tens{U^\alpha_{i,j}}\tens\I^{\tens\infty}=
v\bigl(\I^{\tens(N-1)}\tens{U^\alpha_{i,j}}\tens\I^{\tens\infty}\bigr)v^*
\]
and since $N$, $\alpha$ and $i,j$ are arbitrary, we find that $v\in\cZ(\Linf(\GG))$, so that $\operatorname{Ad}(v)=\id$.
\end{proof}

\begin{remark}\label{remark2}
As noted by the referee, in the situation of Proposition \ref{propAdv} one can calculate the Vaes' invariant of $\GG$ as $T(\Linf(\GG),\Delta_{\GG})=\bigcap\limits_{n=1}^{\infty}T(\Linf(\HH_n),\Delta_{\HH_n})$. Proof of this result relies on the observation that any group-like unitary in $\Linf(\GG)$ is of the form $u_1\tens\dotsm\tens{u_N}\tens\I^{\tens\infty}$ for some $N\in\NN$ and group-like unitaries $u_k\in\Linf(\HH_k)$, ($1\leq{k}\leq{N}$).
\end{remark}

\begin{example}\label{example1}
Fix $q\in\left]-1,1\right[\setminus\{0\}$. Let the cyclic group $\tfrac{\pi}{2\log|q|}\ZZ$ act on $\Linf(\operatorname{SU}_q(2))$ via scaling automorphisms $(t,x)\mapsto\tau^{\operatorname{SU}_q(2)}_t(x)$ and define $\HH$ as the associated bicrossed product quantum group $\HH=(\tfrac{\pi}{2\log|q|}\ZZ)\bowtie\operatorname{SU}_q(2)$ (c.f.~Section \ref{skladniki}). In particular $\Linf(\HH)=\bigl(\tfrac{\pi}{2\log|q|}\ZZ\bigr)\ltimes\Linf(\SU_q(2))$ and, as in the introduction to this section, we let $\iota\colon\Linf(\SU_q(2))\to\Linf(\HH)$ be the canonical inclusion. Finally, set $\HH_n=\HH$ for $n\in\NN$ and $\GG=\bigtimes\limits_{n=1}^{\infty}\HH_n$. According to Proposition \ref{propAdv} we have $\TtauInn(\GG)=\Ttau(\GG)$ and one easily sees that $\Ttau(\GG)=\Ttau(\HH)=\Ttau(\SU_q(2))=\tfrac{\pi}{\log|q|}\ZZ$ (Proposition \ref{propTprod}, Example \ref{exTtau}). We claim that $\TtauAInn(\GG)=\tfrac{\pi}{2\log|q|}\ZZ$. To see inclusion $\supset$ observe first that $\tfrac{\pi}{2\log|q|}\ZZ\subset\TtauInn(\HH)$. Then the claim follows exactly as in the proof of formula \eqref{TtauAInnG} in Theorem \ref{TauGthm}. For the other inclusion, take $t\in\TtauAInn(\GG)$ and assume by contradiction that $t\not\in\tfrac{\pi}{2\log|q|}\ZZ$. Let $(v_m)_{m\in\NN}$ be a sequence of unitaries in $\Linf(\GG)$ such that
\begin{equation}\label{eq3}
\operatorname{Ad}(v_m)\xrightarrow[m\to\infty]{}\tau^{\GG}_t.
\end{equation}
Recall from Section \ref{skladniki} that there is an isomorphism $\Phi_q\colon\Linf(\SU_q(2))\rightarrow{\Int_{\TT}}^{\oplus}\B(\sH_\lambda)\dd\mu(\lambda)$, where $\sH_\lambda=\ell^2(\ZZ_+)$ ($\lambda\in\TT$). Under this isomorphism, the scaling automorphism $\tau^{\SU_q(2)}_t$ acts as a rotation by $|q|^{2\ii{t}}$ (\cite[Proposition 7.3]{modular}), i.e.
\[
\tau^{\SU_q(2)}_t\Biggl(\Phi_q^{-1}\biggl({\Int_{\TT}}^{\oplus}x_\lambda\dd\mu(\lambda)\biggr)\Biggr)
=\Phi_q^{-1}\biggl({\Int_{\TT}}^{\oplus}x_{|q|^{2\ii{t}}\lambda}\dd\mu(\lambda)\biggr)
\]
(c.f.~\cite{invariants}). Choose an arc with positive length $\Omega\subset\TT$ such that $|q|^{-2\ii{t}}\Omega\cap(\Omega\cup(-\Omega))=\emptyset$ (this is possible since $t\not\in\tfrac{\pi}{2\log|q|}\ZZ$) and let $1_\Omega$ be the associated characteristic function. Next, set $y=\Phi_q^{-1}\bigl({\Int_{\TT}}^{\oplus}1_\Omega(\lambda)\I_{{\sH}_{\lambda}}\dd\mu(\lambda)\bigr)\in\Linf(\SU_q(2))$. Fix $\omega_0$, a normal state on $\B(\sH_\lambda)=\B(\ell^2(\ZZ_+))$ and let $\omega\in\Lone(\SU_q(2))$ be given by $\omega\Bigl(\Phi_q^{-1}\bigl({\Int_{\TT}}^{\oplus}x_\lambda\dd\mu(\lambda)\bigr)\Bigr)=\Int_{|q|^{-2\ii{t}}\Omega}\omega_0(x_\lambda)\dd\mu(\lambda)$. Denote by $\EE\colon\Linf(\HH)=\bigl(\tfrac{\pi}{2\log|q|}\ZZ)\ltimes\Linf(\SU_q(2)\bigr)\rightarrow\iota(\Linf(\SU_q(2)))$ the canonical normal conditional expectation. Towards a contradiction, on the one hand we have
\begin{equation}\label{eq1}
\begin{aligned}
\bigl(\omega\comp\iota^{-1}\comp\EE\tens\bh_{\HH}^{\tens\infty}\bigr)
\Bigl(\tau^{\GG}_t\bigl(\iota(y)\tens\I^{\tens\infty}\bigr)\Bigr)
&=(\omega\comp\iota^{-1}\comp\EE)(\iota(\tau^{\SU_q(2)}_t(y)))\\
&=\omega\Biggl(
\Phi_q^{-1}\biggl(
{\Int_{\TT}}^{\oplus}1_{|q|^{-2\ii{t}}\Omega}(\lambda)
\I_{\sH_\lambda}\dd\mu(\lambda)
\biggr)\Biggr)\\
&=\Int_{|q|^{-2\ii{t}}\Omega}
1_{|q|^{-2\ii{t}}\Omega}(\lambda)\dd\mu(\lambda)=\mu(\Omega)>0.
\end{aligned}
\end{equation}
On the other hand, we will show that for any unitary operator $u\in\Linf(\GG)$ we have
\[
\bigl(\omega\comp\iota^{-1}\comp\EE\tens\bh_{\HH}^{\tens\infty}\bigr)\Bigl(\operatorname{Ad}(u)\bigl(\iota(y)\tens\I^{\tens\infty}\bigr)\Bigr)=0.
\]
Together with \eqref{eq1}, this clearly contradicts convergence \eqref{eq3}. In fact, we will show a more general result
\begin{equation}\label{eq4}
\bigl(\omega\comp\iota^{-1}\comp\EE\tens\bh_{\HH}^{\tens\infty}\bigr)\Bigl(z\bigl(\iota(y)\tens\I^{\tens\infty}\bigr)z'\Bigr)=0
\end{equation}
for arbitrary $z,z'\in\Linf(\GG)$. Since $\Linf(\GG)=\barbigotimes\limits_{n=1}^{\infty}\bigl(\tfrac{\pi}{2\log|q|}\ZZ\bigr)\ltimes\Linf(\SU_q(2))$ and $\omega\comp\iota^{-1}\comp\EE\tens\bh_{\HH}^{\tens\infty}$ is a normal functional, it is enough to consider $z=z_{1}\tens\dotsm\tens{z_M}\tens\I^{\tens\infty}$, $z'=z'_1\tens\dotsm\tens{z'_{M}}\tens\I^{\tens\infty}$ and $z_1,z'_1$ of the form
\[
z_1=\iota(w)u_{\frac{\pi}{2\log|q|}m},\qquad
z'_1=\iota(w')u_{\frac{\pi}{2\log|q|}m'}
\]
for some $w,w'\in\Linf(\SU_q(2))$, $m,m'\in\ZZ$, where $\{u_{\frac{\pi}{2\log|q|}k}\}_{k\in\ZZ}$ are the generators of $L\bigl(\tfrac{\pi}{2\log|q|}\ZZ\bigr)$ inside $\Linf(\HH)$. For this choice of $z,z'$ we have
\begin{align*}
\bigl(\omega\comp\iota^{-1}\comp\EE\tens\bh_{\HH}^{\tens\infty}\big)&\Bigl(z\bigl(\iota(y)\tens\I^{\tens\infty}\bigr)z'\Bigr)\\
&=\bigl(\omega\comp\iota^{-1}\comp\EE\bigr)
\bigl(\iota(w)u_{\frac{\pi}{2\log|q|}m}\iota(y)\iota(w')u_{\frac{\pi}{2\log|q|}m'}\bigr)
\biggl(\prod_{n=2}^{M}\bh_{\HH}(z_nz'_n)\biggr)\\
&=(\omega\comp\iota^{-1}\comp\EE)\Bigl(\iota\bigl(w\tau^{\SU_q(2)}_{\frac{\pi}{2\log|q|}m}(yw')\bigr)
u_{\frac{\pi}{2\log|q|}(m+m')}\Bigr)\biggl(\prod_{n=2}^{M}\bh_{\HH}(z_nz'_n)\biggr)\\
&=\delta_{m,-m'}\omega\Biggl(\Phi_q^{-1}\biggl({\Int_{\TT}}^{\oplus}w_\lambda(yw')_{|q|^{2\ii\frac{\pi}{2\log|q|}m}\lambda}\dd\mu(\lambda)\biggr)
\Biggr)\biggl(\prod_{n=2}^{M}\bh_{\HH}(z_nz'_n)\biggr)\\
&=\delta_{m,-m'}\biggl(\Int_{|q|^{-2\ii{t}}\Omega}1_{\Omega}(\ee^{\ii\pi{m}}\lambda)\omega_0(w_\lambda{w'_{\ee^{\ii\pi{m}}\lambda}})\dd\mu(\lambda)\biggr)
\biggl(\prod_{n=2}^{M}\bh_{\HH}(z_nz'_n)\biggr)=0,
\end{align*}
where the last equality follows from $|q|^{-2\ii{t}}\Omega\cap(\Omega\cup(-\Omega))=\emptyset$. This shows \eqref{eq4} and in turn ends the proof of $\TtauAInn(\GG)=\tfrac{\pi}{2\log|q|}\ZZ$.

In particular, observe that $\GG$ can be distinguished from $\SU_q(2)$ using $\TtauAInn$, but not using invariants $\Ttau$ and $\TtauInn$.
\end{example}

Before stating the main theorem of this section let us recall that for a compact quantum group $\HH$ the von Neumann algebra of \emph{class functions} on $\HH$ is
\[
\cC_\HH=\bigl\{\chi_\alpha\,\bigr|\bigl.\,\alpha\in\operatorname{Irr}{\HH}\bigr\}''\subset\Linf(\HH),
\]
where $\chi_\alpha$ denotes the character of any representation from the equivalence class $\alpha$ (\cite[Definition 1.1]{KrajczokWasilewski}, see also \cite{AlaghmandanCrann}).

\newcounter{c}
\begin{theorem}\label{thm5}
Let $\Gamma$ be a subgroup of $\RR$ and let $\{\HH_n\}_{n\in\NN}$ be a family of compact quantum groups such that each $\HH_n$ appears in the sequence $(\HH_n)_{n\in\NN}$ infinitely many times. Equip $\Gamma$ with the discrete topology. Assume furthermore that
\begin{itemize}
\item $\Linf(\HH_n)$ is a factor for every $n$,
\item $\Gamma\cap\bigcap\limits_{n=1}^{\infty}\Ttau(\HH_n)=\{0\}$,
\item for every $n$ and every $\alpha\in\operatorname{Irr}{\HH_n}$ we have $\operatorname{Sp}(\uprho_\alpha)=\operatorname{Sp}\bigl(\uprho_\alpha^{-1}\bigr)$.
\end{itemize}
Let $\GG=\bigtimes\limits_{n=1}^{\infty}\HH_n$ and $\KK=\Gamma\bowtie\GG$ with respect to the action of $\Gamma$ via the scaling automorphisms. Then
\begin{enumerate}
\item\label{thm5-2} $\Linf(\KK)$ and $\Linf(\KK)^\sigma$ are factors,\footnote{\label{footsigma}Since the modular group acts trivially on the center of $\Linf(\GG)$, the fact that $\Linf(\GG)^\sigma$ is a factor implies that $\Linf(\GG)$ is a factor as well.}
\item\label{thm5-3} the invariant $S(\Linf(\KK))$ is
\[
\overline{\bigl\{
\uprho_{\beta_1,i_1}\uprho_{\beta_1,j_1}\dotsm\uprho_{\beta_n,i_n}\uprho_{\beta_n,j_n}
\,\bigr|\bigl.\,
n\in\NN,\:\beta_k\in\operatorname{Irr}{\HH_k},\:i_k,j_k\in\{1,\dotsc,\dim{\beta_k}\},\:k\in\{1,\dotsc,n\}
\bigr\}},
\]
\item\label{thm5-4} We have $\Ttau(\KK)=\Ttau(\GG)=\TtauInn(\GG)=\bigcap\limits_{n=1}^{\infty}\Ttau(\HH_n)$ and $\Gamma+\bigcap\limits_{n=1}^{\infty}\Ttau(\HH_n)\subset\TtauInn(\KK)$.
\setcounter{c}{\value{enumi}}
\end{enumerate}
If furthermore $\cC_{\HH_n}^{\;\prime}\cap\Linf(\HH_n)\subset\cC_{\HH_n}$ for all $n$ then
\begin{enumerate}
\setcounter{enumi}{\value{c}}
\item\label{thm5-5} $\TtauInn(\KK)=\Gamma+\bigcap\limits_{n=1}^{\infty}\Ttau(\HH_n)$.
\end{enumerate}
\end{theorem}

\begin{remark}\label{remark3}
\noindent
\begin{enumerate}
\item In the situation of Theorem \ref{thm5} the factor $\Linf(\KK)$ is never of type $\mathrm{III}_0$, as in this case the invariant $S$ is equal to $\{0,1\}$.
\item If $\Gamma$ is countable and $\HH_n$ second countable for all $n\in\NN$, then $\GG$ and $\KK$ are second countable.
\item Consider a more general situation: let $\Gamma$ be a subgroup of $\RR$ equipped with the discrete topology, acting on a von Neumann algebra $\Linf(\GG)$ of a compact quantum group $\GG$ via scaling automorphisms and let $\KK=\Gamma\bowtie\GG$ be the resulting bicrossed product. It was observed by the referee that in this situation one can compute the Vaes' invariant of $\KK$ as $T(\Linf(\KK),\Delta_{\KK})=\Gamma+T(\Linf(\GG),\Delta_{\GG})$.
\end{enumerate}
\end{remark}

\begin{proof}[Proof of Theorem \ref{thm5}\eqref{thm5-2}--\eqref{thm5-4}]

Ad \eqref{thm5-4}. The equalities $\Ttau(\GG)=\TtauInn(\GG)=\bigcap\limits_{n=1}^{\infty}\Ttau(\HH_n)$ follow immediately from Proposition \ref{propAdv} and Proposition \ref{propTprod}\eqref{propTprod2}, while the equalities $\Ttau(\KK)=\Ttau(\GG)$ and $\Gamma+\bigcap\limits_{n=1}^{\infty}\Ttau(\HH_n)\subset\TtauInn(\KK)$ are proved exactly as \cite[Lemma 4.4]{KrajczokWasilewski} (where $\Gamma$ was $\nu\QQ$). This establishes \eqref{thm5-4}.

Ad \eqref{thm5-2}. Since $\Linf(\HH_n)$ is a factor for all $n$, we have that $\Linf(\GG)$ is a factor and, in view of \eqref{thm5-4}, the assumption $\Gamma\cap\bigcap\limits_{n=1}^{\infty}\Ttau(\HH_n)=\{0\}$ guarantees that the action of $\Gamma$ on $\Linf(\GG)$ is outer, so that $\Linf(\KK)$ is a factor (cf.~Proposition \ref{propCrossed} and remarks preceding it). Next note that by Theorem \ref{theCenters}\eqref{theCenters1} the algebra $\Linf(\GG)^\sigma$ is a factor (in this case this is implied by the fact that $\Linf(\GG)$ is a factor). Furthermore, if for some $\gamma\in\Gamma\setminus\{0\}$ the automorphism $\bigl.\tau^\GG_\gamma\bigr|_{\Linf(\GG)^\sigma}$ were inner then by Proposition \ref{propAdv} it would be trivial (on this subalgebra). Consequently, by Lemma \ref{tau2t} we would have $\tau^\GG_{2\gamma}=\id$. But then $2\gamma\neq{0}$, which in view of \eqref{thm5-4} contradicts the assumption that $\Gamma\cap\bigcap\limits_{n=1}^{\infty}\Ttau(\HH_n)=\{0\}$. Consequently $\tau^\GG_\gamma$ is outer for all non-zero $\gamma$, and thus $\Linf(\KK)^\sigma$ is a factor by Proposition \ref{LKsigma}.

Ad \eqref{thm5-3}. Once we know that $\Linf(\KK)^\sigma$ is a factor, the invariant $S(\Linf(\KK))$ is the spectrum of $\modOp[\bh_\KK]$. The latter is easily seen to be the closure of
\begin{align*}
&\bigl\{
\uprho_{\alpha,i}\uprho_{\alpha,j}\,\bigr|\bigl.\,\alpha\in\operatorname{Irr}{\KK},\:i,j\in\{1,\dotsc,\dim{\alpha}\}
\bigr\}\\
&\qquad=\bigl\{
\uprho_{\alpha,i}\uprho_{\alpha,j}\,\bigr|\bigl.\,\alpha\in\operatorname{Irr}{\GG},\:i,j\in\{1,\dotsc,\dim{\alpha}\}
\bigr\}\\
&\qquad=\bigl\{
\uprho_{\beta_1,i_1}\uprho_{\beta_1,j_1}\dotsm\uprho_{\beta_n,i_n}\uprho_{\beta_n,j_n}
\,\bigr|\bigl.\,
n\in\NN,\:\beta_k\in\operatorname{Irr}{\HH_k},\:i_k,j_k\in\{1,\dotsc,\dim{\beta_k}\},\:k\in\{1,\dotsc,n\}
\bigr\}
\end{align*}
(cf.~\cite[Theorem 6.1]{FimaMukherjeePatri} and Section \ref{skladniki}).
\end{proof}

For the proof of the last statement of Theorem \ref{thm5} we need some preliminary results concerning the condition $\cC_{\HH}^{\;\prime}\cap\Linf(\HH)\subset\cC_{\HH}$.

\begin{lemma}\label{CC}
Let $(\HH_n)_{n\in\NN}$ be a family of compact quantum groups and set $\GG=\bigtimes\limits_{n=1}^{\infty}\HH_n$. Then
\begin{enumerate}
\item\label{CC1} we have
\[
\cC_\GG=\barbigotimes_{n=1}^{\infty}\cC_{\HH_n}\quad\text{and}\quad
\cC_\GG^{\;\prime}\cap\Linf(\GG)=\barbigotimes_{n=1}^{\infty}\bigl(\cC_{\HH_n}^{\;\prime}\cap\Linf(\HH_n)\bigr),
\]
\item\label{CC2} the condition that $\cC_{\HH_n}^{\;\prime}\cap\Linf(\HH_n)\subset\cC_{\HH_n}$ for all $n$ implies $\cC_{\GG}^{\;\prime}\cap\Linf(\GG)\subset\cC_{\GG}$.
\end{enumerate}
\end{lemma}

\begin{proof}
Ad \eqref{CC1}. The equality $\cC_\GG=\barbigotimes\limits_{n=1}^{\infty}\cC_{\HH_n}$ is an easy consequence of the description of $\operatorname{Irr}{\GG}$ provided in Section \ref{sectIPQG}. Also the inclusion $\cC_\GG^{\;\prime}\cap\Linf(\GG)\subset\barbigotimes\limits_{n=1}^{\infty}\bigl(\cC_{\HH_n}^{\;\prime}\cap\Linf(\HH_n)\bigr)$ is immediate. Let us prove the converse.

Take $x\in\cC_\GG^{\;\prime}\cap\Linf(\GG)$. To show that
\[
x\in\barbigotimes_{n=1}^{\infty}\bigl(\cC_{\HH_n}^{\;\prime}\cap\Linf(\HH_n)\bigr)
=\biggl(\barbigotimes_{n=1}^{\infty}\bigl(\cC_{\HH_n}^{\;\prime}\cap\Linf(\HH_n)\bigr)\biggr)''
=\biggl(\barbigotimes_{n=1}^{\infty}\bigl({\cC_{\HH_n}}\vee\Linf(\HH_n)'\bigr)\biggr)'
\]
let us fix $N\in\NN$ and consider $Y=\I^{\tens(N-1)}\tens{y}\tens\I^{\tens\infty}$ with $y\in\cC_{\HH_N}$. Since $Y\in\cC_\GG$, we find that $x$ commutes with $Y$. Furthermore, $x$ also commutes with elements of the form $Y'=\I^{\tens(N-1)}\tens{y'}\tens\I^{\tens\infty}$ with $y'\in\Linf(\HH_N)'$. Consequently $x$ commutes with all elements of $\barbigotimes\limits_{n=1}^{\infty}\bigl({\cC_{\HH_n}}\vee\Linf(\HH_n)'\bigr)$ which ends the proof of \eqref{CC1}.

Ad \eqref{CC2}. Take $x\in\cC_\GG^{\;\prime}\cap\Linf(\GG)$. By \eqref{CC1}
\[
x\in\barbigotimes_{n=1}^\infty\bigl(\cC_{\HH_n}^{\;\prime}\cap\Linf(\HH_n)\bigr)
\]
which by assumption is contained in $\barbigotimes\limits_{n=1}^\infty\cC_{\HH_n}$ which, again by \eqref{CC1}, is $\cC_\GG$.
\end{proof}

\begin{lemma}\label{lemGK}
Let $\Gamma\subset\RR$ be a subgroup (equipped with the discrete topology) and let $\Gamma$ act on a compact quantum group $\GG$ by scaling automorphisms. Let $\KK=\Gamma\bowtie\GG$ be the corresponding bicrossed product. If $\cC_\GG^{\;\prime}\cap\Linf(\GG)\subset\cC_\GG$ then $\cC_\KK^{\;\prime}\cap\Linf(\KK)\subset\cC_\KK$.
\end{lemma}

\begin{proof}
Take $x\in\cC_\KK^{\;\prime}\cap\Linf(\KK)$ and write $x$ in the form
\[
x=\sum_{\gamma\in\Gamma}u_\gamma\iota(x_\gamma),
\]
where $\iota\colon\Linf(\GG)\to\Linf(\KK)$ is the canonical inclusion into the crossed product. For any $y\in\cC_\GG$ we have $\iota(y)\in\cC_\KK\subset\Linf(\KK)^\tau$ (cf.~\eqref{sigmatau}), so that $\iota(y)$ commutes with the unitaries $\{u_\gamma\}_{\gamma\in\Gamma}$ implementing the action of $\Gamma$. Consequently
\[
\sum_{\gamma\in\Gamma}u_\gamma\iota(yx)=\iota(y)x=x\iota(y)=\sum_{\gamma\in\Gamma}u_\gamma\iota(xy)
\]
and it follows that $yx_\gamma=x_{\gamma}y$ for all $\gamma\in\Gamma$ and any $y\in\cC_\GG$. Hence $x_\gamma\in\cC_\GG^{\;\prime}\cap\Linf(\GG)\subset\cC_\GG$ and since $u_\gamma\in\cC_\KK$ and $\iota(\cC_\GG)\subset\cC_\KK$ (see \cite[Section 4.2]{KrajczokWasilewski} as well as Section \ref{skladniki}) we find that denoting by $\Omega$ the canonical cyclic vector for $\KK$ we have
\[
x\Omega=\sum_{\gamma\in\Gamma}u_\gamma\iota(x_\gamma)\Omega\in\overline{\{a\Omega\,|\,a\in\cC_\KK\}}.
\]
Thus the conclusion that $x\in\cC_\KK$ follows from \cite[Lemma 4.6]{KrajczokWasilewski}.
\end{proof}

\begin{proof}[Proof of Theorem \ref{thm5}\eqref{thm5-5}]
We are assuming that $\cC_{\HH_n}^{\;\prime}\cap\Linf(\HH_n)\subset\cC_{\HH_n}$ for all $n$ and aim to show that $\TtauInn(\KK)\subset\Ttau(\KK)+\Gamma$. Let $\bh_\KK$ and $\bh_\GG$ be the Haar measures of $\KK$ and $\GG$ respectively.

Take $t_0\in\TtauInn(\KK)$ and a unitary $v\in\Linf(\KK)$ such that $\tau^\KK_{t_0}=\operatorname{Ad}(v)$. From Lemma \ref{CC}\eqref{CC2} and Lemma \ref{lemGK} we know that $\cC_\KK^{\;\prime}\cap\Linf(\KK)\subset\cC_\KK$, so by \cite[Proposition 3.2]{KrajczokWasilewski} we have
\[
v\in\overline{\operatorname{span}\bigl\{\chi_\alpha\,\bigr|\bigl.\,\alpha\in\operatorname{Irr}{\KK},\:\uprho_\alpha=\I\bigr\}}^{\mathrm{w}^*}
=\!\quad\overline{\operatorname{span}\bigl\{u_\gamma\iota(\chi_\beta)\,\bigr|\bigl.\,\gamma\in\Gamma,\:\beta\in\operatorname{Irr}{\GG},\:\uprho_\beta=\I\bigr\}}^{\mathrm{w}^*}.
\]
Arguing as in the proof of \cite[Proposition 4.8]{KrajczokWasilewski} we can identify $\Ltwo(\KK)$ with $\ell^2(\Gamma)\tens\Ltwo(\GG)$ and write
\begin{equation}\label{vOmega}
v\Omega_{\bh_\KK}=\sum_{\gamma\in\Gamma}\sum_{\substack{\beta\in\operatorname{Irr}{\GG}\\\uprho_\beta=\I}}v_{\beta,\gamma}u_\gamma\iota(\chi_\beta)\Omega_{\bh_\KK}
=\sum_{\gamma\in\Gamma}\sum_{\substack{\beta\in\operatorname{Irr}{\GG}\\\uprho_\beta=\I}}v_{\beta,\gamma}(\delta_\gamma\tens\chi_\beta\Omega_{\bh_\GG}),
\end{equation}
where $\{\delta_\gamma\}_{\gamma\in\Gamma}$ is the standard basis of $\ell^2(\Gamma)$ and $v_{\beta,\gamma}\in\CC$. Now take $y\in\operatorname{Pol}(\GG)$. We have
\begin{equation}\label{tauiota}
\begin{aligned}
\sum_{\gamma\in\Gamma}\sum_{\substack{\beta\in\operatorname{Irr}{\GG}\\\uprho_\beta=\I}}v_{\beta,\gamma}\bigl(\delta_\gamma\tens\tau^\GG_\gamma(y)\chi_\beta\Omega_{\bh_\GG}\bigr)&=\iota(y)v\Omega_{\bh_\KK}
=vv^*\iota(y)v\Omega_{\bh_\KK}\\
&=v\tau^\KK_{-t_0}\bigl(\iota(y)\bigr)\Omega_{\bh_\KK}=v\iota\bigl(\tau^\GG_{-t_0}(y)\bigr)\Omega_{\bh_\KK}.
\end{aligned}
\end{equation}
Observe that Equation \eqref{vOmega} implies
\[
1=\|v\|_2^2=\sum_{\gamma\in\Gamma}\sum_{\substack{\beta\in\operatorname{Irr}{\GG}\\\uprho_\beta=\I}}|v_{\beta,\gamma}|^2.
\]
It follows from this that the series $\sum\limits_{\substack{\beta\in\operatorname{Irr}{\GG}\\\uprho_\beta=\I}}v_{\beta,\gamma}\chi_\beta\Omega_{\bh_\GG}$ converges in $\Ltwo(\GG)$ for any $\gamma$. Since $v$ is unitary, there must exist $\gamma_0\in\Gamma$ such that
\begin{equation}\label{nonzero}
\sum\limits_{\substack{\beta\in\operatorname{Irr}{\GG}\\\uprho_\beta=\I}}v_{\beta,\gamma_0}\chi_\beta\Omega_{\bh_\GG}\neq{0}.
\end{equation}
Applying the map $\bra{\delta_{\gamma_0}}\tens\I_{\Ltwo(\GG)}$ to both sides of \eqref{tauiota} yields
\begin{align*}
\sum_{\substack{\beta\in\operatorname{Irr}{\GG}\\\uprho_\beta=\I}}v_{\beta,\gamma_0}\tau^\GG_{\gamma_0}(y)\chi_\beta\Omega_{\bh_\GG}
&=\bigl(\bra{\delta_{\gamma_0}}\tens\I_{\Ltwo(\GG)}\bigr)\Bigl(v\iota\bigl(\tau^\GG_{-t_0}(y)\bigr)\Omega_{\bh_\KK}\Bigr)\\
&=\bigl(\bra{\delta_{\gamma_0}}\tens\I_{\Ltwo(\GG)}\bigr)\Bigl(J_{\bh_\KK}\sigma^{\bh_\KK}_{\frac{\ii}{2}}\Bigl(\iota\bigl(\tau^\GG_{-t_0}(y)\bigr)\Bigr)^*J_{\bh_\KK}v\Omega_{\bh_\KK}\Bigr)\\
&=\bigl(\bra{\delta_{\gamma_0}}\tens\I_{\Ltwo(\GG)}\bigr)\Bigl(J_{\bh_\KK}\sigma^{\bh_\KK}_{\frac{\ii}{2}}\Bigl(\iota\bigl(\tau^\GG_{-t_0}(y)\bigr)\Bigr)^*J_{\bh_\KK}
\sum_{\gamma\in\Gamma}\sum_{\substack{\beta\in\operatorname{Irr}{\GG}\\\uprho_\beta=\I}}v_{\beta,\gamma}u_\gamma\iota(\chi_\beta)\Omega_{\bh_\KK}\Bigr)\\
&=\sum_{\gamma\in\Gamma}\sum_{\substack{\beta\in\operatorname{Irr}{\GG}\\\uprho_\beta=\I}}v_{\beta,\gamma}\bigl(\bra{\delta_{\gamma_0}}\tens\I_{\Ltwo(\GG)}\bigr)
\Bigl(J_{\bh_\KK}\sigma^{\bh_\KK}_{\frac{\ii}{2}}\Bigl(\iota\bigl(\tau^\GG_{-t_0}(y)\bigr)\Bigr)^*J_{\bh_\KK}u_\gamma\iota(\chi_\beta)\Omega_{\bh_\KK}\Bigr)\\
&=\sum_{\gamma\in\Gamma}\sum_{\substack{\beta\in\operatorname{Irr}{\GG}\\\uprho_\beta=\I}}v_{\beta,\gamma}\bigl(\bra{\delta_{\gamma_0}}\tens\I_{\Ltwo(\GG)}\bigr)
\Bigl(u_\gamma\iota(\chi_\beta)\iota\bigl(\tau^\GG_{-t_0}(y)\bigr)\Omega_{\bh_\KK}\Bigr)\\
&=\sum_{\gamma\in\Gamma}\sum_{\substack{\beta\in\operatorname{Irr}{\GG}\\\uprho_\beta=\I}}v_{\beta,\gamma}\bigl(\bra{\delta_{\gamma_0}}\tens\I_{\Ltwo(\GG)}\bigr)
\bigl(\delta_\gamma\tens\chi_\beta\tau^\GG_{-t_0}(y)\Omega_{\bh_\GG}\bigr)\\
&=\sum_{\substack{\beta\in\operatorname{Irr}{\GG}\\\uprho_\beta=\I}}v_{\beta,\gamma_0}\chi_\beta\tau^\GG_{-t_0}(y)\Omega_{\bh_\GG}.
\end{align*}
Now let $P_\GG$ be the strictly positive self-adjoint operator on $\Ltwo(\GG)$ implementing the scaling group as in \cite[Definition 6.9]{KustermansVaes}. Applying $P_\GG^{\ii{t_0}}$ to both sides of
\[
\sum_{\substack{\beta\in\operatorname{Irr}{\GG}\\\uprho_\beta=\I}}v_{\beta,\gamma_0}\tau^\GG_{\gamma_0}(y)\chi_\beta\Omega_{\bh_\GG}
=\sum_{\substack{\beta\in\operatorname{Irr}{\GG}\\\uprho_\beta=\I}}v_{\beta,\gamma_0}\chi_\beta\tau^\GG_{-t_0}(y)\Omega_{\bh_\GG}
\]
we obtain
\begin{equation}\label{Pit}
\tau^\GG_{t_0+\gamma_0}(y)\sum_{\substack{\beta\in\operatorname{Irr}{\GG}\\\uprho_\beta=\I}}v_{\beta,\gamma_0}\chi_\beta\Omega_{\bh_\GG}
=\sum_{\substack{\beta\in\operatorname{Irr}{\GG}\\\uprho_\beta=\I}}v_{\beta,\gamma_0}\chi_{\beta}y\Omega_{\bh_\GG}
=J_{\bh_\GG}\sigma^{\bh_\GG}_{\frac{\ii}{2}}(y)^*J_{\bh_\GG}\sum_{\substack{\beta\in\operatorname{Irr}{\GG}\\\uprho_\beta=\I}}v_{\beta,\gamma_0}\chi_\beta\Omega_{\bh_\GG}.
\end{equation}

Now let $\EE\colon\Linf(\KK)\to\Linf(\GG)$ be the canonical normal conditional expectation on the crossed product composed with $\iota^{-1}$. For any $\xi\in\Ltwo(\GG)$ we have
\begin{align*}
\is{\xi}{\EE({u_{\gamma_0}^*}v)\Omega_{\bh_\GG}}&=\is{\delta_0\tens\xi}{{u_{\gamma_0}^*}v\bigl(\delta_0\tens\Omega_{\bh_\GG}\bigr)}
=\is{\delta_{\gamma_0}\tens\xi}{v\Omega_{\bh_\KK}}\\
&=\sum_{\gamma\in\Gamma}\sum_{\substack{\beta\in\operatorname{Irr}{\GG}\\\uprho_\beta=\I}}v_{\beta,\gamma}\is{\delta_{\gamma_0}\tens\xi}{\delta_\gamma\tens\chi_\beta\Omega_{\bh_\GG}}
=\biggis{\xi}{\sum_{\substack{\beta\in\operatorname{Irr}{\GG}\\\uprho_\beta=\I}}v_{\beta,\gamma_0}\chi_\beta\Omega_{\bh_\GG}}
\end{align*}
which shows that $\EE({u_{\gamma_0}^*}v)\Omega_{\bh_\GG}=\sum\limits_{\substack{\beta\in\operatorname{Irr}{\GG}\\\uprho_\beta=\I}}v_{\beta,\gamma_0}\chi_\beta\Omega_{\bh_\GG}$. Combining this with \eqref{Pit} results in
\begin{equation}\label{freet0g0}
\tau^\GG_{t_0+\gamma_0}(y)\EE(u_{\gamma_0}^*v)=\EE(u_{\gamma_0}^*v)y
\end{equation}
for any $y\in\operatorname{Pol}(\GG)$ and by continuity for all $y\in\Linf(\GG)$.

Now, if the automorphism $\tau^\GG_{t_0+\gamma_0}$ were inner we would have $\gamma_0+t_0\in\TtauInn(\GG)=\bigcap\limits_{n=1}^{\infty}\Ttau(\HH_n)$ (by statement \eqref{thm5-4}) which would imply
\[
t_0=-\gamma_0+(\gamma_0+t_0)\in\Gamma+\bigcap_{n=1}^{\infty}\Ttau(\HH_n).
\]
Hence, if $t_0\not\in\Gamma+\bigcap\limits_{n=1}^{\infty}\Ttau(\HH_n)$ then $\tau^\GG_{t_0+\gamma}$ is outer and \eqref{freet0g0} for all $y\in\Linf(\GG)$ implies $\EE(u_{\gamma_0}^*v)=0$. Consequently
\[
\sum_{\substack{\beta\in\operatorname{Irr}{\GG}\\\uprho_\beta=\I}}v_{\beta,\gamma_0}\chi_\beta\Omega_{\bh_\GG}=\EE(u_{\gamma_0}^*v)\Omega_{\bh_\GG}=0
\]
which contradicts \eqref{nonzero}. Thus we must have $\TtauInn(\KK)\subset\Gamma+\bigcap\limits_{n=1}^{\infty}\Ttau(\HH_n)$.
\end{proof}

Theorem \ref{thm5} makes it possible to generalize the examples discussed in Corollary \ref{corManyIIIlambda}. The additional bicrossed product by a countable group acting by the scaling automorphisms will allow to produce uncountably many examples with the corresponding von Neumann algebra the injective factor of type $\mathrm{III}_\lambda$ with separable predual. Example \ref{exmoreIIIlambda} deals with the case $\lambda\in\left]0,1\right[$, while in Example \ref{exmoreIII1} we address the case $\lambda=1$.

\begin{example}\label{exmoreIIIlambda}
Fix $\lambda\in\left]0,1\right[$. For $n\in\NN$ let $\HH_n=\HH_{\nu,q}$ with $q=\sqrt{\lambda}$ and $\nu\not\in\tfrac{\pi}{\log{q}}\QQ=\tfrac{2\pi}{\log{\lambda}}\QQ$ (e.g.~$\nu=\tfrac{2\pi^2}{\log{\lambda}}$). Then each $\Linf(\HH_n)$ is a factor and $\Ttau(\HH_n)=\tfrac{\pi}{\log{q}}\ZZ=\tfrac{2\pi}{\log{\lambda}}\ZZ$. Furthermore $\operatorname{Sp}(\uprho_\alpha)=\operatorname{Sp}(\uprho_\alpha^{-1})$ for all $\alpha\in\operatorname{Irr}{\HH_n}$ and $\cC_{\HH_n}^{\;\prime}\cap\Linf(\HH_n)\subset\cC_{\HH_n}$ for any $n$ by \cite[Proposition 4.5]{KrajczokWasilewski}. Let $\Gamma\subset\RR$ be any countable subgroup such that $\Gamma\cap\tfrac{2\pi}{\log{\lambda}}\ZZ=\{0\}$ and set $\GG=\bigtimes\limits_{n=1}^{\infty}\HH_n$, $\KK=\Gamma\bowtie\GG$ with the action of $\Gamma$ on $\Linf(\GG)$ by the scaling automorphisms. By Theorem \ref{thm5}\eqref{thm5-2}--\eqref{thm5-3} the von Neumann algebras $\Linf(\KK)$ and $\Linf(\KK)^\sigma$ are factors with $S(\Linf(\KK))$ equal to
\[
\overline{\bigl\{
\uprho_{\beta_1,i_1}\uprho_{\beta_1,j_1}\dotsm\uprho_{\beta_n,i_n}\uprho_{\beta_n,j_n}
\,\bigr|\bigl.\,
n\in\NN,\:\beta_k\in\operatorname{Irr}{\HH_k},\:i_k,j_k\in\{1,\dotsc,\dim{\beta_k}\},\:k\in\{1,\dotsc,n\}
\bigr\}}.
\]
Recall from Section \ref{skladniki} that we have $\operatorname{Irr}{\HH_{\nu,q}}=\QQ\times\tfrac{1}{2}\ZZ_+$ with
\[
\uprho_{(r,s)}=\uprho_s,\qquad(r,s)\in\QQ\times\tfrac{1}{2}\ZZ_+,
\]
where $\uprho_s$ is the $\uprho$-operator corresponding to the spin-$s$ representation of $\operatorname{SU}_q(2)$, so that
\[
\uprho_{s}=\operatorname{diag}\bigl(q^{-2s},q^{-2s+2},\dotsc,q^{2s}\bigr)
\]
in the appropriate basis. Consequently $S(\Linf(\KK))=\{0\}\cup{q^{2\ZZ}}$ and thus $\Linf(\KK)$ is the injective factor of type $\mathrm{III}_\lambda$. Furthermore, by Theorem \ref{thm5}\eqref{thm5-4}
\[
\Ttau(\KK)=\bigcap_{n=1}^{\infty}\Ttau(\HH_n)=\tfrac{2\pi}{\log{\lambda}}\ZZ
\]
and
\[
\TtauInn(\KK)=\Gamma+\bigcap_{n=1}^{\infty}\Ttau(\HH_n)=\Gamma+\tfrac{2\pi}{\log{\lambda}}\ZZ.
\]
Since $\TtauAInn(\GG)=\RR$ (Theorem \ref{TauGthm}), we also have $\TtauAInn(\KK)=\RR$ because the scaling automorphisms of $\KK$ are essentially the same as for $\GG$ and the correspondence preserves the property of being inner (see \cite[Lemma 4.4]{KrajczokWasilewski}).

Consider $\RR$ as a vector space over $\QQ$ with basis $\{1\}\cup\{\alpha_j\}_{j\in\mathbb{J}}$ obtained by completing the system $\{1\}$ (in particular, we have $|\mathbb{J}|=\mathfrak{c}$). Next put
\[
\Gamma_j=\alpha_j\tfrac{2\pi}{\log{\lambda}}\ZZ,\qquad{j}\in\mathbb{J}.
\]
Then for each $j$ we have $\Gamma_j\cap\tfrac{2\pi}{\log{\lambda}}\ZZ=\{0\}$ and, moreover, if $j\neq{j'}$ then $\Gamma_j+\tfrac{2\pi}{\log{\lambda}}\ZZ\neq\Gamma_{j'}+\tfrac{2\pi}{\log{\lambda}}\ZZ$. If this were not the case there would exist $n,k\in\ZZ$ such that
\[
\alpha_j\tfrac{2\pi}{\log{\lambda}}=\alpha_{j'}\tfrac{2\pi}{\log{\lambda}}n+\tfrac{2\pi}{\log{\lambda}}k
\]
which would contradict linear independence of the system $\{1\}\cup\{\alpha_j\}_{j\in\mathbb{J}}$. Hence defining $\KK^\lambda_j=\Gamma_j\bowtie\GG$ we obtain an uncountable family $\bigl\{\KK^\lambda_j\bigr\}_{j\in\mathbb{J}}$ of compact quantum groups with $\Linf(\KK^\lambda_j)$ the injective factor of type $\mathrm{III}_\lambda$ with separable predual which are pairwise non-isomorphic, as they are told apart by the invariant $T^{\tau}_{\operatorname{Inn}}$.
\end{example}

\begin{example}\label{exmoreIII1}
Choose $q_1,q_2\in\left]0,1\right[$ such that $\tfrac{\pi}{\log{q_1}}\QQ\cap\tfrac{\pi}{\log{q_1}}\QQ=\{0\}$ (e.g.~$q_1=\ee^{-1}$ and $q_2=\ee^{-\pi}$). Furthermore for $i\in\{1,2\}$ let $\nu_i=\tfrac{\pi^2}{\log{q_i}}$, so that $\nu_i\not\in\tfrac{\pi}{\log{q_i}}\QQ$. Define
\[
\HH_n=
\begin{cases}
\HH_{\nu_1,q_1}&n\text{ is odd}\\
\HH_{\nu_2,q_2}&n\text{ is even}
\end{cases},\qquad\quad{n}\in\NN
\]
and fix an arbitrary countable subgroup $\Gamma\subset\RR$.

Since $\bigcap\limits_{n=1}^{\infty}\Ttau(\HH_n)=\tfrac{\pi}{\log{q_1}}\ZZ\cap\tfrac{\pi}{\log{q_1}}\ZZ=\{0\}$, we are in the situation described in Theorem \ref{thm5} and thus with
\[
\GG=\bigtimes_{n=1}^\infty\HH_n,\qquad\KK=\Gamma\bowtie\GG
\]
(with $\Gamma$ acting by the scaling automorphisms) we find that $\Linf(\KK)$ is a factor with
\begin{equation}\label{thecase}
S\bigl(\Linf(\KK)\bigr)=\overline{\bigl\{q_1^kq_2^l\,\bigr|\bigl.\,k,l\in\ZZ\bigr\}}=\RR_{\geq{0}}.
\end{equation}
Indeed, since $S(\Linf(\KK))\cap\RR_{>0}$ is a closed subgroup of $\RR_{>0}$ (\cite[Theorem 3.9 {\emph{a}})]{Connes}) if \eqref{thecase} were not the case we would have $S(\Linf(\KK))\cap\RR_{>0}=t^\ZZ$ for some $t\in\left]0,1\right[$. But then $q_1=t^k$ and $q_2=t^l$ for some $k,l\in\ZZ$ which implies
\[
\tfrac{\pi}{l\log{q_1}}=\tfrac{\pi}{k\log{q_2}}\in\tfrac{\pi}{\log{q_1}}\QQ\cap\tfrac{\pi}{\log{q_2}}\QQ
\]
contradicting our assumption. Therefore we can conclude that $\Linf(\KK)$ is an injective factor of type $\mathrm{III}_1$ with separable predual. By Theorem \ref{thm5} the subalgebra $\Linf(\KK)^\sigma$ is also a factor and we have
\[
\Ttau(\KK)=\bigcap_{n=1}^{\infty}\Ttau(\HH_n)=\{0\}\quad\text{and}\quad\TtauInn(\KK)=\Gamma+\bigcap_{n=1}^{\infty}\Ttau(\HH_n)=\Gamma
\]
(because $\cC_{\HH_n}^{\;\prime}\cap\Linf(\HH_n)\subset\cC_{\HH_n}$ for all $n$). Similarly as in Example \ref{exmoreIIIlambda}, we also have $\TtauAInn(\KK)=\RR$. Since $\Gamma$ may be chosen in uncountably many ways, we obtain an uncountable family of pairwise non-isomorphic compact quantum groups with the corresponding von Neumann algebra the injective factor of type $\mathrm{III}_1$ with separable predual.
\end{example}

\begin{remark}\label{remninjlambda}
Fix $\lambda\in\left]0,1\right]$ and let $\{\KK_i\}_{i\in{I}}$ be an uncountable family of compact quantum groups such that each $\Linf(\KK_i)$ is the injective type $\mathrm{III}_\lambda$ factor with separable predual and $\TtauInn(\KK_i)\neq\TtauInn(\KK_j)$ for $i\neq{j}$ (such families were constructed in the two previous examples) and let $\hh{\mathbb{F}_2}$ denote the compact quantum group dual to the free group on two generators. Then $\bigl\{\KK_i\times\hh{\mathbb{F}_2}\bigr\}_{i\in{I}}$ is an uncountable family of compact quantum groups such that each $\Linf(\KK_i\times\hh{\mathbb{F}_2})$ is a non-injective type $\mathrm{III}_\lambda$ factor with separable predual (\cite[Proposition 28.4]{Stratila}, see also the reasoning in Remark \ref{remninj0}). Furthermore, $\TtauInn(\KK_i\times\hh{\mathbb{F}_2})=\TtauInn(\KK_i)$ (Proposition \ref{propTprod}), so that the quantum groups $\bigl\{\KK_i\times\hh{\mathbb{F}_2}\bigr\}_{i\in{I}}$ are pairwise non-isomorphic.
\end{remark}

\subsection{\texorpdfstring{$\Ttau(\GG)$}{TtG} and \texorpdfstring{$T(\Linf(\GG))$}{T(L(G))}}\label{TtauandT}\hspace*{\fill}

In every example of a compact quantum group $\GG$ with $\Linf(\GG)$ the injective factor of type $\mathrm{III}_\lambda$ ($\lambda\in\left]0,1\right]$) discussed so far we have had $\Ttau(\GG)=T(\Linf(\GG))$. In this section we will prove that this relation holds for all compact quantum groups under a symmetry assumption on the spectra of the $\uprho$-operators and such that $\Linf(\GG)^\sigma$ is a factor, but not in general.

\begin{proposition}\label{TGTLG}
Let $\GG$ be a compact quantum group such that $\operatorname{Sp}(\uprho_\alpha)=\operatorname{Sp}(\uprho_\alpha^{-1})$ for all $\alpha\in\operatorname{Irr}{\GG}$. Then $\Ttau(\GG)\subset{T(\Linf(\GG))}$.
\end{proposition}

\begin{proof}
Take $t\in\Ttau(\GG)$, so that $\tau^\GG_t=\id$. It follows that
\begin{equation}\label{rhorho}
\uprho_{\alpha,i}^{\ii{t}}=\uprho_{\alpha,k}^{\ii{t}},\qquad\alpha\in\operatorname{Irr}{\GG},\:i,k\in\{1,\dotsc,\dim{\alpha}\}.
\end{equation}
Now fix $\alpha\in\operatorname{Irr}{\GG}$ and $i,j\in\{1,\dotsc,\dim{\alpha}\}$. By assumption there exists $k$ such that $\uprho_{\alpha,k}=\uprho_{\alpha,j}^{-1}$, so in view of \eqref{rhorho} we have
$\uprho_{\alpha,i}^{\ii{t}}
\uprho_{\alpha,j}^{\ii{t}}
=
\uprho_{\alpha,i}^{\ii{t}}
\uprho_{\alpha,k}^{-\ii{t}}=1$. Since this holds for all $\alpha$, $i$ and $j$, we conclude that $\sigma^\bh_t=\id$. In particular, the modular automorphism $\sigma^\bh_t$ is inner.
\end{proof}

As we already pointed out in Footnote \ref{footsigma} if $\Linf(\GG)^\sigma$ is a factor then so is $\Linf(\GG)$. Thus in the next theorem $\Linf(\GG)$ is implicitly assumed to be a factor.

\begin{theorem}\label{thmTGTL}
Let $\GG$ be a compact quantum group such that $\Linf(\GG)^\sigma$ is a factor and for all $\alpha\in\operatorname{Irr}{\GG}$ we have $\operatorname{Sp}(\uprho_\alpha)=\operatorname{Sp}(\uprho_\alpha^{-1})$. Then a real number $t$ belongs to $\Ttau(\GG)$ if and only if $\mu^{\ii{t}}=1$ for all $\mu\in{S(\Linf(\GG))\cap\RR_{>0}}$. In particular,
\begin{enumerate}
\item\label{TTIIIlambda} if $\Linf(\GG)$ is a factor of type $\mathrm{III}_\lambda$ for some $\lambda\in\left]0,1\right[$ then $\Ttau(\GG)=T(\Linf(\GG))=\tfrac{2\pi}{\log{\lambda}}\ZZ$,
\item\label{TTIII1} if $\Linf(\GG)$ is a factor of type $\mathrm{III}_1$ then $\Ttau(\GG)=T(\Linf(\GG))=\{0\}$,
\item\label{TTsemi} if $\Linf(\GG)$ is semifinite then $\Ttau(\GG)=T(\Linf(\GG))=\RR$ and, in particular, $\GG$ is of Kac type.
\end{enumerate}
\end{theorem}

\begin{proof}
Since $\bh$ is a faithful normal state on $\Linf(\GG)$ and $\Linf(\GG)^\sigma$ is a factor, we have
\[
S\bigl(\Linf(\GG)\bigr)=\operatorname{Sp}(\modOp[\bh])=\overline{\bigl\{\uprho_{\alpha,i}\uprho_{\alpha,j}\,\bigr|\bigl.\,\alpha\in\operatorname{Irr}{\GG},\:i,j\in\{1,\dotsc,\dim{\alpha}\}\bigr\}}.
\]
Take $t\in\Ttau(\GG)$ and $\mu=\uprho_{\alpha,i}\uprho_{\alpha,j}$ for some $\alpha\in\operatorname{Irr}{\GG}$ and $i,j\in\{1,\dotsc,\dim{\alpha}\}$, so that $\mu\in{S(\Linf(\GG))\setminus\{0\}}$. By assumption there exists $k\in\{1,\dotsc,\dim{\alpha}\}$ such that $\uprho_{\alpha,j}=\uprho_{\alpha,k}^{-1}$ and thus by \eqref{sigmatau} we obtain
\[
U^\alpha_{i,k}=\tau^\GG_t(U^\alpha_{i,k})=\uprho_{\alpha,i}^{\ii{t}}\uprho_{\alpha,k}^{-\ii{t}}U^\alpha_{i,k}=(\uprho_{\alpha,i}\uprho_{\alpha,j})^{\ii{t}}U^\alpha_{i,k},
\]
i.e.~$\mu^{\ii{t}}=1$. Hence by continuity $\mu^{\ii{t}}=1$ for all $\mu\in{S(\Linf(\GG))}\cap\RR_{>0}$.

On the other hand, if $t\in\RR$ and $\mu^{\ii{t}}=1$ for all $\mu\in{S(\Linf(\GG))\setminus\{0\}}$ then $(\uprho_{\alpha,i}\uprho_{\alpha,j})^{\ii{t}}=1$ for all $\alpha$, $i$ and $j$ which, under the symmetry assumption $\operatorname{Sp}(\uprho_\alpha)=\operatorname{Sp}(\uprho_\alpha^{-1})$ forces $t\in\Ttau(\GG)$.

If $\Linf(\GG)$ is of type $\mathrm{III}_\lambda$ with $\lambda\in\left]0,1\right[$ then $S(\Linf(\GG))=\{0\}\cup\lambda^\ZZ$ and $T(\Linf(\GG))=\tfrac{2\pi}{\log{\lambda}}\ZZ$. From Proposition \ref{TGTLG} we already know that $\Ttau(\GG)\subset{T(\Linf(\GG))}$. Moreover the number $t=\tfrac{2\pi}{\log{\lambda}}$ satisfies $\mu^{\ii{t}}=1$ for all $\mu\in{S(\Linf(\GG))}\setminus\{0\}$, so by the reasoning above we have $t\in\Ttau(\GG)$. Finally, since $\Ttau(\GG)$ is a group, we obtain $T(\Linf(\GG))=t\ZZ\subset\Ttau(\GG)$ which proves \eqref{TTIIIlambda}.

Statement \eqref{TTIII1} follows easily from Proposition \ref{TGTLG}, since in this case $\{0\}\subset\Ttau(\GG)\subset{T(\Linf(\GG))}=\{0\}$.

Finally if $\Linf(\GG)$ is semifinite then $T(\Linf(\GG))=\RR$ and $S(\Linf(\GG))=\{1\}$. It follows that any $t\in\RR$ satisfies $\mu^{\ii{t}}=1$ for all $\mu\in{S(\Linf(\GG))}$, so $\Ttau(\GG)=\RR$ as well. This proves statement \eqref{TTsemi}.
\end{proof}

The next example shows that the assumption $\operatorname{Sp}(\uprho_\alpha)=\operatorname{Sp}(\uprho_\alpha^{-1})$ for all $\alpha\in\operatorname{Irr}{\GG}$ cannot be dropped from the hypothesis of Proposition \ref{TGTLG}.

\begin{example}\label{exFkappa}
Consider the invertible matrix $F=\operatorname{diag}(\sqrt{\pi},1,1)$ and the associated free unitary quantum group $\operatorname{U}_F^+$ (see e.g.~\cite[Example 1.1.6]{NeshveyevTuset} or \cite[Section 4.2]{DeCommerFreslonYamashita}). Let $U$ denote the fundamental representation of $\operatorname{U}_F^+$. Then the corresponding $\uprho$-operator is $\uprho_U=\kappa{F^2}$, where $\kappa=\sqrt{\tfrac{2+\pi^{-1}}{2+\pi}}$. It is shown in \cite[Appendix]{DeCommerFreslonYamashita} that $\Linf(\operatorname{U}_F^+)$ is a full\footnote{A von Neumann algebra $\rM$ is \emph{full} if $\operatorname{Inn}(\rM)$ is closed in $\operatorname{Aut}(\rM)$ (\cite{ConnesIII1}).} non-injective factor.

Let us turn our attention to the invariant $\Sd(\Linf(\operatorname{U}_F^+))$ (\cite[Definition 1.2]{ConnesIII1}) which is the intersection of point spectra of all diagonalizable modular operators. By the results of \cite[Appendix]{DeCommerFreslonYamashita} the invariant $\Sd(\Linf(\operatorname{U}_F^+))$ is the subgroup of $\RR_{>0}$ generated by $\kappa^2$ and $\pi$ (this is the subgroup generated by the eigenvalues of $\kappa{F^2}\tens\kappa{F^2}$, see \cite[Theorem 33]{DeCommerFreslonYamashita}). It follows that $\Sd(\Linf(\operatorname{U}_F^+))$ is dense in $\RR$, since otherwise we would have $\Sd(\Linf(\operatorname{U}_F^+))=\mu^\ZZ$ for some $\mu>0$ which is impossible because $\pi$ is not an algebraic number.

Since $\Linf(\operatorname{U}_F^+)$ is full, we can use \cite[Corollary 4.11]{ConnesIII1} to obtain $S(\Linf(\operatorname{U}_F^+))=\overline{\Sd(\Linf(\operatorname{U}_F^+))}$ and hence $S(\Linf(\operatorname{U}_F^+))=\RR_{\geq{0}}$ which shows that $\Linf(\operatorname{U}_F^+)$ is of type $\mathrm{III}_1$ and consequently $T(\Linf(\operatorname{U}_F^+))=\{0\}$ by \cite[Theorem 3.7]{Connes} (see also \cite[Proposition 28.9]{Stratila}). On the other hand $\Ttau(\operatorname{U}_F^+)=\tfrac{2\pi}{\log{\pi}}\ZZ$, as it can be checked by direct computation that $\tau^{\operatorname{U}_F^+}_t=\id$ if and only if $t=\tfrac{2\pi{n}}{\log{\pi}}$ for some $n\in\ZZ$:
\[
\bigl(\id\tens\tau^{\operatorname{U}_F^+}_t\bigr)(U)=F^{2\ii{t}}UF^{-2\ii{t}}
=\begin{bmatrix}
U_{1,1}&\pi^{\ii{t}}U_{1,2}&\pi^{\ii{t}}U_{1,3}\\
\pi^{-\ii{t}}U_{2,1}&U_{2,2}&U_{2,3}\\
\pi^{-\ii{t}}U_{3,1}&U_{3,2}&U_{3,3}
\end{bmatrix},\qquad{t}\in\RR
\]
and matrix elements of $U$ generate $\Linf(\operatorname{U}_F^+)$. In particular, $\Ttau(\operatorname{U}_F^+)\not\subset{T(\Linf(\operatorname{U}_F^+))}$.
\end{example}

\begin{corollary}
For any $\nu\in\RR\setminus\{0\}$ and $q\in\left]-1,1\right[\setminus\{0\}$ such that $\nu\log|q|\notin\pi\QQ$ the subalgebra $\Linf(\HH_{\nu,q})^\sigma$ is not a factor.
\end{corollary}

\begin{proof}
$\Linf(\HH_{\nu,q})$ is a factor of type $\mathrm{II}_\infty$, i.e.~it is semifinite. If $\Linf(\HH_{\nu,q})^\sigma$ were a factor, the quantum group $\HH_{\nu,q}$ would be of Kac type by Theorem \ref{thmTGTL}\eqref{TTsemi}. However, $\HH_{\nu,q}$ is not of Kac type.
\end{proof}

\begin{corollary}
Let $\GG$ be a compact quantum group such that $\Linf(\GG)$ is a factor of type $\mathrm{III}_0$ and
$\operatorname{Sp}(\uprho_\alpha)=\operatorname{Sp}(\uprho_\alpha^{-1})$ for all $\alpha\in\operatorname{Irr}{\GG}$. Then $\Linf(\GG)^\sigma$ is not a factor.
\end{corollary}

\begin{proof}
We have $S(\Linf(\GG))=\{0,1\}$, so if $\Linf(\GG)^\sigma$ were a factor, from the first part of Theorem \ref{thmTGTL} we would infer that $\Ttau(\GG)=\RR$. Thus $\GG$ would be of Kac type, so in particular, it would admit a faithful tracial state which is impossible for an algebra of type $\mathrm{III}$.
\end{proof}

We end this section with three questions which arise naturally in the context of the results presented so far.

\begin{questions}
\noindent
\begin{enumerate}
\item Assume that $\GG$ is a second countable compact quantum group and $\TtauInn(\GG)=\RR$. Is it true that $\GG$ is of Kac type? That might not be the case if $\GG$ is not second countable (e.g.~consider $\RR\bowtie\operatorname{SU}_q(2)$, where $\RR$ is equipped with the discrete topology and acts on $\Linf(\operatorname{SU}_q(2))$ via the scaling automorphisms).
\item Can all injective factors of type $\mathrm{III}_0$ be obtained as $\Linf(\GG)$ for some compact quantum group $\GG$?
\end{enumerate}
\end{questions}

\section{On the type \texorpdfstring{$\mathrm{I}$}{I} case}\label{sectTypeI}

The question whether a given unital \cst-algebra can be the algebra of functions on a compact quantum group was initially addressed in \cite{without}. Next the topic was taken up in \cite{qdisk} and \cite{typeI}. In this section we solve a similar question on the von Neumann algebra level by showing that the algebra $\B(\sK)$ for $\sK$ a Hilbert space of dimension strictly greater than $1$ cannot be $\Linf(\GG)$. In contrast, locally compact quantum groups $\GG$ with $\Linf(\GG)\cong\B(\sK)$ are plentiful (cf.~e.g.~\cite{Fima}).

The techniques used to prove Theorem \ref{thmBH} below owe much to those employed in \cite{qdisk}. One of the new ingredients is the passage to a product quantum group which ensures the symmetry condition on the spectra of the $\uprho$-operators which already appeared in previous sections. As we mentioned in Remark \ref{remSym}, this condition is automatic for some quantum groups (cf.~\cite{symmetry}), but can easily fail as e.g.~in Example \ref{exFkappa}.

\begin{theorem}\label{thmBH}
Let $\sK$ be a Hilbert space of infinite dimension and $\rN$ a von Neumann algebra or the $0$ vector space. Then there does not exist a compact quantum group $\GG$ with $\Linf(\GG)\cong\rN\oplus\B(\sK)$.
\end{theorem}

It is worth comparing Theorem \ref{thmBH} with \cite[Theorem 3.4]{dCKSS} which says that a von Neumann algebra of the form $\rN\oplus\B(\sK)$ with $\sK$ finite-dimensional equipped with an ergodic action of a compact quantum group must be finite-dimensional. The two results are, however, proved by vastly different techniques.

\begin{proof}[Proof of Theorem \ref{thmBH}]
If $\sK$ is not separable then there are no faithful normal states on $\B(\sK)$ (and the Haar measure would induce one). Thus from now on we will assume that $\GG$ is a compact quantum group such that $\Linf(\GG)\cong\rN\oplus\B(\sK)$ with $\dim{\sK}=\aleph_0$ and argue toward a contradiction.

Let $\HH$ be the compact quantum group $\GG\times\GG$, so that
\[
\Linf(\HH)=\Linf(\GG)\vtens\Linf(\GG)\cong
(\rN\vtens\rN)\oplus\bigl(\rN\vtens\B(\sK)\bigr)
\oplus\bigl(\B(\sK)\vtens\rN\bigr)\oplus\bigl(\B(\sK)\vtens\B(\sK)\bigr)\cong\rM\oplus\B(\sK)
\]
where
\[
\rM=(\rN\vtens\rN)
\oplus\bigl(\rN\vtens\B(\sK)\bigr)
\oplus\bigl(\B(\sK)\vtens\rN\bigr)
\]
and let us denote the isomorphism $\Linf(\HH)\to\rM\oplus\B(\sK)$ by $\pi$. Clearly neither $\HH$ nor $\GG$ can be of Kac type (as there are no tracial states on $\B(\sK)$). Let $p\in\rM\oplus\B(\sK)$ be the central projection corresponding to $\B(\sK)\subset\rM$. The Haar measure $\bh$ on $\Linf(\HH)$ is necessarily of the form
\[
\bh(x)=\varphi\bigl((\I-p)\pi(x)\bigr)+\operatorname{Tr}\bigl(A\,p\,\pi(x)\bigr),\qquad{x}\in\Linf(\HH)
\]
for some normal faithful state $\varphi\in\rM_*$ (or zero if $\rM=0$) and a positive trace-class operator $A$ on $\sK$. Let $q_1>q_2>\dotsm$ be the list of all eigenvalues of $A$ in descending order (we note that $0$ is not an eigenvalue of $A$ because $\bh$ is faithful, in particular, the sequence of eigenvalues of $A$ is infinite) and for each $n$ let $\sK(A=q_n)$ be the corresponding eigenspace. The Hilbert space $\sK$ decomposes as
\[
\sK=\bigoplus_{n=1}^\infty\sK(A=q_n)
\]
and we note that each $\sK(A=q_n)$ is finite-dimensional, since $A$ is compact.

Observe that the modular group leaves $\B(\sK)\subset\rM\oplus\B(\sK)$ globally invariant. Indeed, it follows from the fact that $\B(\sK)$ is a factor and the function $\RR\ni{t}\mapsto\pi\bigl(\sigma^{\bh}_t(x)\bigr)\in\rM\oplus\B(\sK)$ is norm-continuous for all $x\in\C(\GG)$. The same argument shows that the scaling group leaves $\B(\sK)\subset\rM\oplus\B(\sK)$ globally invariant.

Let us note that the modular group of $\bh$ restricted to $\B(\sK)$ is implemented by $A$ in the sense that
\[
p\pi\bigl(\sigma_t^\bh(x)\bigr)=A^{\ii{t}}p\pi(x)A^{-\ii{t}},\qquad{x}\in\Linf(\HH),\:t\in\RR.
\]

Next let $B$ be a strictly positive self-adjoint operator on $\sK$ implementing the scaling group of $\HH$ restricted to $\B(\sK)$, so that
\[
p\pi\bigl(\tau^\HH_t(x)\bigr)=B^{\ii{t}}p\pi(x)B^{-\ii{t}},\qquad{x}\in\Linf(\HH),\:t\in\RR.
\]
which exists by Stone's theorem and \cite[Theorem 4.13]{Kadison} (cf.~\cite{qdisk,typeI}). In the same way as in \cite[proof of Theorem 3.5]{typeI} we find that $A$ and $B$ strongly commute, so in particular, for each $n$ the operator $B$ restricts to a positive operator on the finite-dimensional space $\sK(A=q_n)$. Let
\[
\mu_{n,1}>\dotsm>\mu_{n,P_n}
\]
be the eigenvalues of this restriction, so that
\[
\sK=\bigoplus_{n=1}^\infty\bigoplus_{p=1}^{P_n}\sK(A=q_n)\cap\sK(B=\mu_{n,p}),
\]
where $\sK(B=\mu_{n,p})$ denotes the eigenspace of $B$ corresponding to $\mu_{n,p}$.

Since $\GG$ is not of Kac type, there exists an irreducible representation $V$ of $\GG$ with $\uprho_V\neq\I$. Let $U\in\B(\sH_U)\tens\Linf(\GG)$ be the irreducible representation of $\HH=\GG\times\GG$ constructed as the exterior tensor product $V\boxtimes\overline{V}$ of $V$ and its conjugate (cf.~Section \ref{sectIPQG}).

Then $\uprho_U=\uprho_V\tens\uprho_{\overline{V}}$, so that
\[
\operatorname{Sp}(\uprho_U)=\operatorname{Sp}(\uprho_V)\cdot\operatorname{Sp}(\uprho_V)^{-1}.
\]
In particular, $\operatorname{Sp}(\uprho_U)$ is invariant under taking inverses. Choose an orthonormal basis of $\sH_U$ in such a way that the matrix of $\uprho_U$ in this basis is $\operatorname{diag}(\uprho_{U,1},\dotsc,\uprho_{U,\dim{U}})$ with $\uprho_{U,1}\geq\dotsm\geq\uprho_{U,\dim{U}}$.

As $U$ is a unitary representation writing $U_{i,j}$ for its matrix elements in the chosen basis, for all $i,j\in\{1,\dotsc,\dim{U}\}$ we have
\[
\sum_{k=1}^{\dim{U}}{U_{k,i}^{\;*}}U_{k,j}=\sum_{k=1}^{\dim{U}}(U^*)_{i,k}U_{k,j}=\delta_{i,j}\I.
\]
In particular,
\begin{equation}\label{UstarUjj}
\sum_{k=1}^{\dim{U}}U_{k,1}^{\;*}U_{k,1}=\I.
\end{equation}
Now using \eqref{sigmatau} we will show that $p\pi(U_{k,1})$ shifts the eigenspaces of $A$ and $B$ as follows:
\begin{align*}
p\pi(U_{k,1})\bigl(\sK(A=q_n)\bigr)&\subset\sK(A=\uprho_{U,k}\uprho_{U,1}q_n),\\
p\pi(U_{k,1})\bigl(\sK(B=\mu_{n,p})\bigr)&\subset\sK\bigl(B=\uprho_{U,k}\uprho_{U,1}^{-1}\mu_{n,p}\bigr).
\end{align*}
Indeed, if $\xi\in\sK(A=q_n)$ then for all $t\in\RR$ we have
\begin{align*}
A^{\ii{t}}p\pi(U_{k,1})\xi=A^{\ii{t}}p\pi(U_{k,1})A^{-\ii{t}}A^{\ii{t}}\xi
&=p\pi\bigl(\sigma_t^\bh(U_{k,1})\bigr)q_n^{\ii{t}}\xi\\
&=p\pi\bigl(\uprho_{U,k}^{\ii{t}}\uprho_{U,1}^{\ii{t}}U_{k,1}\bigr)q_n^{\ii{t}}\xi
=(\uprho_{U,k}\uprho_{U,1}q_n)^{\ii{t}}p\pi(U_{k,1})\xi
\end{align*}
which proves the first inclusion, and similarly for $\eta\in\sK(B=\mu_{n,p})$
\begin{align*}
B^{\ii{t}}p\pi(U_{k,1})\eta=B^{\ii{t}}p\pi(U_{k,1})B^{-\ii{t}}B^{\ii{t}}\eta
&=p\pi\bigl(\tau^\HH_t(U_{k,1})\bigr)\mu_{n,p}^{\ii{t}}\eta\\
&=p\pi\bigl(\uprho_{U,k}^{\ii{t}}\uprho_{U,1}^{-\ii{t}}U_{k,1}\bigr)\mu_{n,p}^{\ii{t}}\eta
=\bigl(\uprho_{U,k}\uprho_{U,1}^{-1}\mu_{n,p}\bigr)^{\ii{t}}p\pi(U_{k,1})\eta
\end{align*}
for all $t$, which proves the second inclusion.

Now let $\zeta$ be a non-zero vector in $\sK(A=q_1)\cap\sK(B=\mu_{1,P_1})$. Then for any $k\in\{1,\dotsc,\dim{U}\}$ we have
\[
p\pi(U_{k,1})\zeta\in\sK(A=\uprho_{U,k}\uprho_{U,1}q_1)\cap\sK\bigl(B=\uprho_{U,k}\uprho_{U,1}^{-1}\mu_{1,P_1}\bigr).
\]
Depending on $k$ we have two possibilities:
\begin{itemize}
\item[a)] $\uprho_{U,k}=\uprho_{U,1}$. Then $\uprho_{U,k}\uprho_{U,1}q_1=\uprho_{U,1}^2q_1>q_1=\|A\|$, so $\sK(A=\uprho_{U,k}\uprho_{U,1}q_1)=\{0\}$ and consequently $p\pi(U_{k,1})\zeta=0$,
\item[b)] $\uprho_{U,k}<\uprho_{U,1}$. Then first of all
\[
\uprho_{U,k}\uprho_{U,1}q_1\geq\bigl(\min\operatorname{Sp}(\uprho_U)\bigr)\uprho_{U,1}q_1=\uprho_{U,1}^{-1}\uprho_{U,1}q_1=q_1
\]
(because the spectrum of $\uprho_U$ is invariant under taking inverses), so
\[
\sK(A=\uprho_{U,k}\uprho_{U,1}q_1)=\sK(A=q_1)\quad\text{or}\quad\sK(A=\uprho_{U,k}\uprho_{U,1}q_1)=\{0\}.
\]
In the latter case $p\pi(U_{k,1})\zeta=0$, but also in the former situation we have
\[
p\pi(U_{k,1})\zeta\in\sK(A=q_1)\cap\sK\bigl(B=\uprho_{U,k}\uprho_{U,1}^{-1}\mu_{1,P_1}\bigr)
\]
and $\uprho_{U,k}\uprho_{U,1}^{-1}\mu_{1,P_1}=\uprho_{U,1}^{-2}\mu_{1,P_1}<\mu_{1,P_1}=\min\operatorname{Sp}\bigl(\bigl.B\bigr|_{\sK(A=q_1)}\bigr)$, so that
\[
\sK(A=q_1)\cap\sK\bigl(B=\uprho_{U,k}\uprho_{U,1}^{-1}\mu_{1,P_1}\bigr)=\{0\}
\]
\end{itemize}
and hence $p\pi(U_{k,1})\zeta=0$. However, by \eqref{UstarUjj}
\[
0\neq\zeta=\sum_{k=1}^{\dim{U}}p\pi(U_{k,1})^*p\pi(U_{k,1})\zeta=0
\]
which is a contradiction.
\end{proof}

As an immediate corollary we obtain

\begin{corollary}
Let $\sK$ be a Hilbert space of dimension strictly greater than $1$. Then there does not exist a compact quantum group $\GG$ with $\Linf(\GG)\cong\B(\sK)$.
\end{corollary}

This result follows from the previous theorem and an easy observation that $\B(\sK)$ cannot be isomorphic to $\Linf(\GG)$ when $1<\dim(\sK)<+\infty$ because $\B(\sK)$ is simple and $\Linf(\GG)=\operatorname{Pol}(\GG)$ admits a character -- the counit.

\section{Appendix: infinite tensor product of closed operators}\label{sectApp}

\begin{lemma}\label{lemA}
For $n\in\NN$ let $A_n$ be a closed densely defined operator on a Hilbert space $\sH_n$. Furthermore let $\Omega_n$ be a unit vector in $\sH_n$ such that $\Omega_n\in\Dom(A_n)\cap\Dom(A_n^*)$ and $A_n\Omega_n=A_n^*\Omega_n=\Omega_n$. Define a subspace $\Dom(\bA)$ of the Hilbert space $\sH=\bigotimes\limits_{n=1}^{\infty}(\sH_n,\Omega_n)$ by
\[
\Dom(\bA)=\operatorname{span}\bigl\{\xi_1\tens\dotsm\tens\xi_N\tens\Omega_{N+1}\tens\Omega_{N+2}\tens\dotsm\,\bigr|\bigl.\,N\in\NN,\:\xi_i\in\Dom(A_i),\:i\in\{1,\dotsc,N\}\bigr\}
\]
and let $\bA\colon\Dom(\bA)\to\sH$ be the linear extension of the map
\[
\xi_1\tens\dotsm\tens\xi_N\tens\Omega_{N+1}\tens\Omega_{N+2}\tens\dotsm\longmapsto
A_1\xi_1\tens\dotsm\tens{A_N}\xi_N\tens\Omega_{N+1}\tens\Omega_{N+2}\tens\dotsm.
\]
Then $\bA$ is a closable operator and denoting its closure by $\bigotimes\limits_{n=1}^{\infty}A_n$ we have
\[
\biggl(\bigotimes\limits_{n=1}^{\infty}A_n\biggr)^*=\bigotimes\limits_{n=1}^{\infty}A_n^*.
\]
In particular, if each $A_n$ is self-adjoint then so is $\bigotimes\limits_{n=1}^{\infty}A_n$.
\end{lemma}

\begin{proof}
Clearly $\Dom(\bA)$ is dense in $\sH$. Closability of $\bA$ follows from the fact that the dense subspace
\[
\operatorname{span}\bigl\{\eta_1\tens\dotsm\tens\eta_N\tens\Omega_{N+1}\tens\Omega_{N+2}\tens\dotsm\,\bigr|\bigl.\,N\in\NN,\:\eta_i\in\Dom(A_i^*),\:i\in\{1,\dotsc,N\}\bigr\}
\]
is contained in the domain of $\bA^*$. Indeed, for any $N$ and any $\eta_1\tens\dotsm\tens\eta_N\tens\Omega_{N+1}\tens\Omega_{N+2}\tens\dotsm$ with $\eta_i\in\Dom(A_i^*)$ the linear extension of the mapping taking
\[
\xi_1\tens\dotsm\tens\xi_N\tens\Omega_{N+1}\tens\Omega_{N+2}\tens\dotsm\in\Dom(\bA)
\]
to
\begin{align*}
&\is{\eta_1\tens\dotsm\tens\eta_N\tens\Omega_{N+1}\tens\Omega_{N+2}\tens\dotsm}{\bA(\xi_1\tens\dotsm\tens\xi_N\tens\Omega_{N+1}\tens\Omega_{N+2})}\\
&\qquad=\is{\eta_1\tens\dotsm\tens\eta_N\tens\Omega_{N+1}\tens\Omega_{N+2}\tens\dotsm}{A_i\xi_1\tens\dotsm\tens{A_N}\xi_N\tens\Omega_{N+1}\tens\Omega_{N+2}}\\
&\qquad=\is{A_1^*\eta_1\tens\dotsm\tens{A_N^*}\eta_N\tens\Omega_{N+1}\tens\Omega_{N+2}\tens\dotsm}{\xi_1\tens\dotsm\tens\xi_N\tens\Omega_{N+1}\tens\Omega_{N+2}}
\end{align*}
is clearly continuous, so $\eta_1\tens\dotsm\tens\eta_N\tens\Omega_{N+1}\tens\Omega_{N+2}\tens\dotsm\in\Dom(\bA^*)$.

By the same argument the infinite tensor product of the operators $\bigl\{A_n^*\bigr\}_{n\in\NN}$ is closable and denoting the respective closures respectively by
\[
\bigotimes\limits_{n=1}^{\infty}A_n
\quad\text{and}\quad
\bigotimes\limits_{n=1}^{\infty}A_n^*
\]
we obtain
\begin{equation}\label{Anst}
\bigotimes\limits_{n=1}^{\infty}A_n^*\subset\biggl(\bigotimes\limits_{n=1}^{\infty}A_n\biggr)^*.
\end{equation}

In order to obtain the reverse inclusion to \eqref{Anst} we let $p_N$ be the projection
\[
\underbrace{\I\tens\dotsm\tens\I}_{N\text{ times}}\tens\bigl(\ket{\Omega_{N+1}}\bra{\Omega_{N+1}}\bigr)\tens\bigr(\ket{\Omega_{N+2}}\bra{\Omega_{N+2}}\bigr)\tens\dotsm
\]
which is easily seen to be a well defined operator on $\sH$. Now take any $\zeta\in\Dom\Biggl(\biggl(\bigotimes\limits_{n=1}^{\infty}A_n\biggr)^*\Biggr)$ and consider the vector $p_N\zeta$. Clearly
\[
p_N\zeta=\zeta_N\tens\Omega_{N+1}\tens\dotsm
\]
for some $\zeta_N\in\sH_1\tens\dotsm\tens\sH_N$. Furthermore, the linear extension of the mapping
\begin{align*}
\Dom(A_1)\atens\dotsm\atens\Dom(A_N)\ni(\xi_1&\tens\dotsm\tens\xi_N)\longmapsto\is{\zeta_N}{(A_1\tens\dotsm\tens{A_N})(\xi_1\tens\dotsm\tens\xi_N)}\\
&=\is{\zeta_N}{A_1\xi_1\tens\dotsm\tens{A_N}\xi_N}\\
&=\is{\zeta_N\tens\Omega_{N+1}\tens\dotsm}{A_1\xi_1\tens\dotsm\tens{A_N}\xi_N\tens\Omega_{N+1}\tens\dotsm}\\
&=\is{p_N\zeta}{A_1\xi_1\tens\dotsm\tens{A_N}\xi_N\tens\Omega_{N+1}\tens\dotsm}\\
&=\is{\zeta}{p_N(A_1\xi_1\tens\dotsm\tens{A_N}\xi_N\tens\Omega_{N+1}\tens\dotsm)}\\
&=\is{\zeta}{(A_1\xi_1\tens\dotsm\tens{A_N}\xi_N\tens\Omega_{N+1}\tens\dotsm)}\\
&=\is{\zeta}{\biggl(\bigotimes\limits_{n=1}^{\infty}A_n\biggr)(\xi_1\tens\dotsm\tens\xi_N\tens\Omega_{N+1}\tens\dotsm)}\\
&=\is{\biggl(\bigotimes\limits_{n=1}^{\infty}A_n\biggr)^*\zeta}{\xi_1\tens\dotsm\tens\xi_N\tens\Omega_{N+1}\tens\dotsm}
\end{align*}
is continuous, so that $\zeta_N\in\Dom\bigl((A_1\tens\dotsm\tens{A_N})^*\bigr)=\Dom\bigl(A_1^*\tens\dotsm\tens{A_N^*}\bigr)$ (cf.~\cite[Proposition 7.26]{Schmuedgen}). Consequently for each $N$ we have $p_N\zeta\in\Dom\biggl(\bigotimes\limits_{n=1}^{\infty}A_n^*\biggr)$. Furthermore $p_N\zeta\xrightarrow[N\to\infty]{}\zeta$, hence in order to show that $\zeta\in\Dom\biggl(\bigotimes\limits_{n=1}^{\infty}A_n^*\biggr)$ we only need to show that
\[
\Biggl(\biggl(\bigotimes\limits_{n=1}^{\infty}A_n^*\biggr)p_N\zeta\Biggr)_{N\in\NN}
\]
is a Cauchy sequence.

Let $P_N$ be the composition of $p_N$ with the canonical isomorphism of the range of $p_N$ onto $\sH_1\tens\dotsm\tens\sH_N$
\[
\theta\tens\Omega_{N+1}\tens\dotsm\longmapsto\theta
\]
(the inverse of the canonical inclusion). Then $\zeta_N$ as above is $P_N\zeta$. Moreover the calculation above shows that for any $N$ and any $\xi_1,\dotsc,\xi_N$ with $\xi_i\in\Dom(A_i)$
\begin{align*}
&\is{\zeta_N}{(A_1\tens\dotsm\tens{A_N})(\xi_1\tens\dotsm\tens\xi_N)}\\
&\qquad=\is{\biggl(\bigotimes\limits_{n=1}^{\infty}A_n\biggr)^*\zeta}{\xi_1\tens\dotsm\tens\xi_N\tens\Omega_{N+1}\tens\dotsm}\\
&\qquad=\is{\biggl(\bigotimes\limits_{n=1}^{\infty}A_n\biggr)^*\zeta}{p_N\bigl(\xi_1\tens\dotsm\tens\xi_N\tens\Omega_{N+1}\tens\dotsm\bigr)}\\
&\qquad=\is{p_N\biggl(\bigotimes\limits_{n=1}^{\infty}A_n\biggr)^*\zeta}{\xi_1\tens\dotsm\tens\xi_N\tens\Omega_{N+1}\tens\dotsm}\\
&\qquad=\is{P_N\biggl(\bigotimes\limits_{n=1}^{\infty}A_n\biggr)^*\zeta}{\xi_1\tens\dotsm\tens\xi_N}
\end{align*}
so that
\[
\bigl(A_1^*\tens\dotsm\tens{A_N^*}\bigr)P_N\zeta=P_N\biggl(\bigotimes\limits_{n=1}^{\infty}A_n\biggr)^*\zeta.
\]
Thus for any $M,N$
\begin{align*}
&\Biggl\|\biggl(\bigotimes_{n=1}^{\infty}A_n^*\biggr)p_N\zeta-\biggl(\bigotimes_{n=1}^{\infty}A_n^*\biggr)p_M\zeta\Biggr\|\\
&\qquad=\Bigl\|\Bigl(\bigl(A_1^*\tens\dotsm\tens{A_N^*}\bigr)\zeta_N\Bigr)\tens\Omega_{N+1}\tens\dotsm
-\Bigl(\bigl(A_1^*\tens\dotsm\tens{A_M^*}\bigr)\zeta_M\Bigr)\tens\Omega_{M+1}\tens\dotsm\Bigr\|\\
&\qquad=\Biggl\|\Biggl(P_N\biggl(\bigotimes_{n=1}^{\infty}{A_n}\biggr)^*\zeta\Biggr)\tens\Omega_{N+1}\tens\dotsm
-\Biggl(P_M\biggl(\bigotimes_{n=1}^{\infty}{A_n}\biggr)^*\zeta\Biggr)\tens\Omega_{M+1}\tens\dotsm\Biggr\|\\
&\qquad=\Biggl\|
p_N\biggl(\bigotimes_{n=1}^{\infty}A_n\biggr)^*\zeta-p_M\biggl(\bigotimes_{n=1}^{\infty}A_n\biggr)^*\zeta
\Biggr\|\xrightarrow[N,M\to\infty]{}0.
\end{align*}
Consequently we get equality in \eqref{Anst} which ends the proof.
\end{proof}

\begin{remark}\label{uwagaDom}
The above reasoning proves that
\begin{enumerate}
\item if $\theta\in\Dom(A_1\tens\dotsm\tens{A_N})$ for some $N$ then the vector $\theta\tens\Omega_{N+1}\tens\dotsm$ belongs to $\Dom\biggl(\bigotimes\limits_{n=1}^{\infty}A_n\biggr)$ and we have
\[
\biggl(\bigotimes_{n=1}^{\infty}A_n\biggr)\bigl(\theta\tens\Omega_{N+1}\tens\dotsm\bigr)
=\bigl((A_1\tens\dotsm\tens{A_N})\theta\bigr)\tens\Omega_{N+1}\tens\dotsm,
\]
\item if $\zeta\in\Dom\biggl(\bigotimes\limits_{n=1}^{\infty}A_n\biggr)$ then with $p_N$ and $\zeta_N$ as above the vectors $\zeta_N$ and $p_N\zeta$ belong to $\Dom(A_1\tens\dotsm\tens{A_N})$ and $\Dom\biggl(\bigotimes\limits_{n=1}^{\infty}A_n\biggr)$ respectively and
\[
\biggl(\bigotimes\limits_{n=1}^{\infty}A_n\biggr)p_N\zeta=\bigl((A_1\tens\dotsm\tens{A_N})\zeta_N\bigr)\tens\Omega_{N+1}\tens\dotsm.
\]
\end{enumerate}
\end{remark}

In the special case when $A_n=\I$ for $n>N$ we denote the operator $\bigotimes\limits_{n=1}^{\infty}A_n$ constructed in Lemma \ref{lemA} by
\[
\bigotimes_{n=1}^N{A_n}\tens\I^{\tens\infty}.
\]
The next result says that if $A_n=A_n^*$ for all $n$ then the operators $\bigotimes\limits_{n=1}^NA_n\tens\I^{\tens\infty}$ approximate $\bigotimes\limits_{n=1}^{\infty}A_n$ in the \emph{strong resolvent sense} (\cite[Definition in Section VIII.7]{ReedSimon1}), i.e.~for any $\lambda\in\CC\setminus\RR$ the resolvent of $\bigotimes\limits_{n=1}^N{A_n}\tens\I^{\tens\infty}$ at $\lambda$ converges, as $N\to\infty$, to the resolvent of $\bigotimes\limits_{n=1}^{\infty}A_n$ at $\lambda$. In fact, by \cite[Theorem VIII.19]{ReedSimon1} this is equivalent to strong convergence of resolvents at one $\lambda_0\in\CC\setminus\RR$.

\begin{lemma}\label{lemB}
Let $(A_n,\sH_n,\Omega_n)_{n\in\NN}$ be as in Lemma \ref{lemA} and assume additionally that each $A_n$ is self-adjoint. Then
\begin{equation}\label{convergence}
\bigotimes\limits_{n=1}^N{A_n}\tens\I^{\tens\infty}\xrightarrow[N\to\infty]{}\bigotimes\limits_{n=1}^{\infty}A_n
\end{equation}
in the strong resolvent sense and
\begin{equation}\label{widmo}
\operatorname{Sp}\biggl(\bigotimes\limits_{n=1}^{\infty}A_n\biggr)
=\overline{\bigl\{\lambda_1\dotsm\lambda_N\,\bigr|\bigl.\,N\in\NN,\:\lambda_i\in\operatorname{Sp}(A_i),\:i\in\{1,\dotsc,N\}\bigr\}}.
\end{equation}
In particular, if each $A_n$ is positive then so is $\bigotimes\limits_{n=1}^{\infty}A_n$.
\end{lemma}

\begin{proof}
As we mentioned above, the strong resolvent convergence \eqref{convergence} follows once we show
\begin{equation}\label{zbiez}
\biggl(\ii\I-\bigotimes\limits_{n=1}^N{A_n}\tens\I^{\tens\infty}\biggr)^{-1}\xrightarrow[N\to\infty]{\text{\sc{sot}}}\biggl(\ii\I-\bigotimes\limits_{n=1}^{\infty}A_n\biggr)^{-1}
\end{equation}
and since all the operators in the sequence are contractions, it is enough to check convergence on a linearly dense subset of $\sH=\bigotimes\limits_{n=1}^{\infty}\sH_n$.

Take $K\in\NN$ and $\xi_i\in\sH_i$ for $i\in\{1,\dotsc,K\}$. Let us first observe that for any $N\geq{K}$ we have
\begin{equation}\label{odwrot}
\begin{aligned}
\biggl(\ii\I-\bigotimes_{n=1}^{N}A_n\tens\I^{\tens\infty}\biggr)^{-1}&(\xi_1\tens\dotsm\tens\xi_K\tens\Omega_{K+1}\tens\dotsm)\\
&=\biggl(\ii\I-\bigotimes_{n=1}^{\infty}A_n\biggr)^{-1}(\xi_1\tens\dotsm\tens\xi_K\tens\Omega_{K+1}\tens\dotsm)
\end{aligned}
\end{equation}
Indeed,
\begin{align*}
&\biggl(\ii\I-\bigotimes_{n=1}^{N}A_n\tens\I^{\tens\infty}\biggr)^{-1}(\xi_1\tens\dotsm\tens\xi_K\tens\Omega_{K+1}\tens\dotsm)\\
&\qquad=\bigl((\ii\I_{\sH_1\tens\dotsm\tens\sH_N}-A_1\tens\dotsm\tens{A_N})^{-1}\tens\I^{\tens\infty}\bigr)(\xi_1\tens\dotsm\tens\xi_K\tens\Omega_{K+1}\tens\dotsm)\\
&\qquad=\bigl((\ii\I_{\sH_1\tens\dotsm\tens\sH_N}-A_1\tens\dotsm\tens{A_N})^{-1}(\xi_1\tens\dotsm\tens\xi_K\tens\Omega_{K+1}\tens\dotsm\tens\Omega_N)\bigr)\tens\Omega_{N+1}\dotsm.
\end{align*}
Note that the vector
\[
(\ii\I_{\sH_1\tens\dotsm\tens\sH_N}-A_1\tens\dotsm\tens{A_N})^{-1}(\xi_1\tens\dotsm\tens\xi_K\tens\Omega_{K+1}\tens\dotsm\tens\Omega_N)
\]
belongs to $\Dom(A_1\tens\dotsm\tens{A_N})$, so by Remark \ref{uwagaDom} the vector
\[
\bigl((\ii\I_{\sH_1\tens\dotsm\tens\sH_N}-A_1\tens\dotsm\tens{A_N})^{-1}(\xi_1\tens\dotsm\tens\xi_K\tens\Omega_{K+1}\tens\dotsm\tens\Omega_N)\bigr)\tens\Omega_{N+1}\dotsm
\]
belongs to $\Dom\biggl(\bigotimes\limits_{n=1}^{\infty}A_n\biggr)$ and
\[
\resizebox{\textwidth}{!}{\ensuremath{
\begin{aligned}
\biggl(\ii\I-\bigotimes_{n=1}^{\infty}A_n\biggr)
\bigl((\ii\I_{\sH_1\tens\dotsm\tens\sH_N}-A_1\tens\dotsm\tens{A_N})^{-1}(\xi_1\tens\dotsm\tens\xi_K\tens&\Omega_{K+1}\tens\dotsm\tens\Omega_N)\bigr)\tens\Omega_{N+1}\dotsm
\\
&=\xi_1\tens\dotsm\tens\xi_K\tens\Omega_{K+1}\tens\dotsm
\end{aligned}
}}
\]
which proves equation \eqref{odwrot} and the strong resolvent convergence \eqref{zbiez} follows.

Now we pass to the proof of \eqref{widmo}. Take $\lambda$ belonging to the right-hand side. Then for any $\eps>0$ there is an integer $N_\eps$ and $\lambda_i^\eps\in\operatorname{Sp}(A_i)$ ($i=1,\dotsc,N_\eps$) such that $|\lambda-\lambda^\eps|\leq\eps$, where
\[
\lambda^\eps=\lambda_1^\eps\dotsm\lambda_{N_\eps}^\eps.
\]
As $\operatorname{Sp}(A_1\tens\dotsm\tens{A_{N_\eps}})=\overline{\operatorname{Sp}(A_1)\dotsm\operatorname{Sp}(A_{N_\eps})}$ (\cite[Theorem VII.33]{ReedSimon1}), the spectral theorem for self-adjoint operators (\cite[Theorem 10.4]{primer}) implies that there is a sequence of unit vectors $(\zeta_{\eps,k})_{k\in\NN}$ in $\Dom(A_1\tens\dotsm\tens{A_{N_\eps}})$ such that $\bigl\|(A_1\tens\dotsm\tens{A_{N_\eps}})\zeta_{\eps,k}-\lambda^\eps\zeta_{\eps,k}\bigr\|\xrightarrow[k\to\infty]{}0$. Consequently
\[
\biggl\|
\biggl(\bigotimes_{n=1}^{\infty}A_n\biggr)(\zeta_{\eps,k}\tens\Omega_{N_\eps+1}\tens\dotsm)-\lambda^\eps(\zeta_{\eps,k}\tens\Omega_{N_\eps+1}\tens\dotsm)\biggr\|\xrightarrow[k\to\infty]{}0,
\]
i.e. $\lambda^\eps$ is an approximate eigenvalue. Since the spectrum of $\bigotimes\limits_{n=1}^{\infty}A_n$ is closed, we find that $\lambda\in\operatorname{Sp}\biggl(\bigotimes\limits_{n=1}^{\infty}A_n\biggr)$.

Now let $\lambda$ be an arbitrary element of $\operatorname{Sp}\biggl(\bigotimes\limits_{n=1}^{\infty}A_n\biggr)$. For any $\eps>0$ there is a norm-one vector $\zeta^\eps\in\Dom\biggl(\bigotimes_{n=1}^{\infty}A_n\biggr)$ such that
\[
\biggl\|\biggl(\bigotimes_{n=1}^{\infty}A_n\biggr)\zeta^\eps-\lambda\zeta^\eps\biggr\|\leq\eps.
\]
Since $\bigotimes\limits_{n=1}^{\infty}A_n$ is the closure of $\bA$ (as in Lemma \ref{lemA}), for $k\in\NN$ there exists $\zeta^\eps_k\in\Dom(\bA)$ such that
\[
\|\zeta^\eps-\zeta^\eps_k\|\leq\tfrac{1}{k}\quad\text{and}\quad\biggl\|\biggl(\bigotimes_{n=1}^{\infty}A_n\biggr)\zeta^\eps-\biggl(\bigotimes_{n=1}^{\infty}A_n\biggr)\zeta^\eps_k\biggr\|\leq\tfrac{1}{k}.
\]
In particular, for each $k$ there is an integer $N_{\eps,k}$ such that $\zeta^\eps_k$ belongs to
\[
\operatorname{span}
\bigl\{\eta_1\tens\dotsm\tens\eta_{N_{\eps,k}}\tens\Omega_{N_{\eps,k}+1}\tens\dotsm\,\bigr|\bigl.\,\eta_i\in\Dom(A_i),\:i\in\{1,\dotsc,N_{\eps,k}\}\bigr\}
\]
and writing $\zeta^\eps_k$ as $\zeta^{\eps,0}_k\tens\Omega_{N_{\eps,k}+1}\tens\dotsm$ we obtain
\begin{align*}
\Bigl\|(A_1\tens\dotsm\tens{A_{N_{\eps,k}}})\zeta^{\eps,0}_k&-\lambda\zeta^{\eps,0}_k\Bigr\|
=\biggl\|\biggl(\bigotimes_{n=1}^{\infty}A_n\biggr)\zeta^\eps_k-\lambda\zeta^\eps_k\biggr\|\\
&\leq\biggl\|\biggl(\bigotimes_{n=1}^{\infty}A_n\biggr)\zeta^\eps_k-\biggl(\bigotimes_{n=1}^{\infty}A_n\biggr)\zeta^\eps\biggr\|+
\biggl\|\biggl(\bigotimes_{n=1}^{\infty}A_n\biggr)\zeta^\eps-\lambda\zeta^\eps\biggr\|+\|\lambda\zeta^\eps-\lambda\zeta^\eps_k\|\\
&\leq\tfrac{1}{k}+\eps+\tfrac{|\lambda|}{k},
\end{align*}
which implies that $\lambda\in\operatorname{Sp}(A_1\tens\dotsm\tens{A_{N_{\eps,k}}})$.
\end{proof}

\section*{Acknowledgments}

Research presented in this paper was partially supported by the Polish National Agency for the Academic Exchange, Polonium grant PPN/BIL/2018/1/00197 as well as by the FWO–PAS project VS02619N: von Neumann algebras arising from quantum symmetries and by the University of Warsaw Thematic Research Program "Quantum Symmetries". Additionally the first author was supported by EPSRC grants EP/T03064X/1 and EP/T030992/1 and the second author by NCN (National Science Centre, Poland) grant no.~2022/47/B/ST1/00582.

\end{document}